\documentclass[10pt,a4paper,twoside,leqno]%,draft]
{amsart}

\usepackage{verbatim}
\usepackage{amssymb}
\usepackage{upref}
\usepackage{amsmath}
\usepackage{amsthm}
\usepackage{amsfonts}
\usepackage{latexsym}
\usepackage{amssymb}
\usepackage{color}
\usepackage[latin1]{inputenc} 
\usepackage{enumerate}
\theoremstyle{plain}
\newtheorem{Thm}{Theorem}
\newtheorem*{Thm*}{Theorem}

\newtheorem{Prop}{Proposition}%[section]
\newtheorem{Lem}[Prop]{Lemma}

\theoremstyle{definition}
\newtheorem{Def}{Definition}

\theoremstyle{remark}
\newtheorem{Rem}{Remark}

\allowdisplaybreaks

\newcommand{\C}{{\mathbb C}}
\newcommand{\R}{{\mathbb R}}

\newcommand{\base}{\dot}
\newcommand{\fou}{\widehat}
\newcommand{\aux}{{\mathrm{aux}}}

\def\duno{\partial_1}

\def\dt{\partial_t}
\def\d1{\dfrac{d\;}{dx_1}}

\def\H{\mathcal H }

\def\ds{\displaystyle}
\def\eps{\varepsilon}
\def\tdelta{\tilde\delta}
\def\ttau{\tilde\tau}

\newcommand{\Hc}{{\H}} % Campo magnetico nel vuoto

\newcommand{\p}{\partial}
\newcommand{\Ec}{{\mathcal{E}}} % Campo elettrico nel vuoto
\newcommand{\Wc}{{\mathcal{W}}} % Coppia (Campo Elettrico, Campo Magnetico), nel vuoto
\newcommand{\nc}{{\mathrm{nc}}}
\newcommand{\car}{{\mathrm{c}}}

\newcommand{\Lc}{{\mathcal{L}}}
\newcommand{\Ac}{{\mathcal{A}}}
\newcommand{\Bc}{{\mathcal{B}}}
\newcommand{\Cc}{{\mathcal{C}}}
\newcommand{\Dc}{{\mathcal{D}}}
\newcommand{\Vc}{\mathcal V}

\DeclareMathOperator{\dv}{{div}}%{{\nabla \cdot\, }}
\DeclareMathOperator{\rot}{{\nabla \times\, }}
\newcommand{\tm}{^\intercal}
\renewcommand{\doteq}{{\mathrm{:=}}}

\newcommand{\Id}{{\mathrm{I}}}

\newcommand{\rhop}{{\base{\alpha}}}

\let\Im\relax
\DeclareMathOperator{\Im}{{Im}}
\let\Re\relax
\DeclareMathOperator{\Re}{{Re}}
\let\supp\relax
\DeclareMathOperator{\supp}{{supp}}

\newcommand{\doititle}{}

\title[Plasma-vacuum interface]{Weak stability of the plasma-vacuum interface problem}
\author[d. catania, m. d'abbicco and p. secchi]{Davide Catania$^{1,2}$, Marcello D'Abbicco$^1$, Paolo Secchi$^1$}


\email{davide.catania@unibs.it}
\email{marcello.dabbicco@unibs.it}
\email{paolo.secchi@unibs.it}

\address{$^1$ DICATAM, Mathematical Division, University of Brescia, Via Valotti 9, 25133 Brescia, Italy\newline \indent	$^2$ SMART Engineering Solutions \& Technologies (SMARTEST) Research Centre, 
	eCampus University,
	Via Isimbardi 10, 22060
	Novedrate (CO), Italy
	}

\begin{document}

\baselineskip=14pt

\begin{abstract}

We consider the free boundary problem for the two-dimensional plasma-vacuum interface in ideal compressible magnetohydrodynamics (MHD). In the plasma region, the flow is governed by the usual compressible MHD equations, while in the vacuum region we consider the Maxwell system for the electric and the magnetic fields. At the free interface, driven by the plasma velocity, the total pressure is continuous and the magnetic field on both sides is tangent to the boundary.

We study the linear stability of rectilinear plasma-vacuum interfaces by computing the Kreiss--Lopatinski\u{\i} determinant of an associated linearized boundary value problem. Apart from possible resonances, we obtain that the piecewise constant plasma-vacuum interfaces are always weakly linearly stable, independently of the size of tangential velocity, magnetic and electric fields on both sides of the characteristic discontinuity.

We also prove that solutions to the linearized problem obey an energy estimate with a loss of regularity with respect to the source terms, both in the interior domain and on the boundary, due to the failure of the uniform Kreiss--Lopatinski\u{\i} condition, as the Kreiss--Lopatinski\u{\i} determinant associated with this linearized boundary value problem has roots on the boundary of the frequency space. In the proof of the a priori estimates, a crucial part is played by the construction of symmetrizers for a reduced differential system, which has poles at which the Kreiss--Lopatinski\u{\i} condition may fail simultaneously.

\end{abstract}

\maketitle

\section{Introduction}

Plasma-vacuum interface problems appear in the mathematical modeling of plasma confinement by magnetic fields in thermonuclear energy production (as in Tokamaks, Stellarators; see, e.g., \cite{Goed}). 
%In this model, the plasma is confined inside a perfectly conducting rigid wall and isolated from it by a region containing very low density plasma, which may qualify as vacuum, due to the effect of strong magnetic fields.
%This subject is very popular since the 1950--70's, but most of theoretical studies are devoted to finding stability criteria of equilibrium states.
There are also important applications in astrophysics, where the plasma-vacuum interface problem can be used for modeling the motion of a star or the solar corona when magnetic fields are taken into account. %Once again, the interface can be described as a tangential discontinuity, i.e. the magnetic fields do not intersect the interface.

In~\cite{SeTr,SeTrNl}, the authors obtained the local-in-time existence and uniqueness of solutions to the free boundary problem for the plasma-vacuum interface in ideal compressible magnetohydrodynamics (MHD), by considering the {\it pre-Maxwell dynamics} for the magnetic field in the vacuum region, as usually assumed in the classical formulation. The linearized stability of the relativistic case has been addressed by Trakhinin in \cite{trakhinin12}, in the case of plasma expansion in vacuum. The paper \cite{CDAS} is devoted to the study of the linearized stability for the non-relativistic case, but, instead of the pre-Maxwell dynamics, in the vacuum region the displacement current was taken into account and the complete system of {\it Maxwell equations} for the electric and the magnetic fields was considered. The introduction of this model aims at investigating the influence of the electric field in vacuum on the well-posedness of the problem, as in the classical pre-Maxwell dynamics such an influence is hidden. See also \cite{mandrik-trakhinin} for a similar problem.

For the relativistic plasma-vacuum problem, Trakhinin \cite{trakhinin12} has shown the possible ill-posedness in the presence of a sufficiently strong vacuum electric field. Since relativistic effects play a rather passive role in the analysis of \cite{trakhinin12}, it is natural to expect the same for the non-relativistic problem.
On the contrary, in \cite{CDAS} it was shown that a {\it sufficiently weak} vacuum electric field precludes ill-posedness and gives the well-posedness of the linearized problem, thus somehow justifying the practice of neglecting the displacement current in the classical pre-Maxwell formulation when the vacuum electric field is weak enough.

In all the previously cited papers \cite{CDAS, mandrik-trakhinin, SeTr,SeTrNl, trakhinin12} the analysis is performed under a suitable stability condition stating that at each point of the free interface the magnetic fields on both sides are not parallel, see also  \cite{Cat,chenwang,cmst, SWZ,trakhinin09arma} for the similar condition on current-vortex sheets. These works show that non-parallel magnetic fields may stabilize the motion. 
{ The main technical reason of why the stabilization occurs is that the non-collinearity of the magnetic fields is a sufficient condition for the ellipticity of the symbol of the boundary operator, namely the operator that is obtained from the boundary conditions and applies to the function describing the free interface, and this gives a control of the space-time gradient of such a function.}

{
%(As a matter of facts, the relation between the ellipticity of the symbol of the boundary operator and the well-posedness / stability of hyperbolic free boundary problems is not completely clear. For example, let us consider the 2-D supersonic compressible vortex sheets \cite{cs}, and current-vortex sheets in incompressible MHD \cite{SWZ}. For these problems the necessary and sufficient conditions for linearized stability are well-known, and under the respective stability conditions the well-posedness of the nonlinear has been proven. For both problems the stability condition is sufficient for the ellipticity of the symbol of the front. Thus one could think that such ellipticity is a necessary condition for well-posedness.) 
On the other hand, one could guess that the well-posedness could be guaranteed as well for problems without ellipticity of the boundary operator,  and necessarily with less regularity of the free interface, provided a suitable stability condition is assumed.
}

{In this regard, in the recent paper \cite{trakhinin15} Y. Trakhinin considered the three dimensional plasma-vacuum interface problem in the classical pre-Maxwell dynamics formulation with non-elliptic interface symbol. In \cite{trakhinin15} a basic $L^2$ a priori estimate was derived for the linearized problem with variable coefficients, in the case that the unperturbed plasma and vacuum magnetic fields are everywhere parallel on the interface, provided a Rayleigh-Taylor sign condition on the jump of the normal derivative of the total pressure is satisfied at each point of the interface. The general case when the plasma and vacuum magnetic fields are collinear somewhere on the interface, but not everywhere, is still an open problem.
}

%In order to investigate the case of parallel magnetic fields at the free interface, 
In this paper we consider the two dimensional plasma-vacuum interface in ideal compressible magnetohydrodynamics (MHD). In the plasma region, the flow is governed by the usual compressible MHD equations, while in the vacuum region we consider the Maxwell system for the electric and the magnetic fields. At the free interface, driven by the plasma velocity, the total pressure is continuous and the magnetic field on both sides is tangent to the boundary. In particular, in two dimensions the magnetic fields obviously belong to the same plane, thus the condition that magnetic fields are tangential at the free boundary necessarily implies that they are parallel.

We study the linear stability of rectilinear plasma-vacuum interfaces by computing the Kreiss--Lopatinski\u{\i} determinant of an associated linearized boundary value problem. Apart from possible resonances, we obtain that the piecewise constant plasma-vacuum interfaces are {\it always} weakly linearly stable, independently of the size of tangential velocity, magnetic and electric fields on both sides of the characteristic discontinuity. In particular, violent instability never occurs.

In comparison with the possible ill-posedness in three dimensions in presence of a sufficiently strong vacuum electric field for the relativistic plasma-vacuum problem \cite{trakhinin12}, and the similar result for the non-relativistic problem, our study shows that, even with parallel magnetic fields, the two-dimensional plasma-vacuum interfaces are more stable than the three-dimensional ones. 
{This can be explained by noticing that in the two dimensional case the vacuum electric field has only one component $\Ec=\Ec_3$, orthogonal to the plane of motion, while the crucial role in the appearence of violent instability in the three dimensional case is played by the normal component $\Ec_1$ to the interface (which is zero by definition in the 2D case).
}

In the present paper we also prove that solutions to the linearized problem obey an energy estimate with a loss of regularity with respect to the source terms, both in the interior domain and on the boundary, due to the failure of the uniform Kreiss--Lopatinski\u{\i} condition.

It would be interesting to extend the normal modes analysis with the computation of the Kreiss--Lopatinski\u\i\ determinant to the three-dimensional problem, in order to get a complete description of the region of weak stability/instability. However, the great algebraic complexity makes it a very difficult task. 

There are essential difficulties in the study of our problem. First, the problem of compressible plasma-vacuum interfaces is a nonlinear hyperbolic problem with a free boundary, since the interface is part of the unknowns; moreover, this free boundary is characteristic and we only expect a partial control of the trace of the solution, namely of the so-called noncharacteristic part of the solution. 

When we take the Laplace transform of the solution with respect to time and the Fourier transform with respect to the tangential space variable, we obtain an ODE system only for the transform of the noncharacteristic part of the solution. As a consequence of the characteristic boundary, the symbol of the ODE has poles. The Kreiss--Lopatinski\u\i\ determinant associated to the ODE has no root in the interior of the frequency space (if this happened, the problem would be violently unstable, i.e. ill-posed), but it has some roots on the boundary and consequently the Kreiss--Lopatinski\u{\i} condition may only hold in weak form. An additional difficulty comes from the fact that some poles of the symbol coincide with the roots of the determinant. Under suitable restrictions on the basic state, such resonances are not allowed and the roots may be only simple. The case of multiple roots may typically occur at the transition to instability, see \cite{transition}, and gives a loss of regularity of higher order in the a priori estimate.

Following the approach of \cite{nappes}, we prove the energy estimate of solutions to the linearized problem by the construction of a degenerate Kreiss symmetrizer associated with the reduced boundary value problem. However, here the situation is more complicated than for the two-dimensional compressible vortex sheets \cite{nappes}, because there are more poles and roots of the Kreiss--Lopatinski\u{\i} determinant, and they can coincide. %To fix this last case we adapt an argument of Majda and Osher \cite{majda-osher} as in \cite{WYu}.

The paper is organized as follows. In Section \ref{problem} we give the formulation of the free boundary problem, the reduction to the fixed domain with flat boundary and its linearization. The main result of the paper is stated in Section \ref{Main results}. In Section \ref{reductions}
we apply to the problem the Laplace transform in the time variable and the Fourier transform in the tangential space variable. Then we eliminate the unknown front with the help of the ellipticity of the boundary conditions for the front. In Section \ref{normal} we reduce the linearized problem into a boundary value problem of the homogeneous ordinary differential equations, with respect to the normal space variable, for the noncharacteristic part of the solution.  In Section \ref{sec.lopatinskii} we compute the roots of the associated Kreiss--Lopatinski\u{\i} determinant and in Section \ref{symmetrizer} we construct the Kreiss symmetrizer for this boundary value problem, proceeding as in \cite{chazarain-piriou,coulombel1,nappes,kreiss,ralston}.

Special attention is required for the poles of the symbol of the reduced boundary value problem, which split in two categories: points that are poles of the symbol where the Lopatinski\u{\i} condition is satisfied, and points that are poles of the symbol and simultaneously roots of the Kreiss--Lopatinski\u{\i} determinant. To deal with the difficulty arising from these poles in the construction of symmetrizers, we use an approach inspired by Majda and Osher \cite{majda-osher}, as in \cite{WYu}. Finally, in Section \ref{Stima dell'energia}, we derive the energy estimates of solutions to the linearized problem by using the constructed symmetrizer.

\section{Formulation of the problem}\label{problem}

We consider the case when the whole space $\mathbb R^2$ is split into two regions by a smooth hypersurface $\Gamma(t):=\{(t,x_1,x_2)\in (0,\infty)\times\R^2 : F(t,x)=0\}=\{x_1=\varphi(t,x_2)\}$, and define $\Omega^\pm(t):=\{(t,x_1,x_2)\in (0,\infty)\times\R^2 :x_1\gtrless\varphi(t,x_2)\}$. We assume the presence of ideal compressible plasma in $\Omega^+(t)$ and vacuum in $\Omega^-(t)$.

We assume that plasma is governed by the ideal compressible MHD system, that in 3D reads
%
%Plasma in $\Omega^+(t)$ is governed by the ideal compressible MHD system
%
\begin{equation}
\begin{cases}\label{e1}
\rho_t + \dv (\rho v) = 0 \,, \\
\rho \bigl( \p_t v + (v\cdot \nabla) v \bigr) - (H\cdot \nabla) H + \nabla q = 0 \,, \\
\p_t H - \rot (v\times H) = 0 \,,\\% \qquad \dv H =0\,,\\
\p_t \bigl(\rho e + |H|^2/2\bigr) + \dv \bigl( (\rho e+p) v + H \times (v\times H) \bigr) =0 \,,
\end{cases} 
\end{equation} 
where $\rho$ is the density, $v$ and $H$ are the velocity and (plasma) magnetic fields, $q=p+|H|^2/2$ is the total pressure, with $p$ denoting the pressure. Moreover $e=E+|v|^2/2$ denotes the total energy, with $E$ the internal energy. The subscript $t$ denotes differentiation with respect to the time variable $t$. This system is supplemented by the divergence constraint $$\dv H=0$$ on the initial data. 
Given state equations of gases $\rho=\rho(p,S)$ and $E=E(p,S)$, where $S$ is the entropy, and the first principle of thermodynamics, \eqref{e1} is a closed system.

Since we are interested in the 2D planar case, we assume that no variable depends on  $x_3$, so that the terms with $\p_3$ are zero, and that  $v_3=H_3=0$. Choosing as unknown the vector $ U =U (t, x )=(q, v,H, S)\tm$, the two dimensional version of the plasma system \eqref{e1}  can be written in symmetric form as (see \cite{CDAS,SeTr})
\begin{align*}
A_0(U )\partial_tU+\sum_{j=1}^2 A_j(U )\partial_jU=0\, , \qquad \text{in } \Omega^+(t)\,,
\end{align*}
where the explicit expressions of the matrices $A_j$'s are given below in \eqref{A0}--\eqref{A2}. This system is symmetric hyperbolic provided that $A_0$ is positive definite, i.e. if $\rho>0$ and $\rho_p:=\p_p\rho>0$. In particular, if $\rho$ is uniformly bounded away from 0 in $\Omega^+(t)$, this yields that the density has a jump across the interface, because it vanishes in the vacuum region.

In the vacuum region, the electric field $\Ec$ and the magnetic field $\Hc$ are governed by the Maxwell equations that, in a three dimensional domain, can be written in nondimensional form as (see \cite{CDAS,mandrik-trakhinin})
\begin{align} \label{eq.vacuum0} \begin{cases}
\eps\p_t \Hc + \rot \Ec = 0\,, \qquad \eps\p_t \Ec - \rot \Hc = 0 \, ,  \\
\dv \Hc=0\,, \qquad \dv \Ec=0\,, \qquad \text{at} \; t=0\,,
\end{cases}  \end{align}
where $\Ec, \Hc$ are the (vacuum) electric and magnetic fields, $\eps:=\bar{v}/c_{\rm L}$, $c_{\rm L}$ is the speed of light, while $\bar v$ is the velocity of a reference uniform flow, for instance the sound speed.

Again, in order to obtain the two dimensional version of the problem we are considering, we assume that no variable depends on  $x_3$ and that  $\Ec_1=\Ec_2=\Hc_3=0$ in $\Omega^-(t)$.
%In particular, we have
%
%\begin{gather*}
%\rot (v\times H) = \rot (0, 0, v_1H_2-v_2H_1) = \bigl(\p_2(v_1H_2-v_2H_1), -\p_1(v_1H_2-v_2H_1) ,0\bigr)\,,
%\\
%H \times (v\times H) = H \times (0,0,v_1H_2-v_2H_1) = \bigl(H_2(v_1H_2-v_2H_1), -H_1(v_1H_2-v_2H_1)
%,0\bigr)\,,
%\end{gather*}
%%
%in~$\Omega^+(t)$.
As far as the vacuum is concerned, we apply a reflection with respect to $x_1$ (in particular, $\p_1$ becomes $-\p_1$) in order to have again equations that hold in $\Omega^+(t)$.
%(this will be useful for the subsequent analysis):
%%
%\[ \begin{cases}
%\eps\p_t \Hc_1 + \p_2 \Ec_3 = 0\,, \qquad \qquad \text{in} \quad \Omega^+(t)\,, \\
%\eps\p_t \Hc_2 + \p_1 \Ec_3 = 0\,,  \\
%\p_1 \Ec_2 + \p_2 \Ec_1 = 0\,,  \\
%-\p_1 \Hc_1 + \p_2 \Hc_2 = 0\,,\\
%\p_t \Ec_1 = 0\,, \\
%\p_t \Ec_2 = 0\,, \\
%\eps\p_t \Ec_3 + \p_1 \Hc_2 + \p_2 \Hc_1 = 0\,, \\
%\p_1 \Ec_1 - \p_2 \Ec_2 = 0\,.
%\end{cases}  \]
%%
%The previous equations imply that, if we assume  $\Ec_1=\Ec_2=0$ at the initial time~$t=0$, we deduce  $\Ec_1=\Ec_2\equiv0$ for any time. For simplicity, from now on we will denote by ~$\Ec$ the scalar function~$\Ec_3$. Moreover, assuming
%%
%\[ -\p_1 \Hc_1 + \p_2 \Hc_2 = 0 \]
%%
%at the initial time, it still holds true for any~$t>0$, as one can easily obtain by applying $\partial_1$ and  $\partial_2$ respectively to the first and second equation above, and then subtracting.
Consequently, the system \eqref{eq.vacuum0} reduces to
\[ \begin{cases}
\eps\p_t \Hc_1 + \p_2 \Ec_3 = 0\,,  \qquad \text{in } \Omega^+(t)\, ,\\
\eps\p_t \Hc_2 + \p_1 \Ec_3 = 0\,,  \\
\eps\p_t \Ec_3 + \p_1 \Hc_2 + \p_2 \Hc_1 = 0\,.
\end{cases}  \]
As usual, the divergence constraint on $\Hc$ is just a restriction on the initial data.

%%%%%%%%%%%%%%

We assume that the unknown interface $\Gamma(t)$ moves according to the plasma velocity, i.e.
\[
\frac{{\rm d}F }{{\rm d} t}=0\, , \qquad \text{on} \quad \Gamma(t)\,,
\]
and that at the interface the total pressure is continuous, while the magnetic fields can manifest only a tangential discontinuity. In a three dimensional domain, such conditions can be written as
\[
 [q]=0,\quad  H\cdot N=0, \quad \Hc\cdot N=0 ,\quad  N \times \Ec = \eps (N \cdot v)  \Hc \qquad \text{on} \quad \Gamma(t)
\]
where $N=\nabla F$ and $[q]= q|_{\Gamma}-\frac{1}{2}|\mathcal{H}|^2_{|\Gamma}+\frac12|\Ec|^2_{|\Gamma}$ denotes the jump of the total pressure across the interface. For a discussion on the boundary conditions see \cite{Goed}.

Recalling the parametrization of $\Gamma(t)=\{x_1=\varphi(t,x_2)\}$, we have
$N=(1,-\varphi_2,0)$, and in the two dimensional case, the boundary conditions become
\begin{gather*} \varphi_t = v_N = v_1 - \varphi_2 v_2 \,, \qquad [q]=0\,,\qquad 0 = H_N = H_1 -\varphi_2 H_2\,, \\
 0 = \Hc_N = \Hc_1-\varphi_2\Hc_2\,, \qquad \Ec_3+\eps\varphi_t\Hc_2 =0\,
\end{gather*}
(we are omitting $\varphi_2 \Ec_3 + \eps\varphi_t \Hc_1=0$, since it can be obtained by summing the last two conditions, respectively multiplied by $\eps\varphi_t$ and $\varphi_2$).

From now on, we neglect the third component of the vector functions  $v,H,\Hc,N$, that now are vectors in ~$\R^2$, while~$\Ec$ is a scalar function. Notice that, even if~$\Ec$ is a scalar, we write~$\p_j\Ec$ to denote derivatives in order to keep consistency with previous papers. Otherwise, subscripts denote components of vector quantities, but derivatives of scalar functions. Moreover, sometimes we will use the notation $\partial_0=\partial_t$ to denote the partial derivative with respect to $t$.

\medskip

\textbf{Reduction to a fixed domain.} In order to reduce the problem to a fixed domain with flat boundary, independent of time,
\[
\Omega^+ \doteq  \R^2 \cap \{x_1> 0 \} \, ,\qquad \Gamma \doteq \R^2\cap\{x_1=0\} \, ,
\]
we need the diffeomorphism provided by the following lemma.

\begin{Lem}
\label{lemma.diffeo}
Let $m \ge3$ be an integer. For any $T>0$, and
for any
\[ \varphi \in \cap_{j=0}^{m-1} {\mathcal C}^j([0,T];H^{m-j-\frac12}(\R))\,,\]
satisfying without loss of generality $\| \varphi\|_{{\mathcal C}([0,T];H^{2}(\R))} \le 1$, there exists a function \[\Psi \in \cap_{j=0}^{m-1} {\mathcal C}^j([0,T];H^{m-j}(\Omega^+))\]
such that the function
\begin{equation*}
\label{change}
\Phi(t,x) \doteq  \big( x_1 +\Psi(t,x),x_2 \big) \, , \qquad (t,x) \in [0,T]\times \Omega^+ \, ,
\end{equation*}
defines an $H^m$-diffeomorphism of $\Omega^+$ for all $t \in [0,T]$.
Moreover, there holds \[\partial^j_t (\Phi - Id) \in {\mathcal C}([0,T];H^{m-j}(\Omega^+))\] for $j=0,\dots, m-1$,
$\Phi(t,0,x_2)=(\varphi(t,x_2),x_2)$, $\duno \Phi(t,0,x_2)=(1,0)$, as well as
\begin{gather*}
\label{eq:Phitcontrol}
\|\p_t^j \Phi(t,\cdot)\|_{L^\infty(\Omega^+)} \leq \frac1{\sqrt{2\pi}}\,
\|\p_t^j \varphi(t,\cdot)\|_{H^{\frac32}(\R)} \qquad t\in[0,T] \,,   \\
\label{eq:Phiest}
\| \Psi_1(t,\cdot)\|_{L^\infty(\Omega^+)} \leq \frac{1}{2} \qquad t\in[0,T]\,,
\end{gather*}
for $j=1,\ldots,m-{2}$.
\end{Lem}
\begin{proof}
See Appendix \ref{AppA}.
\end{proof}

%
%In particular, from~\eqref{eq:Phiest} we derive $1+\Psi_1(t,x) \geq 1/2$.
We introduce a change of  variables by setting
\begin{gather*}
\tilde{U}(t,x )\doteq  {U}(t,\Phi (t,x))\, , \quad \tilde{\Hc}(t,x )\doteq  \Hc(t,\Phi (t,x))\, , \quad
\tilde{\Ec}(t,x)\doteq  \Ec(t,\Phi(t,x))\, ;% \quad \tilde{\Wc}(t,x )\doteq  \Wc(t,\Phi (t,x)) \, ;
\end{gather*}
then, we conveniently drop the tilde sign. The (reflected) vacuum equations, in $\Omega^+$, become
\[ \begin{cases}
\eps\p_t \Hc_1 + \p_2 \Ec - \frac1{1+\Psi_1}\,(\eps\Psi_t\p_1\Hc_1+\Psi_2\p_1\Ec)= 0\,,  \\
\eps\p_t \Hc_2 - \frac1{1+\Psi_1}\,(\eps\Psi_t\p_1\Hc_2 - \p_1 \Ec) = 0\,,  \\
\eps\p_t \Ec + \p_2 \Hc_1 - \frac1{1+\Psi_1} (\eps\Psi_t\p_1\Ec-\p_1 \Hc_2+\Psi_2\p_1\Hc_1) = 0\,.
\end{cases}  \]
Setting $\Wc=(\Hc_1,\Hc_2,\Ec)\tm$, the previous equations can be written as a symmetric system:
\[ \eps\p_t \Wc + B_1(\Psi)\p_1\Wc + B_2\p_2\Wc = 0, \]
where
\[ B_1(\Psi) = -\frac1{1+\Psi_1} \begin{pmatrix}
\eps\Psi_t & 0 & \Psi_2 \\
0 & \eps\Psi_t & -1 \\
\Psi_2 & -1 & \eps\Psi_t
\end{pmatrix}\,, \qquad B_2 = \begin{pmatrix}
0 & 0 & 1 \\
0 & 0 & 0 \\
1 & 0 & 0
\end{pmatrix}\,. \]

As for plasma, we have $U=(q,v_1,v_2,H_1,H_2,S)\tm$ that satisfies, before the change of variables,
\[ \sum_{j=0,1,2} A_j(U)\p_j U =0\,, \]
where
\begin{align}
A_0(U) \label{A0}
	& =\begin{pmatrix}
	\rho_p/\rho & 0 & 0 & -H_1\rho_p/\rho & -H_2\rho_p/\rho & 0 \\
	0 & \rho & 0 & 0 & 0 & 0 \\
	0 & 0 & \rho & 0 & 0 & 0 \\
	-H_1\rho_p/\rho & 0 & 0 & 1 + H_1^2\rho_p/\rho & H_1H_2\rho_p/\rho & 0 \\
	-H_2\rho_p/\rho & 0 & 0 & H_1H_2\rho_p/\rho & 1 + H_2^2\rho_p/\rho & 0 \\
	0 & 0 & 0 & 0 & 0 & 1
	\end{pmatrix},\\
A_1(U) \label{A1}
	& = \begin{pmatrix}
	v_1\rho_p/\rho & 1 & 0 & -H_1v_1\rho_p/\rho & -H_2v_1\rho_p/\rho & 0 \\
	1 & \rho v_1 & 0 & -H_1 & 0 & 0 \\
	0 & 0 & \rho v_1 & 0 & -H_1 & 0 \\
	-H_1v_1\rho_p/\rho & -H_1 & 0 & (1 + H_1^2\rho_p/\rho)v_1 & H_1H_2v_1\rho_p/\rho & 0 \\
	-H_2v_1\rho_p/\rho & 0 & -H_1 & H_1H_2v_1\rho_p/\rho & (1 + H_2^2\rho_p/\rho)v_1 & 0 \\
	0 & 0 & 0 & 0 & 0 & v_1
\end{pmatrix}, \\
A_2(U) \label{A2}
	& = \begin{pmatrix}
	v_2\rho_p/\rho & 0 & 1 & -H_1v_2\rho_p/\rho & -H_2v_2\rho_p/\rho & 0 \\
	0 & \rho v_2 & 0 & -H_2 & 0 & 0 \\
	1 & 0 & \rho v_2 & 0 & -H_2 & 0 \\
	-H_1v_2\rho_p/\rho & -H_2 & 0 & (1 + H_1^2\rho_p/\rho)v_2 & H_1H_2v_2\rho_p/\rho & 0 \\
	-H_2v_2\rho_p/\rho & 0 & -H_2 & H_1H_2v_2\rho_p/\rho & (1 + H_2^2\rho_p/\rho)v_2 & 0 \\
	0 & 0 & 0 & 0 & 0 & v_2
	\end{pmatrix}.
\end{align}
After changing the variables and removing the tilde from $\tilde U$, we have
\[ \sum_{j=0,2} A_j(U)\p_j U + \tilde{A}_1(U)\p_1U=0\,, \qquad
\tilde{A}_1=\frac1{1+\Psi_1}\,\bigl(A_1-\Psi_tA_0-\Psi_2A_2\bigr)\,. \]

\medskip

\textbf{Linearization.} We linearize the problem about a piecewise constant basic state $\base U =(\base q, \base v, \base H, \base S) $, denoted by a dot sign, with  $\base\varphi=0$, so that we can take $\base\Psi=0$. Consequently, $\base
N=(1,0)$ and $\base{v}_N=v_1=0$, since~$\base\varphi_t=0$. We obtain:
\[ \base H_1\equiv \base H_N=0, \qquad \base \Hc_1\equiv\base\Hc_N=0, \qquad \base \Ec =
-\eps\base\varphi_t\base\Hc_2=0.\]
To summarize,
\[ \base\varphi=0\,,\qquad \base N=(1,0)\,,\qquad \base v=(0,\base v_2)\,,\qquad \base H=(0,\base
H_2)\,,\qquad \base \Hc=(0,\base\Hc_2)\,,\qquad \base\Ec=0\,. \]
To simplify notations, from now on we will omit subscripts in  $\base v_2, \base H_2, \base \Hc_2$, and we will denote by  $\base U$ the variable $\base U=(\base q, \base v_2, \base H_2, \base S)=(\base q, \base v, \base H, \base S)\in\R^4$. We denote by $\base\rho$ the corresponding value of the density, $\base\rho=\rho(\base q -\base H_2^2/2,\base S)$, and with $\base c$ the sound speed defined by $\base c=\sqrt{p'_\rho(\base\rho,\base S)}$.

%\textcolor{red}{Poniamo~$\Omega=\{(t,x)\in\R^3: x_1>0\}$. Le equazioni di Maxwell in~$\Omega_-(t)$ erano gi\`a state ribaltate su $\Omega_-(t)$, cos\`\i\ dopo la linearizzazione valgono anch'esse in~$\Omega$.}

After the linearization, the system for  $\Wc$ in $\Omega^+$ becomes
\[ \eps\p_t \Wc + B_1\p_1\Wc + B_2\p_2\Wc = 0\,, \]
where
\[ B_1 =  \begin{pmatrix}
0 & 0 & 0 \\
0 & 0 & 1 \\
0 & 1 & 0
\end{pmatrix}\,, \qquad B_2 = \begin{pmatrix}
0 & 0 & 1 \\
0 & 0 & 0 \\
1 & 0 & 0
\end{pmatrix}\,. \]
The system for~$U$ is instead
\[ \sum_{j=0,1,2} A_j(\base U)\p_j U =0\,, \]
where, setting $\rhop\doteq (\base\rho\ \! \base c^2)^{-1}$, we have
\begin{align*}
A_0(\base U)
	& =\begin{pmatrix}
	\rhop & 0 & 0 & 0 & -\rhop\base{H} & 0 \\
	0 & \base{\rho} & 0 & 0 & 0 & 0 \\
	0 & 0 & \base{\rho} & 0 & 0 & 0 \\
	0 & 0 & 0 & 1 & 0 & 0 \\
	-\rhop\base{H} & 0 & 0 & 0 & 1 + \rhop\base{H}^2 & 0 \\
	0 & 0 & 0 & 0 & 0 & 1
	\end{pmatrix},\\
A_1(\base U)
	& = \begin{pmatrix}
	0 & 1 & 0 & 0 & 0 & 0 \\
	1 & 0 & 0 & 0 & 0 & 0 \\
	0 & 0 & 0 & 0 & 0 & 0 \\
	0 & 0 & 0 & 0 & 0 & 0 \\
	0 & 0 & 0 & 0 & 0 & 0 \\
	0 & 0 & 0 & 0 & 0 & 0
\end{pmatrix}, \\
A_2(\base U)
	& = \begin{pmatrix}
	\rhop\base{v} & 0 & 1 & 0 & -\rhop\base{v}\base{H} & 0 \\
	0 & \base{\rho} \base{v} & 0 & -\base{H} & 0 & 0 \\
	1 & 0 & \base{\rho} \base{v} & 0 & -\base{H} & 0 \\
	0 & -\base{H} & 0 & \base{v} & 0 & 0 \\
	-\rhop\base{v}\base{H} & 0 & -\base{H} & 0 & (1 + \rhop\base{H}^2)\base{v} & 0 \\
	0 & 0 & 0 & 0 & 0 & \base{v}
	\end{pmatrix}.
\end{align*}
Observe that the equation for the entropy can now be decoupled from the system, since the coefficients depend just on the basic state:
\[ S_t + \base{v} S_2 =0\,. \]
From now on, we will omit the component ~$S$ in~$U$, and will consider the new variable
\begin{align*}
U=(q,v_1,v_2,H_1,H_2)\tm \in\R^5
\end{align*}
 (for simplicity, we keep the same notation $U$). Similarly, $A_j(\base U)$ will denote the  $5\times5$ north-west block in the previous matrices, while  $\base U$ is still the element of~$\R^4$ written above.

The linearized boundary conditions are
\begin{align} \varphi_t = v_1-\base v\varphi_2 +g_1 \,,\qquad q = \base\Hc \Hc_2 + g_2 \,,\qquad \Ec =
-\eps\base\Hc\varphi_t + g_3 \,\label{lbc}\end{align}
(the other conditions, i.e. the normal components of the magnetic fields equal to zero, can be recovered by imposing them at the initial time).

We set $V=(U,\Wc)\tm$ and distinguish between characteristic ($V^\car$) and noncharacteristic ($V^\nc$) variables
\[ V^\car =(v_2,H_1,H_2,\Hc_1) \,, \qquad V^\nc = (q,v_1,\Hc_2,\Ec)\,. \]
Notice that characteristic variables do not appear in the boundary conditions. If we define
\[ \Lc = \begin{pmatrix}
\sum_{j=0,1,2} A_j(\base{U})\p_j & 0 \\
0 & \eps\p_t+B_1\p_1+B_2\p_2
\end{pmatrix} \equiv \sum_{j=0,1,2} \Ac_j(\base{U})\p_j ,\]
\[ M = \begin{pmatrix}
0 & -1 & 0 & 0 \\
1 & 0 & -\base\Hc & 0 \\
0 & 0 & 0 & 1
\end{pmatrix}, \qquad b(\p_t,\p_2)= \begin{pmatrix}
\p_t + \base v \p_2 \\
0 \\
\eps\base\Hc\p_t
\end{pmatrix}, \]
\[
\Omega=\R\times\Omega^+,\qquad \omega=\R\times\Gamma,\] 
the inhomogeneous linearized problem can be written as
\begin{align}
& \Lc V = f\,, \qquad \text{in} \quad \Omega \label{prob.in}\\
& \Bc (V^\nc,\varphi) \doteq M V^\nc + b \varphi = g\,, \qquad \text{on} \quad \omega\,  . \label{bc}
\end{align}

\medskip

\section{Main results}\label{Main results}

First, we define some function spaces and the corresponding norms.

\begin{Def}
For every~$s\geq0$ and $\gamma\geq1$, we set
\[ H^s_\gamma(\R^2)\doteq \{ u\in\Dc'(\R^2_{t,x_2}): \ e^{-\gamma t}u \in H^s(\R^2_{t,x_2}) \} ,\]
with (equivalent) norm
\[ \|u\|_{H^s_\gamma(\R^2)}^2 \doteq \|e^{-\gamma t}u\|_{s,\gamma}^2,\qquad \|v\|_{s,\gamma} \doteq \int_{\R^2} (\gamma^2+\delta^2+\eta^2)^s |\fou v(\delta,\eta)|^2 \,d\delta d\eta, \]
where $\hat v$ is the Fourier transform of $v$ and ~$\delta,\eta$ are the dual variables of~$t,x_2$.

We also define~$L^2((0,\infty),H^s_\gamma(\R^2))$, briefly denoted by~$L^2(H^s_\gamma)$, as the space of distributions with finite $L^2(H^s_\gamma)$-norm, where
\[ \|u\|_{L^2(H^s_\gamma)}^2 \doteq \int_0^\infty \|u(\cdot,x_1,\cdot)\|_{H_\gamma^s(\R^2)}^2 dx_1 , \]
and set
\[ \|v\|_{0,s,\gamma}^2 \doteq \int_0^\infty \|v(\cdot,x_1,\cdot)\|_{s,\gamma}^2 dx_1\,. \]

Finally, we set $L^2_\gamma(\R^2)=H^0_\gamma(\R^2)$ and $L^2_\gamma(\Omega)=L^2(H^0_\gamma)$. Notice that ~$\|u\|_{L^2(H^s_\gamma)}=\|e^{-\gamma t}u\|_{0,s,\gamma}$ and~$\|u\|_{L^2_\gamma}=\|e^{-\gamma t}u\|_{L^2}$.
\end{Def}

\textbf{Assumptions on the basic state.} We assume:
\begin{gather}\label{eq:Hvbase}
\base v \neq0\,, \qquad \base \Hc \neq 0\,, \\
\label{eq:Hrbase} \base \rho >0\,, \qquad \base H \neq 0\,, \\
\label{eq:notnull}  \rhop\base H^2 \neq 1\,, \\
\label{eq:notalfven} |\base v| \neq \dfrac{|\base H|}{\sqrt{\base \rho}}\,,\\
\label{eq:notvH} |\base v| \neq  \dfrac{|\base H|}{\sqrt{\base\rho(1+\base\alpha\base H^2)}}\,, \\
\label{eq:notres} \base\rho\base v^2\neq\ \dfrac{\rhop (\base H^4 - \base \Hc^4) + 2\base H^2 \pm \sqrt{\rhop^2(\base H^4-\base \Hc^4)^2+4\base \Hc^4}}{2(1+\rhop \base H^2)} \quad \text{ if }\quad \base H^2<\min\{\rhop^{-1}, \base\rho\base v^2\}\,,
\end{gather}
where \eqref{eq:notnull} and \eqref{eq:notalfven} respectively mean that the sound speed $\base c$ and the velocity $\base v$ do not coincide with the Alfv\'en velocity ${|\base H|}/{\sqrt{\base \rho}}$.

The usefulness of these conditions will be clear in the following, but we observe immediately that they avoid trivial cases (vanishing velocity or magnetic fields) and prevent the appearance of resonances, i.e. multiple roots of the Lopatinski\u{\i} determinant {(see Section~\ref{sec.lopatinskii}; \eqref{eq:notres} corresponds to \eqref{eq:nottau}).}

The main result that we obtain can be stated as follows.
\begin{Thm}\label{Thm:main0}
Assume that the conditions {\eqref{eq:Hvbase}--\eqref{eq:notres}}  on the basic state are satisfied.

Then there exist constants~$\eps_0>0,C>0$ such that, for every~$0<\eps<\eps_0$, $\gamma\geq1$ and every function~$(V,\varphi)\in H_\gamma^2(\Omega)\times H_\gamma^2(\Gamma)$, there holds the estimate
\[ \gamma\|V\|_{L^2_\gamma(\Omega)}^2 + \|V^\nc_{x_1=0}\|_{L^2_\gamma(\Gamma)}^2 + \|\varphi\|_{H^1_\gamma(\Gamma)}^2\leq C \left( \frac1{\gamma^3}\|\Lc V\|_{L^2(H^1_\gamma)}^2+\frac1{\gamma^2}\|\Bc (V^\nc,\varphi)\|_{H^1_\gamma(\Gamma)}^2 \right).
\]

\end{Thm}

Now, for any~$\gamma\geq1$, let us define
\[ \Lc^\gamma \doteq \Lc + \gamma \Ac_0\,, \qquad b^\gamma(\p_t,\p_2) \doteq b (\p_t+\gamma,\p_2)\,, \qquad \Bc^\gamma (V^\nc,\varphi) \doteq M V^\nc + b^\gamma \varphi = g \,. \]
It is easily shown that Theorem~\ref{Thm:main0} admits the following equivalent proposition.
\begin{Prop}\label{Thm:main}
Assume that the conditions {\eqref{eq:Hvbase}--\eqref{eq:notres}} on the basic state are satisfied.
	
Then	 there exist constants~$\eps_0>0,C>0$ such that, for every~$0<\eps<\eps_0$, $\gamma\geq1$ and every function~$(V,\varphi)\in H^2(\Omega)\times H^2(\Gamma)$, there holds the estimate
\begin{align}
 \gamma\|V\|_{L^2(\Omega)}^2 + \|V^\nc_{x_1=0}\|_{L^2(\Gamma)}^2 + \|\varphi\|_{1,\gamma}^2\leq C \left( \frac1{\gamma^3}\|\Lc^\gamma V\|_{0,1,\gamma}^2+\frac1{\gamma^2}\|\Bc^\gamma (V^\nc,\varphi)\|_{1,\gamma}^2 \right).\label{main-est}
\end{align}
\end{Prop}
The rest of the paper is devoted to the proof of Theorem~\ref{Thm:main0} in the equivalent form of Proposition \ref{Thm:main}.

\section{Some reductions}\label{reductions}

\textbf{Partial homogenization.} As in \cite{nappes,majda-osher}, in order to prove \eqref{main-est} we first remove the forcing term $f$ in~$\Lc^\gamma V=f$.
%
%\begin{Rem}
Given $(V,\varphi)\in H^2(\Omega)\times H^2(\Gamma)$, set $f\doteq\Lc^\gamma V \in H^1(\Omega), g\doteq\Bc^\gamma (V^\nc,\varphi)\in H^1(\Gamma)$ and consider the auxiliary problem
\[\begin{cases}
\Lc^\gamma V_1 = f\,, & \quad x_1>0\,,\\
M^\aux V_1^\nc =0\,, & \quad x_1=0\,,
\end{cases}\]
where
\[ M^\aux=\begin{pmatrix}
1 & 0 & 0 & 0 \\
0 & 0 & 0 & 1
\end{pmatrix},
 \]
which corresponds to the boundary conditions  $q=0$ and $\Ec=0$. Hence, if $V_1=(U,\Wc)\tm$, at the boundary we have $A_1 U\cdot U = 2qv_1 = 0$ and $B_1 \Wc \cdot \Wc=2\Ec \Hc_2=0$, so that the boundary matrix $-\Ac_1(\base U)$ made of $A_1$ and $B_1$ is nonnegative. Moreover, it is maximally nonnegative, since both  $A_1$ and $B_1$ have exactly one positive ($\lambda=1$) and one negative ($\lambda=-1$) eigenvalue, while the other ones equal zero. This means that the stable subspace (corresponding to the negative eigenvalues) has dimension 2, which is the correct dimension that provides maximality.

Since the boundary conditions are maximally dissipative, the standard theory for hyperbolic systems \cite{lax-phillips} guarantees the existence of a unique solution $V_1\in L^2(\R^+;H^1(\Gamma))$ of the problem above, with trace of~$V_1^\nc$ in~$H^1(\Gamma)$ and
\begin{equation}
\begin{array}{ll}\label{stimaV1}
\ds \gamma \|V_1\|^2_{L^2(\Omega)} \leq \frac{C}\gamma \|f\|^2_{L^2(\Omega)}\,, \qquad
\|V_1^\nc |_{x_1=0}\|^2_{1,\gamma}\leq \frac{C}\gamma \|f\|^2_{0,1,\gamma}\,.
\end{array}
\end{equation}
We obtain that ~$V_2\doteq V-V_1$ satisfies the system
\[\begin{cases}
\Lc^\gamma V_2 = 0\, , & \quad  x_1>0\,, \\
\Bc^\gamma (V_2^\nc,\varphi) = g - MV_1^\nc|_{x_1=0}\,  & \quad x_1=0\,,
\end{cases}\]
where
\begin{equation}
\begin{array}{ll}\label{stimag1}

\ds \|g - MV_1^\nc|_{x_1=0}\|^2_{1,\gamma}\le 2\|g\|^2_{1,\gamma}+\frac{C}{\gamma} \|f\|^2_{0,1,\gamma}.

\end{array}
\end{equation}
Consequently, it will be sufficient to prove \eqref{main-est} in the case~$\Lc^\gamma V=0$. By abuse of notation, we continue to write~$g$ for~$g - MV_1^\nc|_{x_1=0}$ and~$V$ for~$V_2$.
%\end{Rem}

\textbf{Eliminating the front.} We proceed as in \cite{nappes}. We perform a Fourier transform with respect to the variables $t,x_2$, whose dual variables will be $\delta,\eta$, of
\[ \Lc^\gamma V = 0 \,, \qquad M V^\nc + b^\gamma \varphi = g \,.\]
Setting \[\tau=\gamma+i\delta\,,\] we obtain
\begin{equation}
\label{15}
\begin{cases}
\bigr(\tau \Ac_0 + i\eta \Ac_2 + \Ac_1 \d1\bigl) \fou V = 0 & \quad x_1>0\,, \\
M \fou{V}^\nc(\delta,0,\eta) + b(\tau,i\eta) \fou\varphi (\delta,\eta) = \fou g (\delta,\eta) & \quad x_1=0\,,
\end{cases} 
\end{equation}
where
\[ b(\tau,i\eta)= \begin{pmatrix}
\tau + i \base v \eta \\
0 \\
\eps\base\Hc\tau
\end{pmatrix}, \]
which is homogeneous of degree~$1$ with respect to $(\tau,\eta)$. In order to take into account homogeneity, we consider the hemisphere
\[ \Sigma = \left\{ (\tau,\eta) \in \C \times \R : \ |\tau|^2 + \eta^2 =1\,, \ \Re \tau \geq0 \right\}\,, \]
and set
\[ \Xi = \left\{ (\gamma,\delta,\eta) \in [0,+\infty)\times\R^2 \right\} \setminus \{(0,0,0)\} = (0,+\infty)
\cdot \Sigma \, \]
(essentially, we are partitioning the halfspace  $\Xi$ using hemispheres with radius~$r\in(0,+\infty)$).
Notice that the symbol  $b(\tau,i\eta)$ is elliptic, i.e. it is always different from zero on~$\Sigma$.

We set $k=\sqrt{|{\tau}|^2+\eta^2}$ and define, in $\Xi$,
\begin{align} Q = \frac{1}{k}\begin{pmatrix}
0 & k & 0 \\
-\eps\base \Hc \tau & 0 & \tau + i \base v \eta \\
\bar\tau - i \base v \eta & 0 & \eps\base \Hc \bar \tau
\end{pmatrix}, \label{Q} \end{align}
where the bar sign denotes complex conjugation, so that $Q\in\mathcal{C}^\infty(\Xi,GL_3(\C))$ is homogeneous of degree~$0$ in $(\tau,\eta)$ and
 $Q(\tau,\eta)b(\tau,i\eta)=(0,0,\theta(\tau,\eta))$ is homogeneous of degree~$1$, with
\[ \theta(\tau,\eta) 	 = k^{-1} |b(\tau,i\eta)|^2\, , \qquad  \min_\Sigma |\theta(\tau,\eta)| >0\,, \]
since the third line of $Q$ is $k^{-1}\bar b(\tau,i\eta)\tm$.
More precisely, on  $\Sigma$ (where $k=1$), we have
\[ \theta(\tau,\eta) = (1+\eps^2\base \Hc^2) |\tau|^2 + \base v^2\eta^2 +2\base v\eta\delta = \gamma^2 +
\eps^2\base\Hc^2|\tau|^2+ (\delta+\base v\eta)^2 \,,\]
which is zero if and only if
\[ \begin{cases}
\gamma=0\,,\\
\base\Hc^2\delta^2=0\,,\\
\delta=-\base v\eta\,;
\end{cases} \]
thanks to~\eqref{eq:Hvbase}, $\theta>0$ on~$\Sigma$.

From the boundary conditions in \eqref{15} we have
\begin{align} \label{eq.elimfront} \begin{pmatrix}
0 \\
0 \\
\theta(\tau,\eta)
\end{pmatrix} \fou \varphi (\delta,\eta)+ \begin{pmatrix}
\beta(\tau,\eta) \\
\ell(\tau,\eta)
\end{pmatrix} \fou V^\nc (\delta,0,\eta) = Q(\tau,\eta)\fou g(\delta,\eta)\,, \end{align}
where~$\beta$ is a $2\times4$ matrix and $\ell$ is a covector in~$\R^4$ such that
\[ Q(\tau,\eta) M = \begin{pmatrix}
\beta(\tau,\eta) \\
\ell(\tau,\eta)
\end{pmatrix}\,. \]

From now on, we assume to work on $\Sigma$, where $k=1$ (otherwise the first line of $\beta$ below has to be multiplied by $k$), and we have
\begin{align*}
\det Q
	& = \eps^2\base \Hc^2 |\tau|^2 + |\tau+i\base v\eta|^2 = |b(\tau,i\eta)|^2\,, \\
\theta(\tau,\eta)
	& = |b(\tau,i\eta)|^2 \,,\\
\beta
	& = \begin{pmatrix}
1 & 0 & -\base \Hc & 0 \\
0 & \eps\base \Hc \tau & 0 & \tau + i\base v\eta
\end{pmatrix}\,,\\
\ell
	& = \begin{pmatrix}
0 & -\bar\tau + i \base v \eta & 0 & \eps\base \Hc \bar \tau
\end{pmatrix}\,.
\end{align*}
The third line of \eqref{eq.elimfront} is
\begin{align*}
\theta(\tau,\eta)\fou\varphi+\ell(\tau,\eta)\fou V^\nc(0)=\bar b(\tau,i\eta)\tm\fou g\, .
\end{align*}
Now, we proceed as in \cite{nappes}. From the ellipticity of $b$ and the uniform boundedness of $\ell$ and $b$,  we obtain
\begin{align*}
k^2|\fou\varphi|^2\leq C(|\fou V^\nc(0)|^2+|\fou g|^2)
\end{align*}
in $\Xi$. Integrating over $(\delta,\eta)\in \R^2$ and using Plancherel's Theorem, we deduce
\begin{align}\label{stimaphi}
\ds \|\varphi\|_{1,\gamma}^2\leq C \Bigl( \|V^\nc_{x_1=0}\|_{L^2(\R^2)}^2 + \frac{1}{\gamma^2} \|g\|_{1,\gamma}^2   \Bigr).
\end{align}
Consequently, in order to prove Proposition~\ref{Thm:main}, it is sufficient to consider
\begin{align} \label{eq.trasf}
\bigr(\tau \Ac_0 + i\eta \Ac_2 + \Ac_1 \d1\bigl) \fou V = 0 \,  & \quad x_1>0\,, \\
\beta(\tau,\eta)\fou V^\nc=Q(\tau,\eta)\fou g\,  & \quad x_1=0\,,\label{eq.trasfbdry}
\end{align}
neglecting the front term in the boundary condition. 
\begin{Lem}%\label{}
Assume that there exists a positive constant $C$ such that the solution of \eqref{eq.trasf}, \eqref{eq.trasfbdry} fulfills the estimate
\begin{equation}
\begin{array}{ll}\label{stimaV2bdry}
\ds \|V^\nc_{x_1=0}\|_{L^2(\R^2)}^2\leq  \frac{C}{\gamma^2} \|g\|_{1,\gamma}^2
 \,.
\end{array}
\end{equation}
Then the thesis of Propositions~\ref{Thm:main} is satisfied.
\end{Lem}
\begin{proof}
We recall that all matrices $\Ac_j$ are symmetric and that  $\Ac_0$ is positive definite. Taking the scalar product of \eqref{eq.trasf} with $\fou V$ and integrating yields the following inequality
\begin{equation}
\begin{array}{ll}\label{stimaV2}
\gamma \|V\|^2_{L^2(\Omega)} \leq C
 \|V^\nc_{x_1=0}\|_{L^2(\R^2)}^2\,.
\end{array}
\end{equation}
Combining \eqref{stimaV1}, \eqref{stimag1}, \eqref{stimaphi} with \eqref{stimaV2bdry}, \eqref{stimaV2} gives \eqref{main-est}.
\end{proof}
Therefore, for the derivation of the estimate \eqref{main-est} it is sufficient to get an estimate of the trace of $\fou V^\nc$ on~$x_1=0$ of the form \eqref{stimaV2bdry}.

%%%%%%%%%%%%%%%%%%%%%%%%%%%%%%%%%%%%%%%%%%%%%%%%%%%%%%

\section{Normal modes analysis}\label{normal}

We begin by distinguishing a few cases, since each one needs a different analysis. We recall that $\tau=\gamma+i\delta$ and set \[{\mu \doteq \tau + i\base v\eta}\,.\]
We introduce the following cases:
\begin{gather}
\label{eq:caso1}\tag{P1}
\mu=0, \qquad \text{i.e.} \qquad \gamma=0 \quad \text{and} \quad \delta = -\base v\eta \quad (\eta\neq 0)\,; \\
\label{eq:caso2}\tag{P2}
\gamma=0 \quad \text{and} \quad \delta = -\base v\eta \pm \frac{\base H}{\sqrt{ \base\rho( 1+\base\alpha\base H^2 )}}\,\eta \quad (\eta\neq 0)\,;\\
\label{eq:caso3}\tag{P3}
\tau=0, \qquad \text{i.e.} \qquad \gamma=0 \quad \text{and} \quad \delta = 0 \quad (\eta\neq 0)\,; \\
\label{eq:caso0}\tag{P0}
\text{all other possibilities.}
\end{gather}

Notice that, in \eqref{eq:caso1}--\eqref{eq:caso3}, we always have $\eta\neq 0$, since $\eta=0$ would imply also  $\delta=0$ and hence $(\gamma,\delta,\eta)=(0,0,0)\notin\Sigma$.

Moreover, these three cases are mutually exclusive, thanks to the hypotheses on the basic state.  Actually, \eqref{eq:caso1} and~\eqref{eq:caso2} would coincide only if~$\eta=0$, since by assumption~$\base H\neq0$. Similarly, \eqref{eq:caso1} and~\eqref{eq:caso3} can not hold at the same time (using $\base v\neq 0$), while \eqref{eq:caso2} and~\eqref{eq:caso3} coincide if and only if
\[%\label{eq:indipendent}
\base v = \pm \frac{\base H}{\sqrt{\base\rho(1+\base\alpha\base H^2)}}\,,
\]
which is prevented thanks to \eqref{eq:notvH}.

\medskip
In this section, we consider only \eqref{eq:caso0}: the other cases correspond to poles for the symbol $\Ac(\tau,\eta)$ in \eqref{symbolic-equ} and
will be handled later on (see Subsections \ref{polotau}, \ref{polomu}, \ref{polocaso2}), see also Remark~\ref{rk.poles}.

We introduce the following quantities, whose meaning and relevance will be clear from the following computations:

\begin{alignat}{2}
& a_{12}(\tau,\eta) = - \frac{\mu^2\base\rho+\eta^2\base
H^2}\mu\, , &\qquad& a_{21}(\tau,\eta) = -\frac{\mu(\rhop \base\rho \mu^2+\eta^2)}{(\mu^2\base\rho\rhop+\eta^2)\base H^2+\mu^2\base\rho}\, ,  \label{def.a12case0} \\
& a_{34}(\tau,\eta) = -\frac{\eps^2\tau^2+\eta^2}{\eps\tau}\, , &\qquad& a_{43}(\tau,\eta) = -\eps\tau\, . \label{def.a34case0}
\end{alignat}

From the plasma part of \eqref{eq.trasf}, we obtain the following algebraic equations:
\begin{gather*}
\tau \base \rho \fou v_2 + i\eta \bigl( \fou q + \base \rho \base v \fou v_2 -\base H \fou H_2 \bigr) =0 \,, \\
\tau \fou H_1 + i\eta \bigl( -\base H \fou v_1 + \base v \fou H_1 \bigr) = 0 \,,\\
\tau \bigl( -\rhop\base H  \fou q + (1+\rhop\base H^2) \fou H_2\bigr) +i\eta \bigl( -\rhop\base v\base H \fou
q - \base H \fou v_2 + (1+\rhop \base H^2)\base v \fou H_2 \bigr) =0\,,
\end{gather*}
or
\begin{gather}
\mu \base \rho \fou v_2 + i\eta (\fou q -\base H \fou H_2) =0 \,, \label{p.1}\\
\mu \fou H_1 - i\eta\base H \fou v_1 = 0 \,, \label{p.2}\\
\mu \bigl( -\rhop\base H  \fou q + (1+\rhop\base H^2) \fou H_2\bigr) -i\eta \base H \fou v_2 =0\,. \label{p.3}
\end{gather}
Since we are not in \eqref{eq:caso1}, from \eqref{p.2} we obtain
\begin{align} \fou H_1 = i\frac\eta\mu \, \base H \fou v_1 \,. \label{plasma_alg_H1v1} \end{align}
Multiplying \eqref{p.1} by~$i\eta\base H$ and adding \eqref{p.3} multiplied by $\mu\base\rho$,
we have
\begin{gather*}
\mu \base \rho \fou v_2 + i\eta (\fou q -\base H \fou H_2) =0 \,, \\
\mu^2\base\rho \bigl( -\rhop\base H  \fou q + (1+\rhop\base H^2) \fou H_2\bigr) -\eta^2\base H(\fou q -\base H \fou H_2)  =0\,,
\end{gather*}
from which
\begin{align}
\label{eq.H2}\fou H_2
	& = \frac{(\mu^2\base\rho\rhop+\eta^2)\base H}{\mu^2\base\rho(1+\rhop\base H^2)+\eta^2\base H^2}\,\fou q  = \frac{(\mu^2\base\rho\rhop+\eta^2)\base H}{(\mu^2\base\rho\rhop+\eta^2)\base H^2+\mu^2\base\rho}\,\fou
q\,, \\
\label{eq.v2}\fou v_2
	& = -i\,\frac\eta{\mu\base\rho}\,(\fou q -\base H \fou H_2)  = -i\,\frac{\eta\mu}{(\mu^2\base\rho\rhop+\eta^2)\base H^2+\mu^2\base\rho}\,\fou q\,,
\end{align}
provided
\begin{equation}\label{eq:denominatori}
\mu^2\base\rho(1+\rhop\base H^2)+\eta^2\base H^2 \neq 0\,.
\end{equation}
Recalling that $\mu=\gamma+i(\delta+\base v\eta)$, we have that if $\gamma\neq0$, then \eqref{eq:denominatori} surely holds, since the imaginary part in the left-hand side is different from zero if $\delta\neq-\base v\eta$, otherwise  $\mu^2=\gamma^2>0$ and we have the sum of positive quantities. If $\gamma=0$,  \eqref{eq:denominatori} becomes
\begin{equation}\label{eq:vanish}
\eta^2\base H^2 - (\delta+\base v\eta)^2\base\rho(1+\rhop\base H^2)\neq
0 \,,
\end{equation}
which is true since we are not in \eqref{eq:caso2}.

The plasma system also provides two differential equations for~$\fou q, \fou v_1$:
\begin{gather}
\tau\rhop(\fou q -\base H \fou H_2) +i\eta (\rhop\base v \fou q + \fou v_2 -\rhop\base v\base H \fou H_2) +
 \frac{d\fou v_1 }{dx_1}= 0\,, \label{plasma_v'}\\
\tau\base \rho \fou v_1 +i\eta (\base\rho \base v \fou v_1 -\base H \fou H_1) +  \frac{d\fou q }{dx_1}=0\,. \label{plasma_q'}
\end{gather}
Substituting \eqref{plasma_alg_H1v1}, \eqref{eq.H2}, \eqref{eq.v2} in \eqref{plasma_v'}, we obtain
\begin{align}
 \frac{d\fou q }{dx_1}
	& = -\mu \base \rho \fou v_1 +i \eta \base H \fou H_1 \notag \\
	& = -\mu \base \rho \fou v_1 - \frac{{\eta}^2}\mu \base H^2 \fou v_1 = - \frac{\mu^2\base\rho+\eta^2\base
H^2}\mu\, \fou v_1 = a_{12}(\tau,\eta)\,\fou v_1 \,, \notag\\
 \frac{d\fou v_1 }{dx_1}
	& = - \mu \rhop\fou q + \mu \rhop\base H \fou H_2 -i\eta\fou v_2 \label{v1p} \\
	& = \frac{- \mu \rhop \bigl( (\mu^2\base\rho\rhop+\eta^2)\base H^2+\mu^2\base\rho \bigr)+ \mu \rhop\base
H^2(\mu^2\base\rho\rhop+\eta^2) -\eta^2\mu}{(\mu^2\base\rho\rhop+\eta^2)\base H^2+\mu^2\base\rho}\,\fou q \notag\\
	& = -\frac{\mu(\rhop \base\rho \mu^2+\eta^2)}{(\mu^2\base\rho\rhop+\eta^2)\base H^2+\mu^2\base\rho}\,\fou
q= a_{21}(\tau,\eta)\,\fou q\,, \notag
\end{align}
where we have used \eqref{def.a12case0}.
Notice that
\[ a_{12}= -\frac{\mu^2\base\rho+\eta^2\base H^2}\mu\]
is well-defined, since $\mu\neq 0$, and different from zero. Indeed, if~$\delta\neq-\base v\eta$, we have
\[ \Im (\mu^2\base\rho+\eta^2\base H^2) = 2\gamma\base\rho(\delta+i\base v\eta)\neq0\,,\]
while if~$\delta=-\base v\eta$, then~$\mu=\gamma$ and
\[ a_{12}=-\frac{\gamma^2\base\rho+\eta^2\base H^2}\gamma \neq0 \,.\]
As far as~$a_{21}$ is concerned, we have already proved that the denominator is different from zero, since this is equivalent to~\eqref{eq:vanish}. Similar arguments show that also the numerator of~$a_{21}$ is different from zero.

Now, let us consider the vacuum block of \eqref{eq.trasf}. Substituting the algebraic equation
\begin{align} \eps\tau \fou \Hc_1 + i\eta \fou \Ec=0\,, \qquad \text{i.e.} \quad \fou\Hc_1 = -i\frac{\eta}{\eps\tau}
\fou \Ec \,, \label{vacuum_alg}\end{align}
we have (since we are not in \eqref{eq:caso3}, $\tau\neq0$, that is $\gamma\neq0$ or $\delta\neq0$):
\begin{align*}
 \frac{d\fou \Hc_2 }{dx_1}& = -\eps\tau \fou \Ec - i\eta\fou\Hc_1 = -\frac{\eps^2\tau^2+\eta^2}{\eps\tau} \fou \Ec
= a_{34}(\tau,\eta) \fou \Ec \,,\\
 \frac{d\fou \Ec }{dx_1} & = -\eps\tau \fou \Hc_2 = a_{43}(\tau) \fou \Hc_2 \,.
\end{align*}
Both $a_{34}, a_{43}$, defined in \eqref{def.a34case0}, are finite and different from zero.

Consequently, in \eqref{eq:caso0}, we have
\begin{align}\label{symbolic-equ}
 \d1 \fou{V}^\nc = \Ac(\tau,\eta)\,\fou{V}^\nc\,, \qquad \Ac \doteq \begin{pmatrix}
0 & a_{12} & 0 & 0 \\
a_{21} & 0 & 0 & 0 \\
0 & 0 & 0 & a_{34} \\
0 & 0 & a_{43} & 0
\end{pmatrix}, \qquad x_1>0\, .
\end{align}
\begin{Rem}\label{rk.poles}
We notice that \eqref{eq:caso1} ($\mu=0$) corresponds to a pole for the coefficient $a_{12}$, \eqref{eq:caso2} (yielding $\base\rho(1+\rhop \base H^2)\mu^2+\eta^2\base H^2=0$) corresponds to a pole for the coefficient $a_{21}$, and \eqref{eq:caso3} ($\tau=0$) corresponds to a pole for the coefficient $a_{34}$. For this reason such cases need a different analysis, see Subsections \ref{polomu}, \ref{polocaso2}, \ref{polotau}.
\end{Rem}

The eigenvalues of $\Ac$ are given by the complex roots of
\begin{equation}
\begin{array}{ll}\label{omega1}
\ds \omega^2=a_{12}\,a_{21}=\frac{(\mu^2\base\rho+\eta^2\base H^2)(\rhop \base\rho
\mu^2+\eta^2)}{(\mu^2\base\rho\rhop+\eta^2)\base H^2+\mu^2\base\rho} = \eta^2 +
\frac{\mu^4\rhop\base\rho^2}{(\mu^2\base\rho\rhop+\eta^2)\base H^2+\mu^2\base\rho} 
\end{array}
\end{equation}
and
\begin{align} \omega^2=a_{34}\,a_{43}=\eps^2\tau^2+\eta^2 = \eps^2(\gamma^2-\delta^2)+\eta^2 +2i\eps^2\delta\gamma \,. \label{omega2^2}\end{align}
The classical results of Hersh \cite{hersh} for hyperbolic systems, generalized by Majda--Osher \cite{majda-osher} for free boundary problems with characteristic boundary of constant multiplicity (which is the case we are considering, since we have assumed a piecewise constant basic state), imply that there are two couples  $\pm \omega_1, \pm\omega_2$ of eigenvalues with nonvanishing real part as long as $\gamma=\Re\tau>0$.

 Notice that this means that $\omega_j^2$ is not a negative real number. For the vacuum part, this can be verified with ease, since if~$\delta\neq0$, then $\Im a_{34}\,a_{43}=2\eps^2\delta\gamma\neq0$, while if~$\delta=0$, we have $a_{34}\,a_{43}=\eps^2\gamma^2+\eta^2>0$.
%Per quanto riguarda $a_{12}a_{21}$, si dimostra che non \`e reale negativo su~$\Sigma$ facendo quattro
%paginate di conti (vedi appunti di Davide).]

 We denote by $\omega_1$ and $\omega_2$ the eigenvalues with strictly positive real part for $\gamma>0$, respectively corresponding to the plasma and vacuum blocks.  We denote by $\chi$ the complex root of
\[ \frac{(\mu^2\base\rho\rhop+\eta^2)\base H^2+\mu^2\base\rho}{\mu^2\base\rho+\eta^2\base H^2} \]
 with positive real part or, if the real part is zero, with positive imaginary part. Notice that $\chi\neq 0$, otherwise we are in \eqref{eq:caso2}, as one can easily verify.

 Now, we consider $-\omega_1, -\omega_2$, which have strictly negative real part, and choose the corresponding eigenvectors
%\footnote{A possible alternative is
%%
%\begin{align*}
%e_1
%	& = \chi (a_{12},-\omega_1,0,0)\tm\,,\\
%e_2
%	& = (0,0,\omega_2/\tau,\eps)\tm\,.
%\end{align*}
%%
%With this choice, the eigenvalues can not be extended with continuity to the two poles $\mu=0$ and~$\tau=0$, but the  Lopatinskii condition would hold at~$\mu=0$ and~$\tau=0$ (i.e. the determinant would admit a nonvanishing extension).
%}
%
\begin{align*}
e_1
	& = \mu\chi (a_{12},-\omega_1,0,0)\tm\,,\\
e_2
	& = (0,0,\omega_2,\eps\tau)\tm
\end{align*}
as generators of the stable subspace $E^-(\tau,\eta)$ of~$\Ac(\tau,\eta)$.
Notice that
\[ |\mu\chi a_{12}|= |(\mu^2\base\rho\rhop+\eta^2)\base
H^2+\mu^2\base\rho|^{\frac12}\,|\mu^2\base\rho+\eta^2\base H^2|^{\frac12}\,, \qquad |\mu\chi \omega_1| =
|\mu|\,|\rhop \base\rho \mu^2+\eta^2|^{\frac12}\,, \]
so that~$e_1$ is finite. Moreover, $e_1$ never vanishes. Actually, if
$\mu=0$, then~$|\mu\chi a_{12}|=\eta^2\base H^2\neq 0$ because of \eqref{eq:Hrbase}(otherwise $\gamma=\delta=\eta=0$ and we are not in $\Sigma$). On the other hand, if~$\eta^2=-\rhop \base\rho \mu^2$, then
\[ |\mu\chi a_{12}|= |\mu|^2\base\rho \,|1-\rhop\base H^2|^{\frac12}\neq 0 \]
thanks to \eqref{eq:notnull}. Even more easily, we have that also $e_2$ is finite and different from zero in every point, because $|\omega_2|=|\eps^2\tau^2+\eta^2|^{\frac12}$.

\begin{Rem} \label{omega=0}
We conclude this section by noticing that $\omega_1=0$ in each of the following cases:
\begin{gather}\label{eq:null1}
\gamma=0 \quad \text{and} \quad \delta = -\base v \eta \pm \eta \base H/\sqrt{\base\rho}, \qquad \text{i.e.} \quad a_{12}=0\,, \\
\label{eq:null2}
\gamma=0 \quad \text{and} \qquad \delta = -\base v \eta \pm \eta/\sqrt{\rhop\base\rho}, \qquad
\text{i.e.} \quad \mu\neq0, a_{21}=0\,,
\end{gather}
while ~$\omega_2=0$ when
\begin{gather}
\label{eq:null3}
\gamma=0 \quad \text{and} \quad \delta = \pm\eta/\eps, \qquad \text{i.e.} \quad a_{34}=0\,.
 \end{gather}
{By virtue of~\eqref{eq:notnull}, conditions~\eqref{eq:null1} and~\eqref{eq:null2} cannot be verified at the same time, i.e. we never have~$a_{12}=a_{21}=0$. Also, for sufficiently small~$\eps$ we may exclude that~\eqref{eq:null3} is verified at the same time of one of the previous two conditions.}
\end{Rem}

\section{The Lopatinski\u{\i} determinant} \label{sec.lopatinskii}

Following Majda--Osher \cite{majda-osher}, we define the Kreiss--Lopatinski\u{\i} determinant, for brevity Lopatinski\u\i\ determinant, associated with the boundary condition~$\beta$ as
\begin{align*}
\triangle(\tau,\eta)
    & \doteq \det [\beta (e_1, e_2)] =  \det \left[ \begin{pmatrix}
1 & 0 & -\base\Hc & 0 \\
0 & \eps\base\Hc\tau & 0 & \mu
\end{pmatrix} \ \begin{pmatrix}
\mu\chi a_{12} & 0 \\
-\mu\chi \omega_1 & 0 \\
0 & \omega_2 \\
0 & \eps\tau
\end{pmatrix} \right] \\
&= \det \begin{pmatrix}
\mu\chi a_{12} & -\omega_2\base\Hc \\
-\mu\eps\tau\chi \omega_1\base\Hc & \mu\eps\tau
\end{pmatrix} = \mu\eps\tau (\mu\chi a_{12} - \chi\omega_1\omega_2\base\Hc^2).
\end{align*}
The Lopatinski\u{\i} determinant is a continuous function defined over~$\Sigma$. To derive our energy estimate, we shall study the zeros of~$\triangle(\tau,\eta)$. An immediate consequence of the above formula is in the following lemma.
\begin{Lem}
The Lopatinski\u{\i} determinant $\triangle(\tau,\eta)$ vanishes on ~$\Sigma$ for
\begin{itemize}
\item $\mu=0$, that is $(\gamma,\delta,\eta)=(1+\dot v^2)^{-1}(0,-\dot v,1)$,
\item $\tau=0$, that is $(\gamma,\delta,\eta)=(0,0,1)$,
\item and at  the possible zeros of
\[ \mu\chi a_{12} - \chi\omega_1\omega_2\base\Hc^2=0\,. \]
\end{itemize}
\end{Lem}
\begin{Rem}
The first two cases will be studied apart in Subsections \ref{polotau} and \ref{polomu}, so we exclude them now. As for the last case,
we remark that if~$\chi=0$, i.e. we assume~\eqref{eq:caso2}, then~$\mu\chi a_{12}=0$, but~$(\chi\omega_1)\omega_2\neq0$, so that we do not find zeros of~$\triangle(\tau,\eta)=0$. Therefore, we may conclude that $\chi\neq 0$ and reduce to study $\mu a_{12} - \omega_1\omega_2\base\Hc^2=0$, i.e. the zeros of
\begin{equation}\label{eq:Lopa0}
-\mu^2\base\rho-\eta^2\base H^2 = \omega_1 \omega_2 \base \Hc^2\, .
\end{equation}
\end{Rem}
We have the following.
\begin{Lem} \label{lemma:roots}
Let us assume the conditions {\eqref{eq:Hvbase}--\eqref{eq:notvH}} on the basic state. Then there exists a sufficiently small~$\eps_0>0$ such that, for any~$\eps\in(0,\eps_0)$, we may distinguish two cases.
\begin{itemize}
\item If~$\base H^2<\min\{\rhop^{-1}, \base\rho\base v^2\}$, then there exist exactly two roots $(\tau,\eta)$ of~\eqref{eq:Lopa0}; moreover, they  lie at the boundary of~$\Sigma$ (i.e., $\gamma=0$), and they do not coincide with the zeros of~$\omega_1$ or~$\omega_2$, nor with the poles in~\eqref{eq:caso1}, \eqref{eq:caso2}; moreover, if
    \begin{equation}\label{eq:nottau}
    \base\rho\base v^2\neq\frac{\rhop (\base H^4 - \base \Hc^4) + 2\base H^2 \pm \sqrt{\rhop^2(\base H^4-\base \Hc^4)^2+4\base \Hc^4}}{2(1+\rhop \base H^2)},
    \end{equation}
    they do not coincide with the poles in~\eqref{eq:caso3}; this means that $\mu=0, \tau=0$ are simple roots.
\item If~$\base H^2>\min\{\rhop^{-1}, \base\rho\base v^2\}$, then there exist no roots of~\eqref{eq:Lopa0}.
\end{itemize}
We emphasize that~$\eps_0\to0$ as~$\rhop\base H^2\to1$.
\end{Lem}
\begin{Rem}
%	Notice that we have \eqref{eq:nottau} from
%	\[
%	(1-\rhop\base\rho\base v^2)\base \Hc^4 \neq (1-\rhop\base\rho\base v^2)\base H^4-\base\rho\base v^2(2-\rhop\base\rho\base v^2)\base H^2+\base\rho^2\base v^4\,,
%	\]
%	and this condition is guaranteed if $|\base\Hc|$ is sufficiently large, i.e. larger than a suitable positive $\Hc^*=\Hc^*(\base\rho,\rhop,\base v,\base H)$. This explains the conditions in Theorem~\ref{Thm:main0} and Proposition~\ref{Thm:main}.
%	
	{We recall that $\rhop\base H^2\neq 1$ and $\base\rho\base v^2\neq \base H^2$ thanks to \eqref{eq:notnull} and \eqref{eq:notalfven}, so that %it is always
		$\base H^2\neq\min\{\rhop^{-1}, \base\rho\base v^2\}$.}
\end{Rem}

\begin{proof}
Multiplying the equation $-\mu^2\base\rho-\eta^2\base H^2 - \omega_1 \omega_2 \base \Hc^2=0$ by the expression
\begin{align} \label{eq:Lopa2}
-\mu^2\base\rho-\eta^2\base H^2 + \omega_1 \omega_2 \base \Hc^2\, ,
\end{align}
we may replace the explicit expression of~$\omega_1^2\omega_2^2$, obtaining:
\begin{equation}\label{eq:Lopa12}
\mu^2\base\rho + \eta^2\base H^2 = (\eps^2\tau^2+\eta^2)
\frac{\rhop\base\rho\mu^2+\eta^2}{(\rhop\base\rho\mu^2+\eta^2)\base H^2+\mu^2\base\rho} \base\Hc^4\,.
\end{equation}
Let us show that for $\eps$ small enough there is no root of the above equation if~$\eta=0$. Indeed, \eqref{eq:Lopa12} would reduce to the relation
\[ \base\rho = \eps^2 \frac{\rhop}{\rhop\base H^2+1} \base\Hc^4, \]
which can be excluded for sufficiently small~$\eps$. Therefore, we may assume~$\eta\neq0$; setting $V:=\frac{\tau}{i\eta}+\base v$, (it yields $\mu=i\eta V$), we may write~\eqref{eq:Lopa12} in the form
\[
(\base H^2-\base\rho V^2) \{\base H^2-\base\rho(1+\rhop\base H^2) V^2 \} = (1-\base\rho\rhop V^2) \base \Hc^4
\{ 1 -\eps^2(V-\base v)^2 \},
\]
i.e. we obtain a fourth-order equation in $V$:
\begin{multline}\label{eq:LopaV}
\{ \base\rho^2(1+\rhop\base H^2) -\eps^2\base\rho\rhop\base\Hc^4 \} V^4 +2\eps^2\base\rho\rhop\base v \base \Hc^4 V^3 + \{
(\eps^2+\base\rho\rhop-\eps^2\base\rho\rhop\base v^2)\base \Hc^4 - \base\rho\base H^2 (2+\rhop\base H^2) \} V^2 \\
-2\eps^2\base v \base\Hc^4 V + \{ \base H^4 - (1-\eps^2\base v^2) \base \Hc^4 \} = 0\, .
\end{multline}
We may prove that, for sufficiently small~$\eps$, \eqref{eq:LopaV} admits four distinct roots.
\medskip

(I) We first consider the case~$\base H^2\neq \base \Hc^2$. For~$\eps=0$, \eqref{eq:LopaV} reduces to the biquadratic equation
\begin{equation}\label{eq:biq}
\base\rho^2(1+\rhop\base H^2) \bar V^4 + (\base\rho\rhop\base \Hc^4 - \base\rho\base H^2 (2+\rhop\base H^2) ) \bar V^2 + (\base H^4 - \base \Hc^4) = 0\,.
\end{equation}
%
%whose discriminant is %\footnote{Per una quartica nella forma
%%
%\[ ax^4 + bx^2 + c=0 \]
%%
%il discriminante \`e dato da~$ac (b^2-4ac)^2$. Se \`e strettamente negativo, ci sono due radici reali distinte e due complesse coniugate (sottintendiamo non reali). Se \`e strettamente positivo, le radici sono tutte reali distinte se~$b^2-4ac>0$ e~$ab<0$, altrimenti sono complesse, coniugate a due a due.}
%
%\[ \base\rho^4(1+\rhop\base H^2) (\base H^4 - \base \Hc^4) \bigl(\rhop^2 (\base H^4-\base \Hc^4)^2+4\Hc^4 \bigr)^2. \]
%
 Setting $z=\base\rho \bar V^2$ (in particular, $\base\rho\mu^2=-z\eta^2$), we may rewrite~\eqref{eq:biq} as a second degree equation in~$z$:
\begin{equation}\label{eq:Lopaz}
z^2 (1+\rhop \base H^2) - z (\rhop \base H^4 - \rhop \base \Hc^4 + 2\base H^2) + \base H^4-\base \Hc^4 =0\,,
\end{equation}
%
%whose discriminant is given by:
%%
%\begin{align*} D \doteq & (\rhop \base H^4 - \rhop \base \Hc^4 + 2\base H^2)^2 -4 (1+\rhop \base H^2) (\base H^4-\base \Hc^4)\\
%= & \rhop^2(\base H^4-\base \Hc^4)^2+4\base \Hc^4>0\, .
%\end{align*}
%_
whose solutions are given by
\begin{align} \label{zpm}
z_\pm = \frac{\rhop (\base H^4 - \base \Hc^4) + 2\base H^2 \pm \sqrt{D}}{2(1+\rhop \base H^2)}\,,
\end{align}
where
\begin{align*} D \doteq & (\rhop \base H^4 - \rhop \base \Hc^4 + 2\base H^2)^2 -4 (1+\rhop \base H^2) (\base H^4-\base \Hc^4)\\
= & \rhop^2(\base H^4-\base \Hc^4)^2+4\base \Hc^4>0\, .
\end{align*}
It is clear that~$z_+>0$, so that $\bar V=\pm \sqrt{\base\rho^{-1}z_+}$ are two real, distinct roots of~\eqref{eq:biq}. On the other hand, $z_->0$ if~$\base H^2>\base\Hc^2$ and $z_-<0$ if~$\base H^2<\base\Hc^2$. In the first case, we find two more real, distinct roots~$\bar V=\pm \sqrt{\base\rho^{-1}z_-}$ of~\eqref{eq:biq}, in the second one we find two complex-valued (purely imaginary) conjugated roots of~\eqref{eq:biq}.

From the previous results we thus have two cases:
\begin{itemize}
	\item if~$\base H^2>\base \Hc^2$, then the four roots of~\eqref{eq:biq} are real and distinct;
	\item if~$\base H^2<\base \Hc^2$, then~\eqref{eq:biq} admits two distinct real roots and two complex-valued, conjugated roots.
\end{itemize}
The same result can be obtained by means of the general theory of quartic equations.
%%%%
%%%%
%%%%
We recall
that the type, real or complex, of solutions of a generic quartic equation
\begin{equation}\label{quartica}
 a_4 x^4 + a_3 x^3 + a_2 x^2 + a_1 x + a_0 =0\,
\end{equation}
depends on the sign of the discriminant $\Delta$ (whose long expression is of no interest here) and of the quantities $P, Q$, defined by
\[ P=8a_4a_2-3a_3^2, \qquad Q=64a_4^3a_0-16a_4^2a_2^2+16a_4a_3^2a_2-16a_4^2a_3a_1-3a_3^4\,. \]
Equation \eqref{eq:biq} has the form \eqref{quartica} with $a_3=a_1=0$, which yields the simple formul\ae
\[
\Delta = 16a_4a_0(4a_4a_0-a_2^2)^2, \quad P=8a_4a_2, \quad Q=16a_4^2(4a_4a_0-a_2^2).
\]
From the algebraic theory of quartic equations, the previous result for $\eps=0$ is easily obtained:
\begin{itemize}
	\item if~$\base H^2>\base \Hc^2$, then $\Delta>0, P<0, Q<0$ and the four roots of~\eqref{eq:biq} are real and distinct;
	\item if~$\base H^2<\base \Hc^2$, then $\Delta<0$ and \eqref{eq:biq} admits two distinct real roots and two complex-valued, conjugated roots.
	\end{itemize}

Now we consider the general case \eqref{eq:LopaV} with $\eps\not=0$ and notice that the coefficients of \eqref{eq:LopaV} are continuous functions of $\eps$. Hence, also  $\Delta,P,Q$ are continuous functions of $\eps$.

Being non-zero at~$\eps=0$, the discriminant of~\eqref{eq:LopaV} remains non-zero and with the same sign of the discriminant of \eqref{eq:biq} for sufficiently small~$\eps$. The same holds true also for the quantities $P, Q$, so that we still have the same cases for the complete equation \eqref{eq:LopaV}:
\begin{itemize}
\item if~$\base H^2>\base \Hc^2$, then the four roots of~\eqref{eq:LopaV} are real and distinct;
\item if~$\base H^2<\base \Hc^2$, then~\eqref{eq:LopaV} admits two distinct real roots and two complex-valued, conjugated roots.
\end{itemize}
Once we fix the four distinct roots~$\bar V_i$ of~\eqref{eq:biq}, for sufficiently small~$\eps_0>0$, there exist four separate neighborhood~$\mathcal V_i$ of~$\bar V_i$ such that for any~$\eps\in(0,\eps_0)$, \eqref{eq:LopaV} admits four roots~$V_i(\eps)$, continuously depending on~$\eps$, with~$V_i(\eps)\in \mathcal V_i$. In this sense, we may say that each root of~\eqref{eq:LopaV} corresponds to a root of~\eqref{eq:biq}.

Among the four roots of~\eqref{eq:LopaV}, we may look for the solutions to~\eqref{eq:Lopa0}. 

First of all, let us see that the two complex conjugated roots of case $\base H^2<\base \Hc^2$ do not give a solution to the equation \eqref{eq:Lopa0}, for small enough $\eps$.
Thus, let us
assume~$\base H^2<\base \Hc^2$ and let~$V$ be a complex-valued root of~\eqref{eq:LopaV} with~$\gamma=-\eta\Im V>0$ (this corresponds to a point in the interior of~$\Sigma$). We claim that such a $V$ (corresponding to a complex solution $\bar V$ of \eqref{eq:biq}), does not provide a solution of~\eqref{eq:Lopa0}, for sufficiently small~$\eps$. In fact, for sufficiently small~$\eps$, by continuity it holds~$\Re(\omega_1\omega_2)>0$, due to
\[ \Re(\omega_1\omega_2)=|\eta|\Re\omega_1>0,\qquad \text{for~$\eps=0$} \]
(in this case $\omega_1\neq 0$, recalling Remark~\ref{omega=0}, and $\eta\neq 0$, as we have already observed).
Therefore, from~\eqref{eq:Lopa0} we derive that~$\base\rho\Re (\mu^2)+\eta^2\base H^2<0$; hence $-\base\rho \Re(V^2)+\base H^2<0$, which is absurd, since $-\base\rho \Re(V^2)$ is close to $-\base\rho \Re(\bar V^2)=-\base\rho \bar V^2=-z_->0$ for small $\eps$.

Therefore, for sufficiently small~$\eps$ any possible solution $V$ of~\eqref{eq:Lopa0} on~$\Sigma$ is real, i.e. it verifies~$\gamma=0$. Due to~$\eta\neq0$, for sufficiently small~$\eps$ we see that~$\omega_2=\sqrt{\eta^2-\eps^2\delta^2}>0$. Then, by~\eqref{eq:Lopa0} it follows that~$\omega_1$ is real as it happens for $-\mu^2\base\rho-\eta^2\base H^2=\eta^2(\base\rho V^2-\base H^2)$; hence, $\omega_1>0$, as a consequence of~$\omega_1\neq0$ (we will show below that $\omega_1\neq 0$ when the Lopatinski\u{\i} determinant vanishes) and~$\Re\omega_1\geq0$. It follows that
\begin{equation}\label{cond.ammiss}
\base\rho V^2>\base H^2
\end{equation}
for small~$\eps$. This implies that we find a couple of real roots to~\eqref{eq:LopaV} if, and only if, setting~$\eps=0$ we get~$z_+>\base H^2$. Indeed, $z_-<\base H^2$, therefore the two corresponding roots are excluded (thus $z_-$ is excluded in both cases $\base H^2 \gtrless \base \Hc^2$). The condition~$z_+>\base H^2$ corresponds to~$\sqrt{D}>\rhop (\base H^4+\base \Hc^4)$, i.e.
\[ \rhop\base H^2<1. \]
It remains to prove that if $z_+>\base H^2$, then we effectively find two real roots of~\eqref{eq:Lopa0}, in suitable neighborhoods of~$\pm \sqrt{\base\rho^{-1}z_+}$, for sufficiently small~$\eps$. In order to do that, it is sufficient to prove that they are not real roots of~\eqref{eq:Lopa2}. But this follows as an immediate consequence of~$\base\rho V^2>\base H^2$: being~$\omega_1\omega_2>0$, the quantity in~\eqref{eq:Lopa2} is positive, in particular it is not zero.

%Pensando $\eta$ come un parametro fissato, abbiamo un'equazione di quarto grado in $\tau$, le cui radici sono radici di \eqref{eq:Lopa0} o \eqref{eq:Lopa2}. Se penso a tutto $\C$, si tratta di quattro radici. Se $\eps$ \`e piccolo, per quanto visto, le due soluzioni corrispondenti a $z_-$ non sono radici di \eqref{eq:Lopa0}, quindi sono radici di \eqref{eq:Lopa2}. Le due soluzioni corrispondenti a $z_+$, che sono immaginarie pure ($V$ reale, cio\`e $\gamma=0$), devono coincidere con zeri di \eqref{eq:Lopa0} o \eqref{eq:Lopa2}. Come gi\`a osservato, se la condizione $z_+>\base H^2$ \`e violata, non sono zeri di \eqref{eq:Lopa0}. Ma se \`e rispettata, possiamo escludere che siano zeri di \eqref{eq:Lopa2} come segue, e dedurre che sono zeri di \eqref{eq:Lopa0}. Infatti, se $z_+>\base H^2$, cio\`e $\base\rho V^2>\base H^2$ (cio\`e $\rhop\base H^2>1$; il caso $\rhop\base H^2=1$ \`e gi\`a stato escluso), usando $\gamma=0$, la definizione di $\mu$ e il fatto, gi\`a osservato, che $\Re(\omega_1\omega_2)>0$, dall'equazione \eqref{eq:Lopa2} divisa per $\eta^2$ ricaviamo
%\begin{align*}
%0<\base\rho V^2 -\base H^2 = -\frac{\base \Hc^2}{\eta^2} \Re(\omega_1\omega_2)<0\, ,
%\end{align*}
%che \`e assurdo. Ci sono quindi degli zeri con $\gamma=0$.

\medskip

(II)
Now let~$\base\Hc^2=\base H^2$; the equation in~\eqref{eq:LopaV} has the coefficient of zero degree proportional to~$\eps^2$:
\begin{multline}\label{eq:LopaVsp}
\{ \base\rho^2(1+\rhop\base H^2) -\eps^2\base\rho\rhop\base H^4 \} V^4 +2\eps^2\base\rho\rhop\base v \base H^4 V^3 + \{
\eps^2(1-\base\rho\rhop\base v^2)\base H^4 - 2\base\rho\base H^2  \} V^2 \\
-2\eps^2\base v \base H^4 V + \eps^2\base v^2 \base H^4 = 0\, ,
\end{multline}
in particular the discriminant vanishes for~$\eps=0$. However, it is easy to check that its discriminant verifies
%
%, utilizzando l'espressione esplicita del discriminante di una equazione di quarto grado, si vede facilmente che:\footnote{Se scrivo la quartica come
%%
%\[ \]
%%
%l'unico termine del discriminante che contiene al pi\`u un monomio di grado $1$ in~$b$, $d$ o $e$, che sono i coefficienti che nel nostro caso si annullano con velocit\`a $\eps^2$, \`e ~$16ac^4e$.}
%
\[ \Delta = 16 \eps^2 \base\rho^2(1+\rhop\base H^2) (2\base\rho\base H^2)^4 \base v^2 \base H^4 + \text{O}(\eps^4), \]
in particular, $\Delta>0$ for sufficiently small, nonzero $\eps$. Then we have four distinct roots. Moreover, they are all real-valued for small, nonzero, $\eps$, due to: %\footnote{Nel caso della quartica di cui sopra, le radici sono reali se~$<0$ e $D=64a^3e-16a^2c^2+16ab^2c-16a^2bd-3b^4<0$.}
\[ P=-16\base\rho^2(1+\rhop\base H^2)\base\rho\base H^2 + \text{O}(\eps^2), \qquad Q=-16(\base\rho^2(1+\rhop\base H^2))^2(2\base\rho\base H^2)^2 + \text{O}(\eps^2), \]
so that~$P<0$ and~$Q<0$ for sufficiently small~$\eps$ (see the discussion above).
These four roots will be close to the four real roots of~\eqref{eq:LopaV} at~$\eps=0$, i.e.
\[ \base\rho^2(1+\rhop\base H^2) V^4 - 2\base\rho\base H^2 V^2 =0. \]
In particular, two of them will be close to~$0$, thus they do not satisfy~\eqref{cond.ammiss}, whereas the other two are close to~$\pm \sqrt{z_+/\base\rho}$, where we put
\begin{align} \label{zpp}
z_+ = \frac{2\base H^2}{1+\rhop\base H^2}.
\end{align}
Clearly, $z_+$ in~\eqref{zpp} coincides with~$z_+$ in~\eqref{zpm}, setting~$\base \Hc^2=\base H^2$. Again, $z_+>\base H^2$ if, and only if, $\rhop\base H^2<1$, as it happens for~$\base H^2\neq \base \Hc^2$.

Summarizing, if~$\rhop\base H^2>1$, then, for sufficiently small~$\eps$, \eqref{eq:Lopa0} admits no solution in~$\Sigma$; on the other hand, if~$\rhop\base H^2<1$, then, for sufficiently small~$\eps$, \eqref{eq:Lopa0} admits exactly two solutions, and they are on the boundary of~$\Sigma$.

It remains to prove that, in the first case, the two roots do not coincide with the poles, i.e.
\[ \mu=0,\qquad \mu=\pm\,i\,\frac{\base H}{\sqrt{\base\rho(1+\alpha\base H^2)}}, \qquad \tau=0\,,\]
or with the zeros of~$\omega_1, \omega_2$.

The case~$\mu=0$ corresponds to~$V=0$, thus~\eqref{cond.ammiss} is trivially false. The second case gives
\[ V=\pm\frac{\base H}{\sqrt{\base\rho(1+\alpha\base H^2)}}; \quad \text{hence,}\quad \base\rho V^2= \frac{\base H^2}{1+\alpha\base H^2}<\base H^2,\]
and~\eqref{cond.ammiss} is violated once again. The case~$\tau=0$, i.e. $\eta=1$, implies $\mu=i\base v $, so that
\[ \base\rho V^2= \base\rho\base v^2, \]
and~\eqref{cond.ammiss} might be satisfied. However, setting~$\eps=0$ and replacing~$z_+=\base\rho\base v^2$ in~\eqref{zpm}, we derive
\begin{equation} \base\rho\base v^2=\frac{\rhop (\base H^4 - \base \Hc^4) + 2\base H^2 \pm \sqrt{D}}{2(1+\rhop \base H^2)}\,. \label{eq:cond}\end{equation}
Therefore, if~\eqref{eq:nottau} holds, the two roots are not given by~$\pm(0,0,1)$ for sufficiently small~$\eps$. Notice that it is not necessary to prevent \eqref{eq:cond}  if $\rhop\base H^2>1$ or $\base H^2>\base\rho\base v^2$, that is, if $\base H^2>\min\{\rhop^{-1}, \base\rho\base v^2\}$.

If~\eqref{eq:null1} holds, then $V=\pm \base H/\sqrt{\base\rho}$, i.e. $\base\rho V^2=\base H^2$, so that~\eqref{cond.ammiss} is violated. If~\eqref{eq:null2} holds, then $V=\pm(\rhop \base\rho)^{-\frac12}$. At~$\eps=0$, we obtain $z=\rhop^{-1}$ which, replaced in \eqref{zpm}, gives $\rhop \base H^2=1$, which contradicts~\eqref{eq:notnull}. Due to~\eqref{eq:notnull}, we may exclude that the roots satisfy~\eqref{eq:null2} for sufficiently small~$\eps$. %Ad esempio, sostituendo in \eqref{zpp} per $\eps$ piccolo, abbiamo $1-\rhop\base H^2=c_\eps\rhop(1+\rhop \base H^2)$, dove $c_\eps\to 0$ quando $\eps\to 0$, quindi l'uguaglianza non pu\`o valere per piccoli $\eps$. Analogamente sostituendo in \eqref{zpm}.
If~\eqref{eq:null3} holds, then $V=\base v \pm \eps^{-1}$, in particular~$V^2\to\infty$ as~$\eps\to0$, and this contradicts the fact that the coefficients in~\eqref{eq:LopaV} are bounded for~$\eps\in[0,1]$. %ha modulo grande per $\eps$ piccoli, mentre le radici di Lopatinskii, quando esistono, sono prossime ai valori corrispondenti a~$z_+$, che non dipende da $\eps$, quindi i punti non coincidono.
\end{proof}

\section{Constructing a symmetrizer}\label{symmetrizer}

\subsection{Classification of the points on $\Sigma$}\label{Classification}

We now turn to the construction of our degenerate Kreiss'
symmetrizer, proceeding as in \cite{nappes}. The construction is microlocal and is achieved near
any point $(\tau_0,\eta_0) \in \Sigma$. The analysis is rather
long since one has to distinguish between many different cases.
In the end, we shall consider a partition of unity to patch
things together and derive our energy estimate.

 We classify the points as follows and construct a symmetrizer in a suitable neighborhood of any point of~$\Sigma$, with the exception of poles, that will be treated in a different way.
\begin{enumerate}
%\item Poli: $\mu =0$, cio\`e $\gamma=0$ e $\delta+\base v\eta=0$. Lopatinskii non vale.\footnote{A meno di cambiare la base di autovettori}
%\item Poli: $\gamma =0$ e $\delta+\base v\eta=\pm\frac{\base H}{\sqrt{\base\rho(1+\rhop\base H^2})} \eta$. Lopatinskii vale.
%\item Poli: $\tau=0$, cio\`e $\gamma=\delta=0$. Lopatinskii non vale.\footnote{A meno di cambiare la base di autovettori}
\item Interior points,  where $\gamma>0$ and the Lopatinski\u{\i} condition holds, i.e. $\triangle\neq 0$ (Subsection~\ref{ssec:interior}). $\Ac$ is diagonalizable.
\item Boundary points, where $\gamma=0$, which are not poles and such that~$\omega_1\omega_2\neq0$ and the Lopatinski\u{\i} condition holds (Subsection~\ref{ssec:boundaryLop}). $\Ac$ is diagonalizable.
\item Boundary points, where $\gamma=0$, which are not poles and such that~$\omega_1\omega_2\neq0$ and the Lopatinski\u{\i} condition does not hold, i.e. $\triangle=0$ (Subsection~\ref{ssec:boundarynoLop}). $\Ac$ is diagonalizable.
\item Points where $\omega_1=0$, i.e. $\gamma=0$ and either $\delta+\base v \eta = \pm \base H\eta/\sqrt{\base\rho}$, or $\delta+\base v \eta = \pm \eta/\sqrt{\rhop\base\rho}$ (Subsection~\ref{ssec:nodiag}). The points are not poles and $\Ac$ is not diagonalizable: the block in the Jordan matrix corresponding to $\omega_1$ is given by $\begin{pmatrix} 0 & 1 \\ 0 & 0 \end{pmatrix}$. The Lopatinski\u{\i} condition holds.
\item Points where $\omega_2=0$, i.e. $\gamma=0$ and $\eta=\pm\eps\delta$ (Subsection~\ref{ssec:nodiag}). The points are not poles and $\Ac$ is not diagonalizable: the block in the Jordan matrix corresponding to $\omega_2$ is given by $\begin{pmatrix} 0 & 0 \\ 1 & 0 \end{pmatrix}$. The Lopatinski\u{\i} condition holds. 
\item Poles (CASE 1, CASE 2, CASE 3, see respectively Subsections~\ref{polotau}, \ref{polomu}, \ref{polocaso2})
\[ \tau=0\,, \qquad \mu=0\,, \qquad \mu= \pm\,i\,\frac{\base H}{\sqrt{\base\rho(1+\rhop\base H^2})} \eta. \]
The Lopatinski\u{\i} condition does not hold in poles $ \tau=0$ and $\mu=0$, while it is satisfied in the third case.
\end{enumerate}
We recall that we assumed~$\eps$ to be sufficiently small. Moreover,  in the following, we will denote by~$\kappa$ a generic, positive, constant.

We begin with some useful results.

\subsection{Preliminary lemmas}

\begin{Lem}\label{lem:gamma0}
Let~$\gamma=0$. Then~$a_{12},a_{21} \in i\R$, $\p_\gamma a_{12}, \p_\gamma a_{21}\in \R$ and $\p_\gamma \omega_1^2 \in i\R$.
\end{Lem}
\begin{proof}
In the following we put
\begin{equation}\label{eq:immu}
\delta+\base v \eta=\pm\sqrt{K}\eta,
\end{equation}
for some~$K\geq0$ if~$\eta\neq0$. If~$\eta=0$, so that~$\delta\neq0$, we formally set~$K=\infty$ and we replace the undetermined object~$\pm\sqrt{K}\eta$ by~$\delta$, where needed.
Due to~$\mu^2=-K\eta^2$, we get:
\begin{align*}
a_{12}
    & = \pm\frac{K\base\rho-\base H^2}{i\sqrt{K}}\,\eta\,, \\
a_{21}
    & = -\frac{\pm i\sqrt{K}(1-\rhop\base\rho K)}{(1-\rhop\base\rho K)\base H^2-\base\rho K}\,\eta\,.
\end{align*}
In view of
\[ \p_\gamma \mu|_{\gamma=0} = \p_\gamma \bar\mu|_{\gamma=0} = 1, \qquad \mu \p_\gamma (\mu^2)|_{\gamma=0} = -2K\eta^2, \]
setting
\[ k=(1-\rhop\base\rho K)\base H^2-\base\rho K, \]
we derive:
\begin{align*}
\p_\gamma a_{12}
    & = -\frac{-K\base\rho+\base H^2}{-K} +2K\, \frac{\base\rho (-K)-(-K\base\rho+\base H^2)}{K^2} = -\frac{K\base\rho+\base H^2}{K} \,,\\
\p_\gamma a_{21}
    & = -\frac{1-\rhop\base\rho K}{k}+2K\,\frac{\rhop\base\rho k-(1-\rhop\base\rho K)\base\rho(\rhop \base H^2+1)}{k^2} \\
    & = \frac{(3\rhop\base\rho K-1)k -2K(1-\rhop\base\rho K)\base\rho(\rhop \base H^2+1)}{k^2} \\
    & = \frac{-\rhop^2\base\rho^2\base H^2 K^2-\rhop\base\rho^2 K^2+2\rhop\base\rho K\base H^2-\base\rho K-\base H^2}{k^2} \\
    & = \frac{-(\rhop\base\rho K-1)^2\base H^2-\rhop\base\rho^2 K^2-\base\rho K}{k^2} \,.
\end{align*}

Finally, we get
\[
\p_\gamma \omega_1^2 = a_{12} \p_\gamma a_{21}+a_{21}\p_\gamma a_{12} \in i\R\,,
\]
which proves our claim.
\end{proof}
From Lemma~\ref{lem:gamma0}, we derive the following.
\begin{Lem}\label{Lem:Regamma}
Let~$\gamma=0$. Then:
\begin{itemize}
\item If we are not in~\eqref{eq:caso2} and~$\omega_1\in i\R\setminus\{0\}$, then~$\p_\gamma\omega_1(\tau,\eta)\in\R\setminus\{0\}$;
\item If~$\omega_2\in i\R\setminus\{0\}$, then~$\p_\gamma\omega_2(\tau,\eta)\in\R\setminus\{0\}$.
\end{itemize}
\end{Lem}
\begin{Rem}\label{remark2}
Being the real part of~$\omega_1$ strictly positive for~$\gamma>0$, we have $\Re\omega_1\geq0$ for all points of~$\Sigma$. There are two possibilities: if in a given point of $\Sigma$ there holds $\Re\omega_1>0$, then in a sufficiently small neighborhood $\Vc$ we have~$\Re\omega_1\geq\kappa$ for some~$\kappa>0$. Because $\gamma\in[0,1]$ on~$\Sigma$, the inequality ~$\Re\omega_1\geq\kappa\gamma$ holds as well. Alternatively, if in a given point of $\Sigma$ (excluding~\eqref{eq:caso2}) there holds $\Re\omega_1=0$, then, by Lemma~\ref{Lem:Regamma}, in a sufficiently small neighborhood $\Vc$ it follows that~$\Re\omega_1\geq\kappa\gamma$.
Similarly, $\Re\omega_2\geq \kappa\gamma$ for any point on~$\Sigma$. See \cite{nappes}.

\end{Rem}

%
%Ovviamente, se~$\Re\omega_1>0$ in~$(\tau,\eta)$, allora~$\Re\omega_1>0$ in un intorno di~$(\tau,\eta)$. Dal Lemma~\ref{Lem:Regamma} segue che se~$\Re\omega_1=0$ e~$\gamma=0$, allora~$\Re\omega_1\geq\kappa\gamma$ per qualche~$\kappa>0$, in un intorno di~$(\tau,\eta)$ (sviluppo $\omega_1$ in $\gamma$ al primo ordine, uso il lemma e $\Re\omega_1>0$ per $\gamma>0$). Allo stesso modo $\Re\omega_2\geq\kappa\gamma$.
%
\begin{proof}%[Prova del Lemma~\ref{Lem:Regamma}]
Let us compute~$\omega_1^2$ when~$\gamma_0=0$. We set~$K$ as in~\eqref{eq:immu}, and
\[ k=(1-\rhop\base\rho K)\base H^2-\base\rho K, \]
which is nonzero thanks to the fact that we excluded~\eqref{eq:caso2}. Then
\[ \omega_1^2(\tau,\eta) = -\frac{(K\base\rho-\base H^2)(1-\rhop\base\rho K)}{k} \eta^2 = \eta^2 \left(1+ \frac{\rhop\base\rho^2 K^2}{k}\right) \,. \]
We notice that~$K$ may not be zero, i.e. we can exclude~\eqref{eq:caso1}, otherwise~$\omega_1^2=\eta^2>0$.

By~$\omega_1^2(\tau,\eta)<0$, we deduce~$k<0$ and $\eta\neq 0$. Moreover,
\begin{equation}\label{eq:compRegamma}
0<-k < \rhop\base\rho^2 K^2.
\end{equation}
From
$2\omega_1\p_\gamma (\omega_1) = \p_\gamma(\omega_1^2)%=a_{12}\p_\gamma(a_{21})+a_{21}\p_\gamma(a_{12}) 
\in i\R$,
thanks to Lemma~\ref{lem:gamma0}, and recalling that~$\omega_1\in i\R$, we derive that~$\p_\gamma(\omega_1)\in \R$. %It is nonzero if:
%
%\[ 0\neq \frac{K\base\rho-\base H^2}{i\sqrt{K}}\,\eta \left(\frac{-(\rhop\base\rho K-1)^2\base H^2-\rhop\base\rho^2 K^2-\base\rho K}{k^2}\right) -\frac{i\sqrt{K}(1-\rhop\base\rho K)}{(1-\rhop\base\rho K)\base H^2-\base\rho K}\,\eta \left(-\frac{K\base\rho+\base H^2}{K}\right)\,. \]
%
Taking the derivative of
\[ \omega_1^2 = \eta^2 + \frac{\mu^4\rhop\base\rho^2}{(\mu^2\base\rho\rhop+\eta^2)\base H^2+\mu^2\base\rho}, \]
replacing~$\gamma=0$ and~$\delta+\base v\eta=\pm\sqrt{K}\eta$, we get
\begin{align*}
\p_\gamma(\omega_1^2) & = \pm 2i\sqrt{K}\eta\,\frac{-2\rhop\base\rho K k-\rhop\base\rho^2 (1+\rhop \base H^2)K^2}{k^2} \\
 & =\pm 2i\sqrt{K}\eta\rhop\base\rho^2 \frac{K}{k^2} [\base \rho (1+\rhop\base H^2)K-2\base H^2]\, .
\end{align*}
%
%
%Se~$\eta=0$, ricordiamo che~$K=\infty$ formalmente e otteniamo subito
%%
%\[ \p_\gamma(\omega_1^2) = -2i\delta\,\frac{\rhop\base\rho^2 (1+\rhop H^2)}{(\rhop\base H^2+1)^2\base\rho^2} \neq0 \,. \]
%%
By contradiction, let $\p_\gamma(\omega_1^2)=0$. Then,
\[ 2\base H^2=\base\rho(1+\rhop\base H^2)K=-2k. \]
Replacing the expressions so obtained for~$k$ and~$K$ in~\eqref{eq:compRegamma}, we get:
\[ \base H^2 =-k<\rhop\base\rho^2 K^2 =  \frac{4\rhop\base H^4}{(1+\rhop\base H^2)^2}, \]
which is false, being it equivalent to~$(1-\rhop\base H^2)^2<0$.

If $\omega_2\in i\R\setminus{0}$, then $\omega_2^2 <0$ and
\[ 2\omega_2 \p_\gamma \omega_2 = \p_\gamma (\omega_2^2) = 2\eps^2(\gamma+i\delta),\]
in particular $\omega_2 \p_\gamma \omega_2=i\eps^2\delta\in i\R$ for~$\gamma=0$. Therefore, $\p_\gamma\omega_2\in\R$. It can not be zero, otherwise~$\tau=0$ and we derive the contradiction $\eta^2=\omega_2^2<0$.
\end{proof}
\begin{Lem}\label{prop:derivgamma}
If~$\gamma=0$, then~$\p_\gamma (\omega_j^2)\in i\R\setminus\{0\}$ for~$j=1,2$.
\end{Lem}
%
%It is clear that from Proposition~\ref{prop:derivgamma} it trivially follows that~$\p_\gamma (\omega_j^4+\omega_j^2) \in i\R\setminus\{0\}$ as well. Moreover, we recall that $\omega_1=0$ or $\omega_2=0$ implies $\gamma=0$.
%
\begin{proof}
From Lemma~\ref{lem:gamma0}, we already have $\p_\gamma (\omega_1^2)\in i\R$.
%We first consider~$\omega_1$. We recall that~$\omega_1=0$ implies that~$\gamma=0$. Let~$K$ be as in~\eqref{eq:immu}. Then
%%
%\[ \p_\gamma(\omega_1^2)|_{\gamma=0} =\pm 2i\sqrt{K}\eta\rhop\base\rho^2 \frac{K}{k^2} [\base \rho (1+\rhop\base H^2)K-2\base H^2] \in i\R. \]
%%
Now, by contradiction, let~$\p_\gamma(\omega_1^2)=0$ when~\eqref{eq:null2} holds, that is, $\rhop\base\rho K=1$. Then:
\[ 0 = \base \rho (1+\rhop\base H^2)K-2\base H^2 = \base\rho K - \base H^2 = \frac1\rhop (1-\rhop\base H^2), \]
which contradicts~\eqref{eq:notnull}. Similarly if~\eqref{eq:null1} holds, that is, $\base\rho K=\base H^2$.

%If~$\omega_2=0$, then~$\gamma=0$ and~$\p_\gamma (\omega_2^2) = 2i\eps^2\delta \in i\R\setminus{0}$.
On the other hand, when $\gamma=0$, we have $\p_\gamma (\omega_2^2) = 2i\eps^2\delta \in i\R\setminus{0}$ by an explicit computation.
\end{proof}
%%%
Thanks to Lemma~\ref{lem:gamma0}, we get another useful result.
\begin{Lem}\label{Lem:notnullp}
We have the following properties:
\begin{itemize}
\item if~$a_{12}=0$, then~$\p_\gamma a_{12},\p_\gamma a_{21}\neq0$,
\item if~$a_{21}=0$, then~$\p_\gamma a_{12},\p_\gamma a_{21}\neq0$,
\item if~$a_{34}=0$, then $\p_\gamma a_{34}\neq0$.
\end{itemize}
\end{Lem}
We remark that $a_{43}=-\eps\tau\neq 0$ if we are not in~\eqref{eq:caso3}, and that~$\p_\gamma a_{43}=-\eps\neq0$. Thanks to~\eqref{eq:notnull}, it follows that~\eqref{eq:null1} and~\eqref{eq:null2} are not verified at the same time.

\begin{proof}
First, we assume~\eqref{eq:null1}. We set~$K$ as in~\eqref{eq:immu}, that is, $K=\base H^2/\base \rho$; hence,
\begin{align*}
\p_\gamma a_{12}|_{\base \rho K=\base H^2}
    & = -2\base\rho \neq0, \\
\p_\gamma a_{21}|_{\base \rho K=\base H^2}
    & = \frac{-(\rhop \base H^2-1)^2\base H^2-\rhop\base H^4-\base H^2}{\rhop^2\base H^8} = - \frac{\rhop^2 \base H^4-\rhop\base H^2+2}{\rhop^2\base H^6}= - \frac{(\rhop\base H^2-1)^2+\rhop\base H^2+1}{\rhop^2\base H^6} \neq0\,.
\end{align*}
On the other hand, if we assume~\eqref{eq:null2}, then~$K=1/(\rhop\base\rho)$ and
\begin{align*}
\p_\gamma a_{12}|_{\rhop\base\rho K=1}
    & = -\base\rho\,(1+\rhop\base H^2)\neq0, \\
\p_\gamma a_{21}|_{\rhop\base\rho K=1}
    & = -2\rhop \neq0.
\end{align*}
If~$0=a_{34}=\eps^2\tau^2+\eta^2=0$, then
\begin{align*}
\p_\gamma a_{34} =-\eps+\frac{\eta^2}{\eps\tau^2} = -2\eps\neq 0\, .
\end{align*}
\end{proof}

Now we proceed with the construction of the degenerate Kreiss' symmetrizer, following the classification of points given in Subsection \ref{Classification}. After that, we will consider the special case of poles.

\subsection{Interior points} \label{ssec:interior}

We consider a point $P_0=(\tau_0, \eta_0)\in\Sigma$ with $\gamma_0\doteq\Re\tau_0>0$, i.e. in the internal part of~$\Sigma$. The Lopatinski\u{\i} condition holds and $\Ac(\tau,\eta)$ is diagonalizable in a neighborhood $\Vc$ of $P_0$. Besides
\begin{align*}
e_1
	& = \mu\chi (a_{12},-\omega_1,0,0)\tm\,,\\
e_2
	& = (0,0,\omega_2,\eps \tau)\tm\,,
\intertext{we introduce}
e_3
	& = \mu\chi (a_{12},\omega_1,0,0)\tm\,,\\
e_4
	& = (0,0,\omega_2,-\eps \tau)\tm\,,
\end{align*}
and we define in~$\Vc$ the invertible matrix $T(\tau,\eta)$, given by $T^{-1}=T^{-1}(\tau,\eta)=\begin{pmatrix}e_1 & e_3 & e_2 & e_4\end{pmatrix}$. Notice the inverted order between the eigenvectors $e_2, e_3$, introduced in order to separate the plasma block from the vacuum one.

Clearly, the eigenvectors are smooth, $T\in\Cc^\infty(\Vc,M_4)$ has a nonzero determinant and
\begin{equation}\label{eq:TAT}
T \Ac T^{-1} = \begin{pmatrix}
-\omega_1 & 0 & 0 & 0 \\
0 & \omega_1 & 0 & 0 \\
0 & 0 & -\omega_2 & 0 \\
0 & 0 & 0 & \omega_2
\end{pmatrix} .
\end{equation}
We define the symmetrizer
\begin{align*}
r(\tau,\eta) = \begin{pmatrix}
-1 & 0 & 0 & 0 \\
0 & K' & 0 & 0 \\
0 & 0 & -1 & 0 \\
0 & 0 & 0 & K'
\end{pmatrix} ,
\end{align*}
with~$K'\geq1$ to be fixed sufficiently large, so that
\begin{align*}
rT \Ac T^{-1} = \begin{pmatrix}
\omega_1 & 0 & 0 & 0 \\
0 & K'\omega_1 & 0 & 0 \\
0 & 0 & \omega_2 & 0 \\
0 & 0 & 0 & K'\omega_2
\end{pmatrix} .
\end{align*}
Due to~$\Re\omega_j\geq \kappa$, $j=1,2$, in~$\Vc$ for some~$\kappa>0$, we immediately derive
\begin{equation}\label{eq:symmint0}
\Re(r(\tau,\eta)T(\tau,\eta)\Ac(\tau,\eta) T(\tau,\eta)^{-1})\geq \kappa \Id\,, \qquad\forall (\tau,\eta)\in\Vc,
\end{equation}
where~$\Re M = (M+M^*)/2$ for any~$M\in M_4(\C)$. On the other hand, as in \cite{nappes} we show that there holds
\begin{equation}\label{eq:betaint}
r(\tau,\eta)+ C\tilde\beta(\tau,\eta)^* \tilde\beta(\tau,\eta) \geq \Id\, ,  \qquad\forall (\tau,\eta)\in\Vc,
\end{equation}
where $C$ is a positive constant and $\tilde\beta \doteq \beta T^{-1}$, for sufficiently large~$K'\geq1$. In order to prove it, let us recall that the first and the third columns of $T^{-1}$ are the generators $e_1, e_2$ of the stable subspace $E^-(\tau,\eta)$ of $\Ac(\tau,\eta)$. Since the Lopatinski\u{\i} determinant does not vanish at $(\tau_0,\eta_0)$, there exists a constant $C_0$ independent of $(\tau,\eta)\in \Vc$ such that
\[
|Z_1|^2+|Z_3|^2 \le C_0(|Z_2|^2+|Z_4|^2+|\tilde\beta(\tau,\eta) Z|^2)
\]
for all $Z\in \C^4$. Then
\begin{align*}
&\langle r(\tau,\eta)Z, Z\rangle_{\C^4} + 2C_0|\tilde\beta(\tau,\eta) Z|^2\\
&=-|Z_1|^2-|Z_3|^2 + K'(|Z_2|^2+|Z_4|^2) + 2C_0|\tilde\beta(\tau,\eta) Z|^2\\
 &\ge |Z_1|^2 + |Z_3|^2 + (K'-2C_0)(|Z_2|^2+|Z_4|^2),
\end{align*}
which gives \eqref{eq:betaint} for $K'$ large enough (e.g. $K'=2C_0+1$).

\subsection{Boundary points where the Lopatinski\u{\i} condition holds (except poles or $\omega_1\omega_2=0$)} \label{ssec:boundaryLop}

We consider a point $P_0=(\tau_0, \eta_0)\in\Sigma$ with $\gamma_0\doteq\Re\tau_0=0$, which is not a pole, such that the Lopatinski\u{\i} condition holds at~$P_0$, i.e. $\Delta(\tau_0,\eta_0)\not=0$, and $\Ac(P_0)$ is diagonalizable. We fix a sufficiently small neighborhood $\Vc$ of $P_0$, not containing any pole, in which the Lopatinski\u{\i} condition holds and $\Ac$ is diagonalizable. We may proceed as we did for the interior points, but now we only have $\Re\omega_j\geq \kappa\gamma$, $j=1,2$, in~$\Vc$, for a suitable constant $\kappa>0$. This is due to Lemma~\ref{Lem:Regamma}, see also Remark \ref{remark2}. Therefore we may replace \eqref{eq:symmint0} by
\begin{equation}\label{eq:symmint}
\Re(r(\tau,\eta)T(\tau,\eta)\Ac(\tau,\eta) T(\tau,\eta)^{-1})\geq \kappa\gamma \Id\,, \qquad\forall (\tau,\eta)\in\Vc,
\end{equation}
whereas~\eqref{eq:betaint} still holds true, because of the Lopatinski\u{\i} condition. In the interior points considered in Subsection \ref{ssec:interior} estimate~\eqref{eq:symmint} holds as well, because~\eqref{eq:symmint0} trivially implies~\eqref{eq:symmint}, since~$\gamma\in[0,1]$ on~$\Sigma$.

\subsection{Boundary points where the Lopatinski\u{\i} condition does not hold  (except poles or $\omega_1\omega_2=0$)} \label{ssec:boundarynoLop}

We consider a point $P_0=(\tau_0, \eta_0)\in\Sigma$ with $\gamma_0\doteq\Re\tau_0=0$, which is not a pole and such that the Lopatinski\u{\i} condition does not hold at~$P_0$, that is $\Delta(\tau_0,\eta_0)=0$. From Lemma \ref{lemma:roots} it follows that $\Ac(P_0)$ is diagonalizable. We fix a neighborhood $\Vc$ of $P_0$, not containing any pole, in which $\Ac$ is diagonalizable. We construct~$T$ as we did for the interior points, obtaining~\eqref{eq:TAT}, but now we define the symmetrizer
\begin{align*}
r(\tau,\eta) = \begin{pmatrix}
-\gamma^2 & 0 & 0 & 0 \\
0 & K' & 0 & 0 \\
0 & 0 & -\gamma^2 & 0 \\
0 & 0 & 0 & K'
\end{pmatrix} ,
\end{align*}
where $K'\ge1$ is taken sufficiently large, so that
\begin{align*}
rT \Ac T^{-1} = \begin{pmatrix}
\gamma^2\omega_1 & 0 & 0 & 0 \\
0 & K'\omega_1 & 0 & 0 \\
0 & 0 & \gamma^2\omega_2 & 0 \\
0 & 0 & 0 & K'\omega_2
\end{pmatrix}.
\end{align*}
Taking into account of Lemma~\ref{Lem:Regamma} once again, it follows that~\eqref{eq:symmint} is replaced by
\begin{equation}\label{eq:symmLopa}
\Re(r(\tau,\eta)T(\tau,\eta)\Ac(\tau,\eta) T(\tau,\eta)^{-1})\geq \kappa\gamma^3 \Id\,, \qquad\forall (\tau,\eta)\in\Vc.
\end{equation}
On the other hand, due to the fact that the roots of Lopatinski\u{\i} determinant are simple, we may derive the estimate
\[
\left| \beta(\tau,\eta)\left( e_1(\tau,\eta), e_2(\tau,\eta)  \right)Z'\right|^2\ge \kappa_0\gamma^2|Z'|^2, \qquad \forall Z'\in\C^2,
\]
which implies
\[
\kappa_0\gamma^2(|Z_1|^2+|Z_3|^2) \le C_0(|Z_2|^2+|Z_4|^2+|\tilde\beta(\tau,\eta) Z|^2),
\]
for all $Z\in \C^4$, where  $\tilde\beta \doteq \beta T^{-1}$. 
Then, following~\cite{nappes}, we modify the proof given in Subsection \ref{ssec:interior} and obtain
\begin{equation}\label{eq:betaLopa}
r(\tau,\eta)+ C\tilde\beta(\tau,\eta)^* \tilde\beta(\tau,\eta) \geq \gamma^2 \Id\, , \qquad\forall (\tau,\eta)\in\Vc,
\end{equation}
which replaces~\eqref{eq:betaint}.

\subsection{Boundary points where $\omega_1\omega_2=0$} \label{ssec:nodiag}

We consider a point $P_0=(\tau_0, \eta_0)\in\Sigma$ with $\gamma_0\doteq\Re\tau_0=0$, which is not a pole but such that $\Ac(P_0)$ is not diagonalizable. It follows that the Lopatinski\u{\i} condition holds at~$P_0$. We fix a neighborhood $\Vc$ of $P_0$, not containing any pole, and such that the Lopatinski\u{\i} condition holds in every point.

Let us assume~\eqref{eq:null2}. Thanks to~\eqref{eq:notnull}, it is not restrictive to assume that~\eqref{eq:null1} is nowhere verified in~$\Vc$. Also, we may assume that~\eqref{eq:null3} is nowhere verified in~$\Vc$, provided $\eps$ is sufficiently small. We replace the matrix~$T$, used to deal with points where~$\Ac$ was diagonalizable, with a new matrix~$T$, invertible on~$\Vc$. We define:
\[ T=\begin{pmatrix}
a_{12}^{-1} & 0 &0&0 \\
a_{21} & -i&0&0 \\
0&0& (2\omega_2)^{-1}& (2\eps\tau)^{-1}\\
0&0& (2\omega_2)^{-1}& -(2\eps\tau)^{-1}
\end{pmatrix}, 
 \]
so that
\begin{equation}
\begin{array}{ll}\label{60}

T^{-1}=\begin{pmatrix}
a_{12} & 0 &0&0\\
-i\omega_1^2 & i&0&0\\
0&0& \omega_2 & \omega_2 \\
0&0&\eps\tau &-\eps\tau
\end{pmatrix},
\qquad
T \Ac T^{-1} = \begin{pmatrix}
-i\omega_1^2 & i &0&0\\
-i\omega_1^4-i\omega_1^2 & i\omega_1^2 &0&0\\
0&0& -\omega_2 & 0\\
0&0&0& \omega_2
\end{pmatrix}\,.
\end{array}
\end{equation}
Notice that we have modified only the first two columns of $T$, and consequently of $T^{-1}$. We can proceed similarly if we assume that~\eqref{eq:null1} holds, setting
\[ T=\begin{pmatrix}
-i & a_{12} &0&0\\
0 & a_{21}^{-1} &0&0\\
0&0& (2\omega_2)^{-1}& (2\eps\tau)^{-1}\\
0&0& (2\omega_2)^{-1}& -(2\eps\tau)^{-1}
\end{pmatrix},
 \]
so that
%%
%\[a'\doteq T\begin{pmatrix}
%0 & a_{12} \\
%a_{21} & 0 \end{pmatrix}T^{-1} = i\begin{pmatrix}
%\omega_1^2 & -\omega_1^4-\omega_1^2 \\
%1 & -\omega_1^2
%\end{pmatrix}\,.\]
%%
%
\[ 
T^{-1}=\begin{pmatrix}
i & -i\omega_1^2 &0&0 \\
0 & a_{21} &0&0\\
0&0& \omega_2 & \omega_2 \\
0&0&\eps\tau &-\eps\tau
\end{pmatrix},
\qquad 
T \Ac T^{-1} = \begin{pmatrix}
i\omega_1^2 & -i\omega_1^4-i\omega_1^2 &0&0\\
i &- i\omega_1^2 &0&0\\
0&0& -\omega_2 & 0\\
0&0&0& \omega_2
\end{pmatrix}\,.\]
%
%che si riduce in forma canonica di Jordan nel caso~\eqref{eq:null1}, cio\`e~$\omega_1=0$.
%In particolare,~$\gamma_0=0$ e $(\delta+\base v\eta)=\pm\sqrt{K}\eta$, con~$K=\base H^2/\base\rho$.
%
Moreover, if~\eqref{eq:null3} holds, then we follow the same reasoning by defining
\[ T=\begin{pmatrix}
(2\mu\chi a_{12})^{-1}& -(2\mu\chi \omega_{1})^{-1}&0&0\\
(2\mu\chi a_{12})^{-1}& (2\mu\chi \omega_{1})^{-1}&0&0\\
0&0&-i & a_{34} \\
0&0&0 & a_{43}^{-1} \\
\end{pmatrix}, 
 \]
so that
%
%\[ T\begin{pmatrix}
%0 & a_{34} \\
%a_{43} & 0 \end{pmatrix}T^{-1} = i\begin{pmatrix}
%\omega_2^2 & -\omega_2^4-\omega_2^2 \\
%1 & -\omega_2^2
%\end{pmatrix}\,.\]
%
%
\[ 
 T^{-1}=\begin{pmatrix}
\mu\chi a_{12} & \mu\chi a_{12} &0&0\\
-\mu\chi \omega_{1} & \mu\chi \omega_{1} &0&0\\
0&0& i & -i\omega_2^2 \\
0&0&0 & a_{43} \\
\end{pmatrix},\qquad
T \Ac T^{-1} = \begin{pmatrix}
-\omega_1 &0&0&0\\
0& \omega_1 &0&0\\
0&0 & i\omega_2^2 & -i\omega_2^4-i\omega_2^2 \\
0&0 & i &- i\omega_2^2 \\
\end{pmatrix}\,.\]
In this last case we have modified only the last two columns of $T$. 

From Proposition~\ref{prop:derivgamma}, ~$\p_\gamma (\omega_j^2) \in i\R\setminus\{0\}$ and  it trivially follows that~$\p_\gamma (\omega_j^4+\omega_j^2) \in i\R\setminus\{0\}$ as well. Moreover, we recall that $\omega_1=0$ or $\omega_2=0$ implies $\gamma=0$.
%
%\begin{proof}
%From Lemma~\ref{lem:gamma0}, we already have $\p_\gamma (\omega_1^2)\in i\R$.
%%We first consider~$\omega_1$. We recall that~$\omega_1=0$ implies that~$\gamma=0$. Let~$K$ be as in~\eqref{eq:immu}. Then
%%%
%%\[ \p_\gamma(\omega_1^2)|_{\gamma=0} =\pm 2i\sqrt{K}\eta\rhop\base\rho^2 \frac{K}{k^2} [\base \rho (1+\rhop\base H^2)K-2\base H^2] \in i\R. \]
%%%
%Now, by contradiction, let~$\p_\gamma(\omega_1^2)=0$ when~\eqref{eq:null2} holds, that is, $\rhop\base\rho K=1$. Then:
%%
%\[ 0 = \base \rho (1+\rhop\base H^2)K-2\base H^2 = \base\rho K - \base H^2 = \frac1\rhop (1-\rhop\base H^2) \]
%%
%which contradicts~\eqref{eq:notnull}. Similarly if~\eqref{eq:null1} holds, that is, $\base\rho K=\base H^2$.
%
%%If~$\omega_2=0$, then~$\gamma=0$ and~$\p_\gamma (\omega_2^2) = 2i\eps^2\delta \in i\R\setminus{0}$.
%On the other hand, when $\gamma=0$, we have $\p_\gamma (\omega_2^2) = 2i\eps^2\delta \in i\R\setminus{0}$ by an explicit computation.
%\end{proof}

In the following, we assume that~$P_0=(\tau_0,\eta_0)$ verifies~\eqref{eq:null2}, in particular~$\gamma_0\doteq\Re\tau_0=0$, being the other cases analogous.
We look for a symmetrizer $r$ under the form
\[ r(\tau,\eta)=\begin{pmatrix}
s(\tau,\eta) & 0 & 0\\
 0 & -1 & 0\\
 0 & 0 & K'
\end{pmatrix}, \]
where $K'\ge1$ is a real number, and $s$ is some $2\times 2$ hermitian matrix, smoothly depending on $(\tau,\eta)$.
Let us focus on the first block of $T \Ac T^{-1}$ in \eqref{60},  
\[ a' \doteq\begin{pmatrix}
-i\omega_1^2 & i \\
-i\omega_1^4-i\omega_1^2 & i\omega_1^2
 \end{pmatrix}, \]
as we may symmetrize the second block as we did previously. Recall that~$a'$ is purely imaginary when~$\gamma=0$. Let
\[ s(\tau,\eta)=\underbrace{\begin{pmatrix}
0 & \epsilon_1 \\
\epsilon_1 & \epsilon_2 \end{pmatrix}}_E + \underbrace{\begin{pmatrix}
f(\tau,\eta) & 0 \\
0 & 0 \end{pmatrix}}_{F(\tau,\eta)} -i\gamma \underbrace{\begin{pmatrix}
0 & -g \\
g & 0 \end{pmatrix}}_G \,, \]
for suitable real numbers~$\epsilon_1, \epsilon_2, g$ and $f$ real valued $C^\infty$ function such that~$f(\tau_0,\eta_0)=0$. By virtue of Proposition~\ref{prop:derivgamma}, we may fix
\begin{equation}\label{eq:e1}
\epsilon_1= i (\partial_\gamma (\omega_1^2)(\tau_0,\eta_0))^{-1} \in \R\setminus\{0\}\,,
\end{equation}
%
%assuming with no loss of generality that~$\partial_\gamma (\omega_1^2)\neq0$ in~$\Vc$.
%
%
%
The form of $s$ yields ($\gamma_0=0$)
\[ r(\tau_0,\eta_0)=\begin{pmatrix}
0 & \epsilon_1 &0&0\\
\epsilon_1 & \epsilon_2  & 0 & 0\\
0& 0 & -1 & 0\\
0& 0 & 0 & K'
\end{pmatrix}. \]
Next, we notice that the third column of $T^{-1}$ in \eqref{60} is simply $e_2$ and that the first one evaluated at $(\tau_0,\eta_0)$ (where $\omega_1=0$) equals $(a_{12},0,0,0)\tm$ with $a_{12}(\tau_0,\eta_0)\not=0$. On the other hand, we have $e_1(\tau_0,\eta_0)=(\mu\chi a_{12},0,0,0)\tm$
 with $\mu\chi a_{12}(\tau_0,\eta_0)\not=0$. It follows that the two vectors are parallel. Thanks to the Lopatinski\u{\i} condition, we can find a constant $\kappa_0$ such that
 \[
\left| \beta(\tau_0,\eta_0)\left( (a_{12}(\tau_0,\eta_0),0,0,0)\tm, e_2(\tau_0,\eta_0)  \right)Z'\right|^2\ge \kappa_0|Z'|^2, \qquad \forall Z'\in\C^2,
\]
and consequently there exists a constant $C_0$ such that
\[
|Z_1|^2+|Z_3|^2 \le C_0(|Z_2|^2+|Z_4|^2+|\tilde\beta(\tau_0,\eta_0) Z|^2),
\]
for all $Z\in \C^4$. Then
\begin{align*}
&\langle r(\tau_0,\eta_0)Z, Z\rangle_{\C^4} + C'C_0|\tilde\beta(\tau_0,\eta_0) Z|^2\\
&=2\epsilon_1\Re\langle Z_1, Z_2\rangle_{\C^4} +\epsilon_2 |Z_2|^2  - |Z_3|^2 + K'|Z_4|^2 + C'C_0|\tilde\beta(\tau_0,\eta_0) Z|^2\\
 &\ge (C'-\max\{|\epsilon_1|,1\} ) (|Z_1|^2 + |Z_3|^2) + (\epsilon_2-|\epsilon_1|-C'C_0)|Z_2|^2+ (K'-C'C_0)|Z_4|^2.
\end{align*}
We choose $C'=\max\{|\epsilon_1|,1\}+2,\, \epsilon_2=|\epsilon_1|+C'C_0+2,\, K'=C'C_02+2$, and obtain
\[ r(\tau_0,\eta_0) + C'C_0 \tilde{\beta}^*(\tau_0,\eta_0) \tilde{\beta}(\tau_0,\eta_0) \geq 2\Id. \]
Up to shrinking $\Vc$ we have thus derived the estimate
\[ r(\tau,\eta) + C \tilde{\beta}(\tau,\eta)^* \tilde{\beta}(\tau,\eta) \geq \Id, \qquad \forall(\tau,\eta)\in \Vc,\]
for a suitable constant $C$.

By Taylor's formula at~$(i\delta,\eta)$ with respect to the variable~$\gamma$, we may write:
\[ a'(\tau,\eta) = a'(i\delta,\eta) + \gamma \p_\gamma a'(i\delta,\eta) + \gamma^2 M(\tau,\eta), \]
for a suitable continuous function $M$.
Recalling that $a'(i\delta,\eta)$ is purely imaginary, we may choose~$f\doteq  -2\epsilon_1 \omega_1^2 - \epsilon_2(\omega_1^4+\omega_1^2)$, so that the matrix
\[ -i(E+F(\tau,\eta))a'(i\delta,\eta)   =
\begin{pmatrix}
 -f\omega_1^2 -\epsilon_1 (\omega_1^4+\omega_1^2) & f+\epsilon_1\omega_1^2  \\
 -\epsilon_1 \omega_1^2 - \epsilon_2(\omega_1^4+\omega_1^2) & \epsilon_1 +\epsilon_2\omega_1^2
\end{pmatrix} \]
is real and symmetric for all $(\tau,\eta)$. Therefore, being $\p_\gamma a'(i\delta,\eta)$ real from Proposition~\ref{prop:derivgamma}, we get
\[ \Re (s(\tau,\eta)a'(\tau,\eta)) = \gamma \Re (-iGa' +(E+F)\p_\gamma a') + \gamma^2 \Re (sM). \]
Setting
\[ N \doteq \begin{pmatrix}
0 & 1 \\
0 & 0
\end{pmatrix}, \qquad N_1 := a'(i\delta,\eta) - a'(\tau_0,\eta_0)\,, \]
so that~$a'(i\delta,\eta)=iN+N_1$, we may write
\[ \Re (s(\tau,\eta)a'(\tau,\eta)) = \gamma (GN +E\p_\gamma a'(i\delta,\eta)+L(\tau,\eta)), \]
where~$L(\tau_0,\eta_0)=0$ (notice that $N_1, F, \gamma M$ are included in $L$ since they vanish at $(\tau_0,\eta_0)=(i\delta_0,\eta_0)$). Having in mind~\eqref{eq:e1}, we obtain
\[ GN +E\p_\gamma a'(i\delta_0,\eta_0)=\begin{pmatrix}
0 & 0 \\
0 & g \end{pmatrix} + \begin{pmatrix}
1 & -1 \\
1+\epsilon_2/\epsilon_1  & -\epsilon_2/\epsilon_1
\end{pmatrix}. \]
For~$g$ sufficiently large, it is not restrictive to assume that
\[ \Re (s(\tau,\eta)a'(\tau,\eta))\geq \frac12 \gamma \Id, \qquad \text{so that}\quad\Re(r(\tau,\eta)T(\tau,\eta)\mathcal{A}(\tau,\eta)T(\tau,\eta)^{-1})\geq \kappa\gamma\Id,\]
in~$\Vc$, for a suitable constant $\kappa$. Summarizing, we obtained~\eqref{eq:betaint} and~\eqref{eq:symmint}.

\bigskip

%\newpage

\subsection{Pole $\mu=0$}\label{polomu}

We consider a point $P_0=(\tau_0, \eta_0)\in\Sigma$ with $\tau_0+i \base v\eta_0=\mu_0=0$. This is a simple pole for the plasma part of $\Ac(P_0)$, specifically for $a_{12}$. In this point the other coefficient $a_{21}$ of the plasma block vanishes and $\omega_1=|\eta_0|\not=0$. Moreover, $P_0$ is a point where the Lopatinski\u{\i} condition doesn't hold because the Lopatinski\u{\i} determinant vanishes. We fix a small neighborhood $\Vc$ of $P_0$, not containing any other critical point. From \eqref{omega1}, \eqref{omega2^2} we can assume that both eigenvalues $\omega_1,\omega_2$ are bounded and different from zero in $\Vc$; the inequality for $\omega_2$ is true provided $\eps$ is taken sufficiently small.

Inspired by \cite{nappes,WYuan}, we use a different approach to derive the energy estimate by converting problem \eqref{symbolic-equ} into a system of simple form. Let us consider the new variables 

\begin{equation}
\begin{array}{ll}\label{defLambda}
\ds
W
= \Lambda V^{\nc}\doteq
\begin{pmatrix}
1 & 1&0&0\\
-1  & 1&0&0\\
0&0&1&0\\
0&0&0&1
\end{pmatrix}
V^{\nc}.
\end{array}
\end{equation}
It follows from \eqref{symbolic-equ} that
\begin{equation}
\begin{array}{ll}\label{64}
 \d1 \fou{W}=\Lambda\,\Ac\,\Lambda^{-1} \fou{W}= \begin{pmatrix}
m_1 & -m_2 &0&0\\
m_2 & -m_1 &0&0\\
0&0&0& a_{34}\\
0&0& a_{43} &0
\end{pmatrix} \fou{W},
\end{array}
\end{equation}
where we have set
\begin{equation}
\begin{array}{ll}\label{defm1m2}

m_1= \dfrac{ a_{21}+a_{12}}{2} \,,
\qquad
m_2= \dfrac{ a_{21}-a_{12}}{2} \,.

\end{array}
\end{equation}
Notice that both coefficients $m_1$ and $m_2$ have a pole in $P_0$. We also remark that $\omega_1^2=m_1^2-m_2^2$.
The reader may recognize in the plasma block of \eqref{64} the form of the symbol in \cite{nappes}, see (4.12).

We define a new matrix~$T$, invertible on~$\Vc$, in the following way:
\begin{equation}
\begin{array}{ll}\label{defTmu}

T^{-1}=\begin{pmatrix}
\tilde{T}^{-1} &0&0\\
0& \omega_2 & \omega_2 \\
0&\eps\tau &-\eps\tau
\end{pmatrix},
\end{array}
\end{equation}

where
\[
\tilde{T}^{-1}=\begin{pmatrix}
\mu (m_1-\omega_1) &  -\mu m_2\\
 \mu m_2 & \mu (m_1-\omega_1)
\end{pmatrix}.
\]
We have finite
\[
\det\, \tilde{T}^{-1}=\mu^2(m_1-\omega_1)^2 + \mu^2m_2^2 = 2m_1\mu^2 (m_1-\omega_1)\not=0 \qquad \forall(\tau,\eta)\in\Vc,
\]
and
\[ T=\begin{pmatrix}
\tilde{T}&0&0 \\
0& (2\omega_2)^{-1}& (2\eps\tau)^{-1}\\
0& (2\omega_2)^{-1}& -(2\eps\tau)^{-1}
\end{pmatrix}, 
 \]
where
\[ \tilde{T}=\dfrac1{\det\, \tilde T^{-1}}\begin{pmatrix}
{\mu (m_1-\omega_1)} &  {\mu m_2 } \\
-  {\mu m_2} & {\mu (m_1-\omega_1)}
\end{pmatrix}.
 \]
The matrix $T$ is such that
\begin{equation}
\label{65}
T \begin{pmatrix}
m_1 & -m_2 &0&0\\
m_2 & -m_1 &0&0\\
0&0&0& a_{34}\\
0&0& a_{43} &0
\end{pmatrix} T^{-1} = \Ac' \doteq \begin{pmatrix}
-\omega_1 & -2m_2 &0&0\\
0 & \omega_1 &0&0\\
0&0& -\omega_2 & 0\\
0&0&0& \omega_2
\end{pmatrix}
\,, \qquad\forall (\tau,\eta)\in\Vc.
\end{equation}
We shall derive the energy estimates directly by making use of the ODE system derived by the above transformation, instead of constructing the symmetrizer of this problem. This will be done in the Subsection \ref{third}.

\bigskip

%\newpage

\subsection{Pole of \eqref{eq:caso2}}\label{polocaso2}

We consider a point $P_0=(\tau_0, \eta_0)\in\Sigma$ with $\gamma_0=0$ and $\delta_0+\base v \eta_0 = + \frac{|\base H|}{\sqrt{\base\rho(1+\base\alpha\base H^2)}}\eta_0$, with $\eta_0> 0$. (We shall not detail the case $\gamma_0=0$, $\delta_0+\base v \eta_0 = - \frac{|\base H|}{\sqrt{\base\rho(1+\base\alpha\base H^2)}}\eta_0$, that is entirely similar.)

This is a simple pole for the plasma part of $\Ac(P_0)$, specifically for $a_{21}$, and the eigenvalues $\pm\omega_1$ have a simple pole as well. In this point the other coefficient $a_{12}$ of the plasma block is well defined and different from zero. In $P_0$ the quantity $\chi$ vanishes, and $\chi\omega_1$ is well defined and different from zero because
\[
\chi^2\omega_1^2(P_0)=\eta_0^2/(1+\base\alpha\base H^2)>0.
\]
Moreover, at $P_0$ the Lopatinski\u{\i} condition is satisfied. We fix a small neighborhood $\Vc$ of $P_0$, not containing any other critical point, where we assume that $a_{12}\not=0$ and $\chi\omega_1\not=0$.

We deal with this case by using an argument inspired by the work of Majda and Osher \cite{majda-osher}, see also \cite{WYu}. 

Let us set
\[
\ttau = \gamma +i\tdelta \doteq\mu -i\frac{|\base H|}{\sqrt{\base\rho(1+\base\alpha\base H^2)}}\eta ,
\]
which means that $P_0$ corresponds to $\ttau=0$.
Let us consider the variables $W$ defined in \eqref{defLambda}, with the ODE \eqref{64}.

Recalling that $m_1-m_2=a_{12}\not=0$, we define a new matrix~$T$, invertible on~$\Vc$, in the following way:
\begin{equation}
\begin{array}{ll}\label{defT-1}
T^{-1}=\begin{pmatrix}
1 &i(m_1-m_2)^{-1} &0&0\\
-1 &i(m_1-m_2)^{-1} &0&0\\
0& 0&\omega_2 & \omega_2 \\
0&0& \eps\tau &-\eps\tau
\end{pmatrix}.

\end{array}
\end{equation}
The matrix $T$ is such that
\begin{equation}
\label{65'}
T \begin{pmatrix}
m_1 & -m_2 &0&0\\
m_2 & -m_1 &0&0\\
0&0&0& a_{34}\\
0&0& a_{43} &0
\end{pmatrix} T^{-1} = \Ac'' \doteq \begin{pmatrix}
0 & i &0&0\\
-i\omega_1^2 &0&0&0\\
0&0& -\omega_2 & 0\\
0&0&0& \omega_2
\end{pmatrix}
\,, \qquad\forall (\tau,\eta)\in\Vc.
\end{equation}
Denote by 
\[
\omega_0(\tau,\eta)=-\omega_1^2 \ttau=ie_0(i\delta,\eta)+\gamma d_0(\tau,\eta)
\]
with $e_0(i\delta,\eta)\in\R$. By a direct computation and \eqref{omega1} it follows that
\[
\omega_0(\tau_0,\eta_0)= i\frac{|\base H|}{\sqrt{\base\rho}}\frac{\base\alpha\base H^2}{2(1+\base\alpha\base H^2)^{5/2}}\eta_0^3,
\]
which implies $e_0(i\delta,\eta)>0$, because $\eta_0>0$. As in \cite{WYu} we look for a symmetrizer $r$ under the form
\[ r(\tau,\eta)=\begin{pmatrix}
\tilde r(\tau,\eta) & 0 & 0\\
 0 & -1 & 0\\
 0 & 0 & K'
\end{pmatrix}, \]
where $K'\ge2$ is a real number, and $\tilde r$ is a $2\times 2$ hermitian matrix
\[ \tilde r(\tau,\eta)=\begin{pmatrix}
d_1 & d_2+i\gamma s\\
 d_2-i\gamma s & d_1\dfrac{\tdelta}{e_0}
\end{pmatrix}, \]
with $d_1>0, d_2<0$ and $s>0$ constants to be determined later. We easily obtain
\begin{equation}
\begin{array}{ll}\label{74}
\Re \left( r(\tau,\eta)\Ac''(\tau,\eta) \right) = \Re \begin{pmatrix}
(d_2+i\gamma s)\dfrac{i\omega_0}{\ttau} & id_1 & 0 & 0\\
i d_1\dfrac{\tdelta\omega_0}{e_0\ttau} &\gamma s +id_2&0&0\\
 0&0 & \omega_2 & 0\\
0& 0 & 0 & K'\omega_2
\end{pmatrix}
\\
=  \begin{pmatrix}
R(\tau,\eta) & \bar{J}(\tau,\eta) & 0 & 0\\
J(\tau,\eta) &\gamma s  &0&0\\
 0&0 & \Re\omega_2 & 0\\
0& 0 & 0 & K' \Re\omega_2
\end{pmatrix},

\end{array}
\end{equation}
where
\[
R(\tau,\eta)=\dfrac12 \left(  (d_2+i\gamma s)\dfrac{i\omega_0}{\ttau} + (d_2- i\gamma s)\dfrac{\overline{i\omega_0}}{\overline\ttau} \right), \qquad
J(\tau,\eta)= \dfrac{i d_1}2 \left(  \dfrac{\tdelta\omega_0}{e_0\ttau} -1 \right) .
\]
By calculation we get
\begin{equation*}
\begin{array}{ll}\label{}
R(\tau,\eta)=\dfrac1{2|\ttau|^2} \left\{  (d_2+i\gamma s)(-e_0+i\gamma d_0)(\gamma -i\tdelta) + (d_2- i\gamma s)(-e_0- i\gamma \overline{d_0})(\gamma + i\tdelta)
\right\}
\\
=\dfrac1{2|\ttau|^2} \left\{  d_2(-e_0 - \gamma \Im d_0 + \tdelta \Re d_0) - s( \gamma^2 \Re{d_0} +\tdelta e_0 + \gamma \tdelta \Im d_0)
\right\}.
\end{array}
\end{equation*}
By observing that in $\Vc$, $e_0(i\delta,\eta)>0$, $\gamma$ and $\tdelta$ are very small and $\Re d_0, \Im d_0$ are bounded, we easily obtain the following inequalities
\[
R(\tau,\eta) \ge \dfrac{\gamma}{|\ttau|^2} \left(- \dfrac{e_0}2 d_2  - C_1|\ttau| s \right), \qquad | J (\tau,\eta)| \le \dfrac{d_1\gamma}{|\ttau|}
\,, \qquad\forall (\tau,\eta)\in\Vc,
\]
for a suitable positive constant $C_1$. It follows that for all $Z=(Z_1,Z_2)\tm \in \C^2$,
\begin{equation}
\begin{array}{ll}\label{75}
\overline{Z}\tm  \begin{pmatrix}
R(\tau,\eta) & \bar{J}(\tau,\eta) \\
J(\tau,\eta) &\gamma s  
\end{pmatrix} Z = R(\tau,\eta) |Z_1|^2 +\gamma s |Z_2|^2 + 2\Re( J(\tau,\eta) Z_1\overline Z_2)
\\
\quad \ge R(\tau,\eta) |Z_1|^2 +\gamma s |Z_2|^2 - 2\dfrac{d_1\gamma}{|\ttau|}| Z_1| | Z_2|
\\
\quad \ge R(\tau,\eta) |Z_1|^2 +\gamma s |Z_2|^2 - \left(\dfrac{\epsilon\gamma}{|\ttau|^2}| Z_1|^2 +\dfrac{\gamma d_1^2}{\epsilon} | Z_2|^2 \right)
\\
\quad \ge \left(- \dfrac{e_0}2 d_2  - C_1|\ttau| s - \epsilon\right) \dfrac{\gamma}{|\ttau|^2}|Z_1|^2 +\gamma  \left(s -\dfrac{d_1^2}{\epsilon} \right)| Z_2|^2 ,
\end{array}
\end{equation}
for $\epsilon>0$ to be determined later.

%%%%%%%%%%%%%%%%%%%%%%%%%%%%%
%%%%%%%%%%%%%%%%%%%%%%%%%%%%%
%%%%%%%%%%%%%%%%%%%%%%%%%%%%%
%%%%%%%%%%%%%%%%%%%%%%%%%%%%%

Now we consider the matrix $\Lambda^{-1}T^{-1}$, for $\Lambda$ defined in \eqref{defLambda} and $T$ as in \eqref{defT-1}. We compute
\[
\Lambda^{-1}T^{-1}(\tau_0,\eta_0)=
\begin{pmatrix}
1 &0 &0&0\\
0 &i a_{12}(\tau_0,\eta_0)^{-1}  &0&0\\
0& 0&\omega_2(\tau_0,\eta_0) & \omega_2(\tau_0,\eta_0) \\
0&0& \eps\tau_0 &-\eps\tau_0
\end{pmatrix},
\]
where we notice that the third column of the matrix is the eigenvector $e_2(\tau_0,\eta_0)$. At $P_0$ we have
\[
e_1(\tau_0,\eta_0)=- (0,\mu\chi\omega_1(\tau_0,\eta_0),0,0)\tm\not=0  ,
\]
which is therefore parallel to the second column of the above matrix.
%%%%
Since at $P_0$ the Lopatinski\u{\i} condition is satisfied, it follows that there exists a constant $C_0$ such that
%\begin{equation}
%\begin{array}{ll}\label{stimaZ'}
%
%|(Z'_2,Z'_3)|^2 \le C'_0(|(Z_1',Z'_4)|^2+|\tilde\beta(\tau_0,\eta_0) Z'|^2)
%
%\end{array}
%\end{equation}
%for all $Z'=(Z'_1,Z'_2,Z'_3,Z'_4)\tm \in \C^4$, where
%\[
%\tilde\beta(\tau_0,\eta_0) = \beta(\tau_0,\eta_0)
%\begin{pmatrix}
%1 &0 &0&0\\
%-1 &1 &0&0\\
%0& 0&\omega_2(\tau_0,\eta_0) & \omega_2(\tau_0,\eta_0) \\
%0&0& \eps\tau_0 &-\eps\tau_0
%\end{pmatrix}.
%\]
%Here we are using that the vector $(0,1,0,0)\tm$ is parallel to $e_1(\tau_0,\eta_0)$, and that the third column is the eigenvector $e_2$. On the other hand, 
%\[
%\begin{pmatrix}
%1 &0 &0&0\\
%-1 &1 &0&0\\
%0& 0&\omega_2 & \omega_2 \\
%0&0& \eps\tau &-\eps\tau
%\end{pmatrix}
%=
%\begin{pmatrix}
%1 &i a_{12}^{-1} &0&0\\
%-1 &i a_{12}^{-1} &0&0\\
%0& 0&\omega_2 & \omega_2 \\
%0&0& \eps\tau &-\eps\tau
%\end{pmatrix}
%%
%\begin{pmatrix}
%1 &- 1/2 &0&0\\
%0 &a_{12}/{2i} &0&0\\
%0& 0&1 & 0 \\
%0&0& 0 & 1
%\end{pmatrix},
%\]
%where the first matrix on the right hand side of the equality is the matrix $T^{-1}$, defined in \eqref{defT-1}. 
%Choosing 
%\[
%Z'=\begin{pmatrix}
%1 &- 1/2 &0&0\\
%0 &a_{12}(\tau_0,\eta_0)/{2i} &0&0\\
%0& 0&1 & 0 \\
%0&0& 0 & 1
%\end{pmatrix}^{-1} Z
%\]
%in \eqref{stimaZ'} yields
\begin{equation}
\begin{array}{ll}\label{stimaZ}

|(Z_2,Z_3)|^2 \le C_0(|(Z_1,Z_4)|^2+|\beta(\tau_0,\eta_0)\Lambda^{-1}T^{-1}(\tau_0,\eta_0) Z|^2)

\end{array}
\end{equation}
for all $Z=(Z_1,Z_2,Z_3,Z_4)\tm \in \C^4$. Moreover, recalling that at $P_0$ it holds $\gamma=\tilde\delta=0$,  we have
\begin{equation*}
\begin{array}{ll}\label{}
\langle r(\tau_0,\eta_0)Z,Z   \rangle_{\C^4} 
= d_1 |Z_1|^2  - |Z_3|^2  +K'  |Z_4|^2
+2  d_2 \Re(Z_1 \overline{Z_2}) 
\\
\ge  (d_1 +d_2 )|Z_1|^2 + d_2   |Z_2|^2 - |Z_3|^2  +K'  |Z_4|^2

.
\end{array}
\end{equation*}
It implies
\begin{equation*}
\begin{array}{ll}\label{}
\langle r(\tau_0,\eta_0)Z,Z   \rangle_{\C^4} -2d_2C_0|\beta(\tau_0,\eta_0)\Lambda^{-1}T^{-1}(\tau_0,\eta_0) Z|^2
\\
\ge  (d_1 +(1+2C_0)d_2  )|Z_1|^2   - d_2   |Z_2|^2 - (2d_2 +1 )|Z_3|^2  +(K' +2d_2C_0)  |Z_4|^2
.
\end{array}
\end{equation*}
Now, we choose $d_2<-2, d_1> - 2(1+C_0)d_2, \epsilon=- \dfrac{e_0}4 d_2, s> \dfrac{d_1^2}{\epsilon}$ and $K'$ large enough. We obtain
\begin{equation*}
\begin{array}{ll}\label{}
 r(\tau_0,\eta_0) +C (\beta(\tau_0,\eta_0)\Lambda^{-1}T^{-1}(\tau_0,\eta_0) )^\ast
\beta(\tau_0,\eta_0)\Lambda^{-1} T^{-1}(\tau_0,\eta_0)
\ge  2\Id\,,

\end{array}
\end{equation*}
for a suitable constant $C$.
Up to shrinking $\Vc$ we have thus derived the estimate
\begin{equation}
\begin{array}{ll}\label{estpolo22}
 r(\tau,\eta) +C \left( \beta(\tau,\eta)\Lambda^{-1}T^{-1}(\tau,\eta) \right)^\ast
\beta(\tau,\eta)\Lambda^{-1}T^{-1}(\tau,\eta)
\ge  \Id\,, \qquad\forall (\tau,\eta)\in\Vc.

\end{array}
\end{equation}
Moreover, from \eqref{65'}, \eqref{74}, \eqref{75} we have

\begin{equation}
\begin{array}{ll}\label{estpolo21}
\Re \left( r(\tau,\eta) T(\tau,\eta)
\Lambda\,\Ac(\tau,\eta)\,(T(\tau,\eta)\Lambda)^{-1} \right)\\
=\Re \left( r(\tau,\eta) T(\tau,\eta)
\Ac''(\tau,\eta) T(\tau,\eta)^{-1} \right)
\ge \kappa \begin{pmatrix}
\dfrac{\gamma}{|\ttau|^2} &0 &0&0\\
0 & \gamma &0&0\\
0& 0&\gamma & 0 \\
0&0& 0 &\gamma
\end{pmatrix}
\,, \qquad\forall (\tau,\eta)\in\Vc,

\end{array}
\end{equation}
for a positive constant $\kappa$.

%\bigskip

%\newpage

\subsection{Pole $\tau=0$}\label{polotau}

We consider a point $P_0=(\tau_0, \eta_0)\in\Sigma$ with $\tau_0=0$. This is a simple pole for the vacuum part of $\Ac(P_0)$, specifically for $a_{34}$. In this point the other coefficient $a_{43}$ of the vacuum block vanishes and $\omega_2=|\eta_0|\not=0$. Moreover, $P_0$ is a point where the Lopatinski\u{\i} condition doesn't hold because the Lopatinski\u{\i} determinant vanishes. We fix a small neighborhood $\Vc$ of $P_0$, not containing any other critical point. From \eqref{omega1}, \eqref{omega2^2} we can assume that both eigenvalues $\omega_1,\omega_2$ are bounded and different from zero in $\Vc$.

We use the same approach of Subsection \ref{polomu} to derive the energy estimate by converting problem \eqref{symbolic-equ} into a system of simple form. Let us consider the new variables 

\begin{equation}
\begin{array}{ll}\label{defLambda'}
\ds
W
= \Lambda' V^{\nc}\doteq
\begin{pmatrix}
1 & 0&0&0\\
0  & 1&0&0\\
0&0&1&1\\
0&0&-1&1
\end{pmatrix}
V^{\nc}.
\end{array}
\end{equation}
It follows from \eqref{symbolic-equ} that
\begin{equation}
\begin{array}{ll}\label{64'}
 \d1 \fou{W}=\begin{pmatrix}
0 & a_{12} &0&0\\
a_{21} & 0 &0&0\\
0&0&n_1& -n_2\\
0&0& n_2 &-n_1
\end{pmatrix} \fou{W},
\end{array}
\end{equation}
where we have set
\begin{equation}
\begin{array}{ll}\label{defn1n2}

n_1= \dfrac{ a_{43}+a_{34}}{2} \,,
\qquad
n_2= \dfrac{ a_{43}-a_{34}}{2} \,.

\end{array}
\end{equation}
Notice that both coefficients $n_1$ and $n_2$ have a pole in $P_0$. We also remark that $\omega_2^2=n_1^2-n_2^2$.
We define a new matrix~$T$, invertible on~$\Vc$, in the following way:
\[
T^{-1}=\begin{pmatrix}
\mu\chi a_{12} & \mu\chi a_{12} &0\\
-\mu\chi \omega_{1} & \mu\chi \omega_{1} &0\\
0& 0 &\tilde{T}^{-1}
\end{pmatrix},
\]
where
\[
\tilde{T}^{-1}=\begin{pmatrix}
\tau(n_1-\omega_2) &  -  \tau n_2\\
\tau  n_2 & \tau(n_1-\omega_2)
\end{pmatrix}.
\]
We have
\[
\det\, \tilde{T}^{-1}= \tau^2 (n_1-\omega_2)^2 + \tau^2 n_2^2 = 2\tau^2 n_1 (n_1-\omega_2)\not=0 \qquad \forall(\tau,\eta)\in\Vc,
\]
and
\[ T=\begin{pmatrix}
(2\mu\chi a_{12})^{-1}& -(2\mu\chi \omega_{1})^{-1}&0\\
(2\mu\chi a_{12})^{-1}& (2\mu\chi \omega_{1})^{-1}&0\\
0& 0& \tilde{T}
\end{pmatrix}, 
 \]
where
\[ \tilde{T}=\dfrac1{\det\, \tilde T^{-1}}\begin{pmatrix}
\tau( n_1-\omega_2) &  \tau{n_2 } \\
-  \tau{n_2} & \tau( n_1-\omega_2)
\end{pmatrix}.
 \]
The matrix $T$ is such that
\begin{equation}
\label{65''}
T \begin{pmatrix}
0 & a_{12} &0&0\\
a_{21} & 0 &0&0\\
0&0&n_1& -n_2\\
0&0& n_2 &-n_1
\end{pmatrix} T^{-1} = \Ac'''\doteq \begin{pmatrix}
-\omega_1 & 0 &0&0\\
0 & \omega_1 &0&0\\
0&0& -\omega_2 & -2n_2\\
0&0&0& \omega_2
\end{pmatrix}
\,, \qquad\forall (\tau,\eta)\in\Vc.
\end{equation}
As for the case $\mu=0$, we shall derive the energy estimates directly by making use of the ODE system derived by the above transformation, instead of constructing the symmetrizer of this problem. This will be done in the Subsection \ref{fifth}.

\bigskip

\section{Energy estimate}\label{Stima dell'energia}

We now turn to the derivation of the estimate \eqref{stimaV2bdry}. After the reduction of the problem to the homogeneous equation and the elimination of the front, recall that we are considering a function $V\in H^1(\Omega)$ such that
\[ \begin{cases}
(\tau \Ac_0 + i\eta\Ac_2 + \Ac_1\p_1)\fou V=0 & x_1>0, \\
\beta(\tau,\eta) \fou V^\nc = \fou h & x_1=0,
\end{cases}\]
where
\[
\fou h=Q(\tau,\eta) \fou g=\begin{pmatrix}
0 & 1 & 0 \\
-\eps\base \Hc \tau & 0 & \tau + i \base v \eta \\
\bar\tau - i \base v \eta & 0 & \eps\base \Hc \bar \tau
\end{pmatrix}\fou g, \qquad \forall (\tau,\eta)\in\Sigma.
\] 

The previous analysis shows that for all $(\tau_0,\eta_0)\in\Sigma$, there exists a neighborhood $\Vc$ of $(\tau_0,\eta_0)$ and mappings defined on this neighborhood that satisfy suitable properties. Because $\Sigma$ is a $C^\infty$ compact manifold, there exists a finite covering $(\Vc_1,\dots,\Vc_I)$ of $\Sigma$ by such neighborhoods, and a smooth partition of unity $(\chi_1,\dots,\chi_I)$ associated with this covering.
The $\chi_i's$ are nonnegative $C^\infty$ functions with
\[
\supp\chi_i\subset\Vc_i, \qquad \sum_{i=1}^I\chi_i^2=1.
\]
We consider the different cases.
%\medskip

\subsubsection{The first case.}\label{first}
$\Vc_i$ is a neighborhood of an interior point or a neighborhood of a boundary point corresponding to cases in Subsections \ref{ssec:boundaryLop} or \ref{ssec:nodiag}, that is boundary points that are not poles and such that the Lopatinski\u{\i} condition is satisfied.

On such a neighborhood there exist two $C^\infty$ mappings~$r_i$ e $T_i$ such that $r_i$ is hermitian, $T_i$ has values in $GL_4(\C)$, and the following estimates hold for all $ (\tau,\eta)\in\Vc_i$:
\begin{subequations}\label{symmint}
\begin{align}
\Re(r_i(\tau,\eta)T_i(\tau,\eta)\Ac(\tau,\eta) T_i(\tau,\eta)^{-1})\geq \kappa_i\gamma \Id\,, \label{symminta}\\
r_i(\tau,\eta)+ C(\beta(\tau,\eta) T_i(\tau,\eta)^{-1})^* (\beta(\tau,\eta) T_i(\tau,\eta)^{-1})\geq \Id\, ,\label{symmintb}
\end{align}
\end{subequations}
see \eqref{eq:betaint}, \eqref{eq:symmint}.

We define
\[ U_i(\tau,x_1,\eta)\doteq \chi_i(\tau,\eta)\,T_i(\tau,\eta)\,\fou V^\nc(\tau,x_1,\eta). \]
Here $r_i$ and $T_i$ are only defined on $\Vc_i$, but for convenience we first extend the definition to the whole $\Sigma$, then extend~$\chi_i, r_i$ and $ T_i$ to the whole set of frequencies $\Xi$, as homogeneous mappings of degree 0 with respect to $(\tau,\eta)$. We easily show that $U_i$ satisfies
\[\begin{cases}
{\dfrac{dU_i}{dx_1}}  = T_i(\tau,\eta)\Ac(\tau,\eta) T_i(\tau,\eta)^{-1} U_i &\quad x_1>0,\\
\beta(\tau,\eta) T_i(\tau,\eta)^{-1} U_i = \chi_i\fou h &\quad x_1=0.
\end{cases}\]
We take the scalar product of the above ordinary differential equation with~$r_iU_i$ and integrate w.r.t. $x_1$ on~$[0,+\infty)$. Then we take the real part and use \eqref{symminta} to obtain
\[
\Re\int_0^{+\infty} \langle r_iU_i, {\dfrac{dU_i}{dx_1}} \rangle\, dx_1 = \Re\int_0^{+\infty} \langle U_i, r_iT_i\Ac T_i^{-1} U_i  \rangle\, dx_1 \ge \kappa_i\gamma\int_0^{+\infty}|U_i(\tau,x_1,\eta)|^2\, dx_1.
\]
On the other hand, from \eqref{symmintb} we get
\begin{equation*}
\begin{array}{ll}\label{}
\ds
\Re\int_0^{+\infty} \langle r_iU_i, {\dfrac{dU_i}{dx_1}} \rangle\, dx_1 = -\dfrac12 r_i|U_i(\tau,0,\eta)|^2
\le \dfrac{C}2|\beta T_i^{-1} U_i(\tau,0,\eta)|^2 - \dfrac12 |U_i(\tau,0,\eta)|^2.

\end{array}
\end{equation*}
This yields the classical Kreiss' estimate
\begin{equation}
\begin{array}{ll}\label{kreiss}
\ds \kappa_i\gamma\int_0^{+\infty}|U_i(\tau,x_1,\eta)|^2\, dx_1 + \dfrac12 |U_i(\tau,0,\eta)|^2 \le C_i\chi_i(\tau,\eta)^2|\fou h|^2.
\end{array}
\end{equation}
Now we use the definition of $U_i$ and a uniform bound for $\|T_i(\tau,\eta)^{-1}\|$ on the support of $\chi_i$ to derive
\begin{equation}
\begin{array}{ll}\label{stimafirst}
\ds \gamma\chi_i(\tau,\eta)^2\int_0^{+\infty} |\fou V^\nc(\tau,x_1,\eta)|^2dx_1 + \chi_i(\tau,\eta)^2 |\fou V^\nc(\tau,0,\eta)|^2 \leq C_i\chi_i(\tau,\eta)^2|\fou h|^2,
\end{array}
\end{equation}
for all $(\tau,\eta)\in\R^+\cdot \Vc_i$.

\subsubsection{The second case.}
$\Vc_i$ is a neighborhood of a boundary point which is a zero of the Lopatinski\u{\i} determinant but not a pole, see Subsection \ref{ssec:boundarynoLop}.

On such a neighborhood there exist two $C^\infty$ mappings~$r_i$ e $T_i$ such that $r_i$ is hermitian, $T_i$ has values in $GL_4(\C)$, and the following estimates hold for all $ (\tau,\eta)\in\Vc_i$:
\begin{subequations}\label{symmlop}
\begin{align}
\Re(r_i(\tau,\eta)T_i(\tau,\eta)\Ac(\tau,\eta) T_i(\tau,\eta)^{-1})\geq \kappa_i\gamma^3 \Id\,, \label{symmlopa}\\
r_i(\tau,\eta)+ C(\beta(\tau,\eta) T_i(\tau,\eta)^{-1})^* (\beta(\tau,\eta) T_i(\tau,\eta)^{-1})\geq \gamma^2\Id\, ,\label{symmlopb}
\end{align}
\end{subequations}
see \eqref{eq:symmLopa}, \eqref{eq:betaLopa}.
As done before, we first extend the definition of $r_i$ and $T_i$ to the whole hemisphere $\Sigma$. Then we extend~$\chi_i $ and $T_i$ to the whole set of frequencies $\Xi$, as homogeneous mappings of degree 0 with respect to $(\tau,\eta)$, and we extend~$r_i$ to the whole set of frequencies $\Xi$, as homogeneous mappings of degree 2 with respect to $(\tau,\eta)$. Thus \eqref{symmlop} reads
\begin{subequations}\label{symmlop2}
\begin{align}
\Re(r_i(\tau,\eta)T_i(\tau,\eta)\Ac(\tau,\eta) T_i(\tau,\eta)^{-1})\geq \kappa_i\gamma^3 \Id\,, \label{symmlop2a}\\
r_i(\tau,\eta)+ C(|\tau|^2+ \eta^2)(\beta(\tau,\eta) T_i(\tau,\eta)^{-1})^* (\beta(\tau,\eta) T_i(\tau,\eta)^{-1})\geq \gamma^2\Id\, ,\label{symmlop2b}
\end{align}
\end{subequations}
for all $ (\tau,\eta)\in\R^+\cdot \Vc_i$. 
We define
\[ U_i(\tau,x_1,\eta)\doteq \chi_i(\tau,\eta)\,T_i(\tau,\eta)\,\fou V^\nc(\tau,x_1,\eta). \]
Because $\Vc_i$ does not contain any pole we still have
\[\begin{cases}
{\dfrac{dU_i}{dx_1}}  = T_i(\tau,\eta)\Ac(\tau,\eta) T_i(\tau,\eta)^{-1} U_i &\quad x_1>0,\\
\beta(\tau,\eta) T_i(\tau,\eta)^{-1} U_i = \chi_i\fou h &\quad x_1=0,
\end{cases}\]
Performing the same calculations as above, with only \eqref{symmlop2} instead of \eqref{symmint}, yields
\begin{equation}
\begin{array}{ll}\label{stimasecond}
\ds \gamma\chi_i(\tau,\eta)^2\int_0^{+\infty} |\fou V^\nc(\tau,x_1,\eta)|^2dx_1 + \chi_i(\tau,\eta)^2 |\fou V^\nc(\tau,0,\eta)|^2 \leq \dfrac{C_i}{\gamma^2}(|\tau|^2+ \eta^2)\chi_i(\tau,\eta)^2|\fou h|^2,
\end{array}
\end{equation}
for all $(\tau,\eta)\in\R^+\cdot \Vc_i$.

%\newpage

\subsubsection{The third case.}\label{third}
$\Vc_i$ is a neighborhood of the boundary point $P_0=(\tau_0, \eta_0)$ with $\mu_0=\tau_0+i \base v\eta_0=0$. This is a pole for the plasma part of $\Ac$, specifically for $a_{12}$, and it is a a zero for the Lopatinski\u{\i} determinant. 

On such a neighborhood there exists a $C^\infty$ mapping~$T_i$ defined on $\Vc_i$ with values in $GL_4(\C)$, see \eqref{defTmu}, and such that for the matrix defined in \eqref{64}, the equation \eqref{65} is satisfied.
As done above, we first extend the definition of $T_i$ to the whole hemisphere $\Sigma$. Then we extend~$\chi_i $ and $T_i$ to the whole set of frequencies $\Xi$, as homogeneous mappings of degree 0 with respect to $(\tau,\eta)$.

We define
\[ U_i(\tau,x_1,\eta)\doteq \chi_i(\tau,\eta)\,T_i(\tau,\eta) \Lambda\,\fou V^\nc(\tau,x_1,\eta), 
\]
where the $4\times4$ constant matrix $\Lambda$ is defined in \eqref{defLambda}.
We deduce from \eqref{eq.trasfbdry}, \eqref{64}, \eqref{65} the following problem for $U_i$
\begin{equation}
\begin{cases}\label{66}
{\dfrac{dU_i}{dx_1}} =\Ac' (\tau,\eta) U_i \qquad& x_1>0,\\
 \beta(\tau,\eta) \Lambda^{-1}T(\tau,\eta)^{-1}U_i =\chi_i\fou{h}\qquad& x_1=0.
\end{cases}
\end{equation}
More specifically, the ODE system reads
\begin{subequations}
\label{systempolemu}
\begin{align}
&{\dfrac{dU_{i,1}}{dx_1}}  =-\omega_1U_{i,1}-2m_2U_{i,2},\label{67a}\\
&{\dfrac{dU_{i,2}}{dx_1}} =\omega_1U_{i,2}\label{67b},\\
&{\dfrac{dU_{i,3}}{dx_1}} =-\omega_2U_{i,3}\label{67c},\\
&{\dfrac{dU_{i,4}}{dx_1}} =\omega_2U_{i,4}.\label{67d}
\end{align}
\end{subequations}
Recall that $m_2$ is defined in \eqref{defm1m2} and that it has a pole at $P_0$. $\Vc_i$ is sufficiently small so that we may assume that 
\[
\Re\omega_1(\tau,\eta)\ge \kappa_i(|\tau|^2+ \eta^2)^{1/2}, \qquad \Re\omega_2(\tau,\eta)\ge \kappa_i(|\tau|^2+ \eta^2)^{1/2}, \qquad \forall  (\tau,\eta)\in \R^+\cdot\Vc_i,
\]
for a suitable constant $\kappa_i>0$. The above inequality for $\omega_1$ is obvious because $\omega_1(P_0)=|\eta_0|\not=0$. The inequality concerning $\omega_2$ is true provided $\epsilon$ is taken sufficiently small, as it follows from $\omega_2^2(\tau_0,\eta_0)=\eta_0^2(1-\epsilon^2\base v^2)$.

Since  $U_{i,2}(x_1)$ and $U_{i,4}(x_1)$ belong to $L^2(\R^+)$, from \eqref{67b} and \eqref{67d} we get $U_{i,2}\equiv0$ and $U_{i,4}\equiv0$. Hence, even if $m_2$ has a pole in $P_0$, the first equation \eqref{67a} is well defined and actually reads
\begin{equation}
\begin{array}{ll}\label{78}
{\dfrac{dU_{i,1}}{dx_1}}  =-\omega_1U_{i,1}.
\end{array}
\end{equation}
From \eqref{67c}, \eqref{78}  and the above properties of $\omega_1$ and $\omega_2$ we derive
\begin{equation}
\begin{array}{ll}\label{79}
\ds
(|\tau|^2+ \eta^2)^{1/2}\int_0^{+\infty}|U_{i,1}(\tau,x_1,\eta)|^2\, dx_1 \le C|U_{i,1}(\tau,0,\eta)|^2,
\\
\ds
(|\tau|^2+ \eta^2)^{1/2}\int_0^{+\infty}|U_{i,3}(\tau,x_1,\eta)|^2\, dx_1 \le C|U_{i,3}(\tau,0,\eta)|^2,

\end{array}
\end{equation}
for all $ (\tau,\eta)\in \R^+\cdot\Vc_i$.
On the other hand, the boundary condition in \eqref{66} reduces to
\begin{equation}
\begin{array}{ll}\label{80}
\begin{pmatrix}
\dfrac12\mu(m_1-\omega_1- m_2) & -\base\Hc\omega_2 \\

\dfrac12\epsilon\base\Hc\tau \mu(m_1-\omega_1+ m_2)  & \eps \tau\mu
\end{pmatrix}
\begin{pmatrix}
 U_{i,1} \\
 U_{i,3} 
\end{pmatrix}=\chi_i\fou h.
\end{array}
\end{equation}
Let us denote by $\Delta'(\tau,\eta)$ the determinant of the matrix in \eqref{80}. Substituting \eqref{defm1m2} in the calculation of this determinant gives
\[
\Delta'(\tau,\eta)\doteq \dfrac12\eps \tau\mu  \left( \mu a_{12}-\mu\omega_1 +\base\Hc^2\omega_2(a_{21}-\omega_1 ) \right) .
\]
If $\Vc_i$ is taken sufficiently small, it is easily verified that there exists a positive constant $C$ such that
\[
|\Delta'(\tau,\eta)|\ge C\gamma , \qquad \forall  (\tau,\eta)\in \Vc_i.
\]
Even if from its definition $\Delta'(\tau,\eta)$ appears as a homogeneous function of degree 4 with respect to $(\tau,\eta)$, actually it is a homogeneous function of degree 0, as follows from the extension of $T_i$ that we did to the whole set of frequencies $\Xi$. Therefore we have
\[
|\Delta'(\tau,\eta)|\ge C\gamma (|\tau|^2+ \eta^2)^{-1/2}, \qquad \forall  (\tau,\eta)\in \R^+\cdot\Vc_i,
\]
and from \eqref{80} it follows
\begin{equation}
\begin{array}{ll}\label{81}
|U_{i,1}(\tau,0,\eta)|+|U_{i,3}(\tau,0,\eta)| \le C \dfrac{(|\tau|^2+ \eta^2)^{1/2}}{\gamma}|\chi_i\fou h| .
\end{array}
\end{equation}
Now, combining \eqref{79}, \eqref{81} and using $(|\tau|^2+ \eta^2)^{1/2}\ge\gamma$ gives
\begin{multline*}
\ds
\gamma \int_0^{+\infty} \left( |U_{i,1}(\tau,x_1,\eta)|^2 +  |U_{i,3}(\tau,x_1,\eta)|^2 \right) dx_1 \\+ |U_{i,1}(\tau,0,\eta)|^2 + |U_{i,3}(\tau,0,\eta)|^2 
\le \dfrac{C_i }{\gamma^2}(|\tau|^2+ \eta^2)|\chi_i(\tau,\eta) \fou h|^2 
 ,
\end{multline*}
for all $ (\tau,\eta)\in \R^+\cdot\Vc_i$.
Finally, also recalling that $U_{i,2}\equiv0$ and $U_{i,4}\equiv0$, we obtain
\begin{equation}
\begin{array}{ll}\label{stimathird}
\ds \gamma\chi_i(\tau,\eta)^2\int_0^{+\infty} |\fou V^\nc(\tau,x_1,\eta)|^2dx_1 + \chi_i(\tau,\eta)^2 |\fou V^\nc(\tau,0,\eta)|^2 \leq \dfrac{C_i}{\gamma^2}(|\tau|^2+ \eta^2)\chi_i(\tau,\eta)^2|\fou h|^2, 
\end{array}
\end{equation}
for all $(\tau,\eta)\in\R^+\cdot \Vc_i$.

%\newpage

\subsubsection{The fourth case.}\label{fourth}
$\Vc_i$ is a neighborhood of the boundary point $P_0=(\tau_0, \eta_0)$ with $\gamma_0=0$ and $\delta_0+\base v \eta_0 = + \frac{|\base H|}{\sqrt{\base\rho(1+\base\alpha\base H^2)}}\eta_0$, with $\eta_0> 0$. (The case $\gamma_0=0$, $\delta_0+\base v \eta_0 = - \frac{|\base H|}{\sqrt{\base\rho(1+\base\alpha\base H^2)}}\eta_0$ can be studied similarly.) This is a pole for the plasma part of $\Ac(P_0)$, specifically for $a_{21}$, and the eigenvalues $\pm\omega_1$ have a pole as well. Moreover, at $P_0$ the Lopatinski\u{\i} condition is satisfied.

On such a neighborhood there exist two $C^\infty$ mappings~$r_i$ and $T_i$ (defined in \eqref{defT-1}) such that $r_i$ is hermitian, $T_i$ has values in $GL_4(\C)$, and such that the inequalities \eqref{estpolo22}, \eqref{estpolo21} are satisfied.

As done above, we first extend the definition of $r_i$ and $T_i$ to the whole hemisphere $\Sigma$. Then we extend~$\chi_i ,r_i$ and $T_i$ to the whole set of frequencies $\Xi$, as homogeneous mappings of degree 0 with respect to $(\tau,\eta)$. Consequently, $r_i$ and $T_i$ satisfy
\begin{equation*}
\begin{array}{ll}\label{estpolo212}
\Re \left( r_i(\tau,\eta) T_i(\tau,\eta)
\Lambda\,\Ac(\tau,\eta)\,(T_i(\tau,\eta)\Lambda)^{-1} \right)
\ge \kappa_i \begin{pmatrix}
\gamma\dfrac{|\tau|^2+ \eta^2}{|\ttau|^2} &0 &0&0\\
0 & \gamma &0&0\\
0& 0&\gamma & 0 \\
0&0& 0 &\gamma
\end{pmatrix}
\,, 

\end{array}
\end{equation*}
\begin{equation*}
\begin{array}{ll}\label{estpolo222}
 r_i(\tau,\eta) +C \left( \beta(\tau,\eta)\Lambda^{-1}T_i^{-1}(\tau,\eta) \right)^\ast
\beta(\tau,\eta)\Lambda^{-1}T_i^{-1}(\tau,\eta)
\ge  \Id\,, 

\end{array}
\end{equation*}
for all $(\tau,\eta)\in\R^+\cdot \Vc_i$. We define
\[ U_i(\tau,x_1,\eta)\doteq \chi_i(\tau,\eta)\,T_i(\tau,\eta) \Lambda\,\fou V^\nc(\tau,x_1,\eta). 
\]
Using the same argument as in {\it the first case}, see \ref{first}, we get the estimate
\begin{equation}
\begin{array}{ll}\label{stimafourth}
\ds \gamma\chi_i(\tau,\eta)^2\int_0^{+\infty} |\fou V^\nc(\tau,x_1,\eta)|^2dx_1 + \chi_i(\tau,\eta)^2 |\fou V^\nc(\tau,0,\eta)|^2 \leq C_i\chi_i(\tau,\eta)^2|\fou h|^2,
\end{array}
\end{equation}
for all $(\tau,\eta)\in\R^+\cdot \Vc_i$.

%\newpage

\subsubsection{The fifth case.}\label{fifth}
$\Vc_i$ is a neighborhood of the boundary point $P_0=(\tau_0, \eta_0)$ with $\tau_0=0$. This is a pole for the vacuum part of $\Ac$, specifically for $a_{34}$, and it is a zero for the Lopatinski\u{\i} determinant.

On such a neighborhood there exists a $C^\infty$ mapping~$T_i$ defined on $\Vc_i$ with values in $GL_4(\C)$, and such that for the matrix defined in \eqref{64'}, the equation \eqref{65''} is satisfied.
As done above, we first extend the definition of $T_i$ to the whole hemisphere $\Sigma$. Then we extend~$\chi_i $ and $T_i$ to the whole set of frequencies $\Xi$, as homogeneous mappings of degree 0 with respect to $(\tau,\eta)$.

We define
\[ U_i(\tau,x_1,\eta)\doteq \chi_i(\tau,\eta)\,T_i(\tau,\eta) \Lambda'\,\fou V^\nc(\tau,x_1,\eta), 
\]
where the $4\times4$ constant matrix $\Lambda'$ is defined in \eqref{defLambda'}.
We deduce from \eqref{eq.trasfbdry}, \eqref{64'}, \eqref{65''} the following problem for $U_i$
\begin{equation}
\begin{cases}\label{66'}
{\dfrac{dU_i}{dx_1}} =\Ac''' (\tau,\eta) U_i \qquad& x_1>0,\\
 \beta(\tau,\eta) \Lambda'^{-1}T(\tau,\eta)^{-1}U_i =\chi_i\fou{h}\qquad& x_1=0.
\end{cases}
\end{equation}
More specifically, the ODE system reads
\begin{subequations}
\label{systempoletau'}
\begin{align}
&{\dfrac{dU_{i,1}}{dx_1}}  =-\omega_1U_{i,1} ,\label{67a'}\\
&{\dfrac{dU_{i,2}}{dx_1}} =\omega_1U_{i,2}\label{67b'},\\
&{\dfrac{dU_{i,3}}{dx_1}} =-\omega_2U_{i,3}-2n_2U_{i,4}\label{67c'},\\
&{\dfrac{dU_{i,4}}{dx_1}} =\omega_2U_{i,4}.\label{67d'}
\end{align}
\end{subequations}
Recall that $n_2$ is defined in \eqref{defn1n2} and that it has a pole at $P_0$. $\Vc_i$ is sufficiently small so that we may assume that 
\[
\Re\omega_1(\tau,\eta)\ge \kappa_i\gamma, \qquad \Re\omega_2(\tau,\eta)\ge \kappa_i(|\tau|^2+ \eta^2)^{1/2}, \qquad \forall  (\tau,\eta)\in \R^+\cdot\Vc_i,
\]
for a suitable constant $\kappa_i>0$. The inequality concerning $\omega_1$ follows from Remark \ref{remark2}.
The above inequality for $\omega_2$ is obvious because $\omega_2(P_0)=|\eta_0|\not=0$. 

As in Subsection \ref{third} we show that $U_{i,2}\equiv0$ and $U_{i,4}\equiv0$, and \eqref{67c'} reduces to
\begin{equation}
\begin{array}{ll}\label{78'}
{\dfrac{dU_{i,3}}{dx_1}}  =-\omega_2U_{i,3}.
\end{array}
\end{equation}
From \eqref{67a'}, \eqref{78'}  and the above properties of $\omega_1$ and $\omega_2$ we derive
\begin{equation}
\begin{array}{ll}\label{79'}
\ds
\gamma \int_0^{+\infty}|U_{i,1}(\tau,x_1,\eta)|^2\, dx_1 \le C|U_{i,1}(\tau,0,\eta)|^2,
\\
\ds
(|\tau|^2+ \eta^2)^{1/2}\int_0^{+\infty}|U_{i,3}(\tau,x_1,\eta)|^2\, dx_1 \le C|U_{i,3}(\tau,0,\eta)|^2,

\end{array}
\end{equation}
for all $ (\tau,\eta)\in \R^+\cdot\Vc_i$.
On the other hand, the boundary condition in \eqref{66'} reduces to
\begin{equation}
\begin{array}{ll}\label{80'}
\begin{pmatrix}
\mu\chi a_{12} & -\dfrac12\base\Hc\tau (n_1-\omega_2- n_2)
\\
-\epsilon\base\Hc\tau \mu\chi \omega_1 & \dfrac12\mu\tau (n_1-\omega_2+ n_2)
\end{pmatrix}
\begin{pmatrix}
 U_{i,1} \\
 U_{i,3} 
\end{pmatrix}=\chi_i\fou h.
\end{array}
\end{equation}
Let us denote by $\Delta''(\tau,\eta)$ the determinant of the matrix in \eqref{80}. Substituting \eqref{defn1n2} in the calculation of this determinant gives
\[
\Delta''(\tau,\eta)\doteq \dfrac12\chi \tau\mu  \left( \mu a_{12}(a_{43}-\omega_2) -\eps\base\Hc^2\omega_1(\tau a_{34}-\tau\omega_2 ) \right) .
\]
If $\Vc_i$ is taken sufficiently small, it is easily verified that there exists a positive constant $C$ such that
\[
|\Delta''(\tau,\eta)|\ge C\gamma , \qquad \forall  (\tau,\eta)\in \Vc_i.
\]
Even if from its definition $\Delta''(\tau,\eta)$ appears as a homogeneous function of degree 5 with respect to $(\tau,\eta)$, actually it is a homogeneous function of degree 0, as follows from the extension of $T_i$ that we did to the whole set of frequencies $\Xi$. Therefore we have
\[
|\Delta''(\tau,\eta)|\ge C\gamma (|\tau|^2+ \eta^2)^{-1/2}, \qquad \forall  (\tau,\eta)\in \R^+\cdot\Vc_i,
\]
and from \eqref{80'} it follows
\begin{equation}
\begin{array}{ll}\label{81'}
|U_{i,1}(\tau,0,\eta)|+|U_{i,3}(\tau,0,\eta)| \le C \dfrac{(|\tau|^2+ \eta^2)^{1/2}}{\gamma}|\chi_i\fou h| .
\end{array}
\end{equation}
Now, combining \eqref{79'}, \eqref{81'} and using $(|\tau|^2+ \eta^2)^{1/2}\ge\gamma$ gives
\begin{equation*}
\begin{array}{ll}\label{}
\ds
\gamma \int_0^{+\infty} \left( |U_{i,1}(\tau,x_1,\eta)|^2 +  |U_{i,3}(\tau,x_1,\eta)|^2 \right) dx_1 + |U_{i,1}(\tau,0,\eta)|^2 + |U_{i,3}(\tau,0,\eta)|^2 \le \dfrac{C_i }{\gamma^2}(|\tau|^2+ \eta^2)|\chi_i(\tau,\eta) \fou h|^2 
 ,
\end{array}
\end{equation*}
for all $ (\tau,\eta)\in \R^+\cdot\Vc_i$.
Finally, also recalling that $U_{i,2}\equiv0$ and $U_{i,4}\equiv0$, we obtain
\begin{equation}
\begin{array}{ll}\label{stimafifth}
\ds \gamma\chi_i(\tau,\eta)^2\int_0^{+\infty} |\fou V^\nc(\tau,x_1,\eta)|^2dx_1 + \chi_i(\tau,\eta)^2 |\fou V^\nc(\tau,0,\eta)|^2 \leq \dfrac{C_i}{\gamma^2}(|\tau|^2+ \eta^2)\chi_i(\tau,\eta)^2|\fou h|^2,
\end{array}
\end{equation}
for all $(\tau,\eta)\in\R^+\cdot \Vc_i$.

\subsubsection{Proof of estimate \eqref{stimaV2bdry}}
Adding \eqref{stimafirst}, \eqref{stimasecond}, \eqref{stimathird}, \eqref{stimafourth}, \eqref{stimafifth}, and using the partition of unity gives
\begin{equation*}
\begin{array}{ll}\label{}
\ds \gamma\int_0^{+\infty} |\fou V^\nc(\tau,x_1,\eta)|^2dx_1 +  |\fou V^\nc(\tau,0,\eta)|^2 \leq \dfrac{C_i}{\gamma^2}(|\tau|^2+ \eta^2)|\fou h|^2,
\end{array}
\end{equation*}
for all $(\tau,\eta)\in\Xi$. We integrate with respect to $(\delta,\eta)\in\R^2$ and obtain the estimate
\begin{equation*}
\begin{array}{ll}\label{}
\ds \gamma \|V^\nc\|^2_{L^2(\Omega)} + \|V^\nc_{x_1=0}\|_{L^2(\R^2)}^2\leq  \frac{C}{\gamma^2} \|g\|_{1,\gamma}^2
 \,,
\end{array}
\end{equation*}
which yields \eqref{stimaV2bdry}.

%\newpage

%%%%%%%%%%%%%%%%%%%%%%%%%%%%%%%%%%%%
%%%%%%%%%%%%%%%%%%%%%%%%%%%%%%%%%%%%
%%%%%%%%%%%%%%%%%%%%%%%%%%%%%%%%%%%%
%%%%%%%%%%%%%%%%%%%%%%%%%%%%%%%%%%%%%
\appendix
\section{Proof of Lemma \ref{lemma.diffeo}}
\label{AppA}

The proof is similar to that one of Lemma 3 in \cite{SeTr}. We first proof the following result.
\begin{Lem}
\label{lemma1}

Let $m\ge 3$. For all $\eps>0$ there exists a continuous linear map $\varphi\in H^{m-0.5}(\R) \mapsto
\Psi \in H^m(\Omega^+)$ such that $\Psi(0,x_2)=\varphi(x_2)$, $\duno \Psi (0,x_2)=0$ on $\Gamma$, and
\begin{equation}
\begin{array}{ll}\label{diffeopiccolo}
\|\duno\Psi\|_{L^\infty(\Omega^+)}\le\eps \, \| \varphi\|_{H^{2}(\R)}.
\end{array}
\end{equation}

\end{Lem}
%%%
\begin{proof}
The first part of the proof is similar to the proof of Lemma 1 in \cite{SeTr} and we repeat it here for reader's convenience.
Given an even function $\chi\in C^\infty_0(\R)$, with $\chi=1$ on $[-1,1]$, we define
\begin{equation}
\label{def-f1}
\Psi(x_1,x_2) :=\chi (x_1\langle D\rangle) \, \varphi(x_2) \, ,
\end{equation}
where $\chi(x_1 \langle D\rangle)$ is the pseudo-differential operator with $\langle D\rangle=(1+|D|^2)^{1/2}$ being the Fourier multiplier in the
variables $x_2$. From the definition it readily follows that $\Psi(0,x_2)=\varphi(x_2)$ for
all $x_2\in\R$. Moreover,
\begin{equation}
\label{d1Psi}
\duno \Psi(x_1,x_2) = \chi'(x_1\langle D\rangle) \, \langle D\rangle \, \varphi(x_2) \, ,
\end{equation}
which vanishes if $x_1=0$. We compute
\begin{equation*}
\|\Psi(x_1,\cdot)\|^2_{H^m(\R)} =\int_{\R}\langle \xi\rangle^{2m}\chi^2(x_1\langle \xi\rangle)|\hat\varphi(\xi)|^2d\xi \, ,
\end{equation*}
where $\hat\varphi(\xi)$ denotes the Fourier transform in $x_2$ of $\varphi$.
It follows that
\begin{align*}
\| \Psi \|^2_{L^2(\R^+_{x_1};H^m(\R))} =\int_{\R}\int_{\R} \langle \xi\rangle^{2m}\chi^2(x_1\langle \xi\rangle)|\hat\varphi(\xi')|^2d\xi \,dx_1
\\
=\int_{\R}\int_{\R} \langle \xi\rangle^{2m-1}\chi^2(s)|\hat\varphi(\xi)|^2d\xi \,ds
\le C\|\varphi\|^2_{H^{m-0.5}(\R)}\, .
\end{align*}
In a similar way, from \eqref{d1Psi}, we obtain
\begin{align*}
\| \duno \Psi \|^2_{L^2(\R^+_{x_1};H^{m-1}(\R))} =\int_{\R}\int_{\R} \langle \xi\rangle^{2m-2}|\chi'(x_1\langle \xi\rangle)\langle \xi\rangle|^2|\hat\varphi(\xi)|^2d\xi \,dx_1
\\
=\int_{\R}\int_{\R} \langle \xi\rangle^{2m-1}|\chi'(s)|^2|\hat\varphi(\xi')|^2d\xi \,ds
\le C\|\varphi\|^2_{H^{m-0.5}(\R)}\, .
\end{align*}
Iterating the same argument yields
\begin{equation*}
\| \duno^j \Psi \|^2_{L^2(\R^+_{x_1};H^{m-j}(\R))} \le C\, \| \varphi \|_{H^{m-0.5}(\R)}^2 \, ,\quad j=0,\dots,m \, .
\end{equation*}
Adding over $j=0,\dots, m$ finally gives $\Psi \in H^m(\Omega^+)$ and the continuity of the map $\varphi \mapsto \Psi$.

%%%%
%%%%

We now show that the cut-off function $\chi$, and accordingly the map $\varphi \mapsto \Psi$, can be chosen to give \eqref{diffeopiccolo}. From \eqref{d1Psi} we have
\begin{equation*}
\label{}
\duno \Psi(x_1,x_2) = (2\pi)^{-1}\int _{\R} e^{i\xi\cdot x_2}\chi'(x_1\langle \xi\rangle) \, \langle \xi\rangle \, \hat\varphi(\xi) \, d\xi .
\end{equation*}
By the Cauchy-Schwarz inequality and writing $\rho=| \xi | $ we get
\begin{align*}
\label{}
|\duno \Psi(x)| \le C \| \varphi\|_{H^{2}(\R)} \left( \int _{\R} |\chi'(x_1\langle \xi\rangle)|^2 \, \langle \xi\rangle^{-2} \, \, d\xi \right)^{1/2}
\\
\qquad\qquad =C \| \varphi\|_{H^{2}(\R)} \left( \int _0^{\infty} |\chi'(x_1\langle \rho\rangle)|^2 \, \langle \rho\rangle^{-2} \,  d\rho \right)^{1/2} .
\end{align*}
We change variables in the integral above by setting $s=x_1 \langle \rho\rangle$. It follows that
\begin{align}
\label{d1Psi2}
|\duno \Psi(x)| \le C \| \varphi\|_{H^{2}(\R)} \left( \int _{x_1}^{\infty} |\chi'(s)|^2 \, \frac{x_1}{s \sqrt{s^2-x_1^2}}\, {ds} \right)^{1/2}
%\le C \| \varphi\|_{H^{2}(\R)} \left( \int _{1}^{\infty} |\chi'(s)|^2 \, \frac{ds}{s} \right)^{1/2}
.
\end{align}
Given any $M>1$, we choose $\chi$ such that $\chi(s)=0$ for $|s|\ge M$, and $|\chi'(s)|\le 2/M$ for every $s$.
Analyzing the above integral for all possible values of $x_1>0$ with respect to 1 and $M$ gives
\[
\int _{x_1}^{\infty} |\chi'(s)|^2 \, \frac{x_1}{s \sqrt{s^2-x_1^2}}\, {ds}
\le C M^{-3/2} \qquad \forall x_1>0\, .
\]
Then from \eqref{d1Psi2} one gets
\begin{align*}
\label{d1Psi3}
|\duno \Psi(x)| \le {C}{M}^{-3/4} \| \varphi\|_{H^{2}(\R)} .
\end{align*}
Given any $\eps>0$, if $M$ is such that ${C}{M}^{-3/4}<\eps$, then \eqref{diffeopiccolo} immediately follows.
\end{proof}

The following lemma gives the time-dependent version of Lemma \ref{lemma1}.

\begin{Lem}
\label{lemma2}
Let $m \ge3$ be an integer and let $T>0$. For all $\eps>0$ there exists a continuous linear map $\varphi\in \cap_{j=0}^{m-1}
{\mathcal C}^j([0,T];H^{m-j-0.5}(\R)) \mapsto \Psi \in \cap_{j=0}^{m-1} {\mathcal C}^j([0,T];H^{m-j}(\Omega^+))$
such that $\Psi(t,0,x_2)=\varphi(t,x_2)$,  $\duno \Psi (t,0,x_2)=0$ on $\Gamma$, and
\begin{equation}
\begin{array}{ll}\label{diffeopiccolo2}
\|\duno\Psi\|_{{\mathcal C}([0,T];L^\infty(\Omega^+))}\le\eps \, \| \varphi\|_{{\mathcal C}([0,T];H^{2}(\R))}.
\end{array}
\end{equation}
Furthermore, there exists a constant $C>0$ that is independent of $T$ and only depends on $m$, such that
\begin{multline*}
\forall \, \varphi\in \cap_{j=0}^{m-1} {\mathcal C}^j([0,T];H^{m-j-0.5}(\R)) \, ,\quad
\forall \, j=0,\dots,m-1 \, ,\quad \forall \, t \in [0,T] \, ,\\
\| \partial_t^j \Psi (t,\cdot) \|_{H^{m-j}(\Omega^+)} \le C \,
\| \partial_t^j \varphi (t,\cdot) \|_{H^{m-j-0.5}(\R)} \, .
\end{multline*}
\end{Lem}
%%%
\begin{proof}

The proof of Lemma \ref{lemma2} follows from Lemma \ref{lemma1}, with $t$ as a parameter. Notice also that
the map $\varphi\to \Psi$, defined by  \eqref{def-f1}, is linear and that the time regularity is conserved because, with
obvious notation, $\Psi(\dt^j \varphi) =\dt^j \Psi(\varphi)$. The conclusions of Lemma \ref{lemma2} follow directly.
\end{proof}

\begin{proof}[Proof of Lemma \ref{lemma.diffeo}]
The proof follows directly from Lemma \ref{lemma2} because
$$
\duno \Phi_1(t,x) =1 +\duno \Psi(t,x)\ge 1-\| \duno\Psi(t,\cdot) \|_{{\mathcal C}([0,T];L^{\infty}(\Omega^+))}
\ge 1 -\eps \, \|\varphi\|_{{\mathcal C}([0,T];H^{2}(\R))} \ge 1/2 \, ,
$$
provided $\eps$ is taken sufficiently small, e.g. $\eps<1/2$. The other properties of $\Phi$ follow directly from Lemma \ref{lemma2}.
\end{proof}

%%%%%%%%%%%%%%%%%%%%%%%%%%%%%%%%%%%%
%%%%%%%%%%%%%%%%%%%%%%%%%%%%%%%%%%%%
%%%%%%%%%%%%%%%%%%%%%%%%%%%%%%%%%%%%
%%%%%%%%%%%%%%%%%%%%%%%%%%%%%%%%%%%%%


\begin{thebibliography}{99}


\bibitem{Cat}
\newblock D. Catania,
\newblock \doititle{Existence and Stability for the 3D Linearized Constant-Coefficient Incompressible Current-Vortex Sheets},
\newblock \emph{Int. J. Differ. Equ.}, 2013, 1--13.


\bibitem{CDAS}
\newblock D. Catania and M. D'Abbicco and P. Secchi,
\newblock \doititle{Stability of the linearized MHD-Maxwell free interface problem},
\newblock \emph{Commun. Pure Appl. Anal.}, \textbf{13} (6) (2014), 2407--2443.


\bibitem{chazarain-piriou}
J.~Chazarain, A.~Piriou.
\newblock {\em Introduction to the theory of linear partial differential
  equations}.
\newblock North-Holland Publishing Co., Amsterdam, 1982.



\bibitem{chenwang}
G.-Q. Chen, Y.-G.~Wang.
\newblock Existence and stability of compressible current-vortex sheets in
  three-dimensional magnetohydrodynamics.
\newblock {\em Arch. Ration. Mech. Anal.}, 187(3):369--408, 2008.


\bibitem{coulombel1}
J.-F. Coulombel.
\newblock Weak stability of nonuniformly stable multidimensional shocks.
\newblock {\em SIAM J. Math. Anal.}, 34(1):142--172, 2002.



\bibitem{cmst}
J.~F. Coulombel, A.~Morando, P.~Secchi, and P.~Trebeschi.
\newblock A priori estimate for 3-{D} incompressible current-vortex sheets.
\newblock {\em Comm. Math. Phys.}, 311(1):247--275, 2012.



\bibitem{transition} 
 J. F. Coulombel and P. Secchi,
\newblock \doititle{On the transition to instability for compressible vortex
              sheets},
\newblock \emph{Proc. Roy. Soc. Edinburgh Sect. A}, \textbf{134} (2004), 885--892.


\bibitem{nappes} 
\newblock J. F. Coulombel and P. Secchi,
\newblock \doititle{The stability of compressible vortex sheets in two space dimensions},
\newblock \emph{Indiana Univ. Math. J.}, \textbf{53} (2004), 941--1012.


\bibitem{cs}
J.~F. Coulombel, P.~Secchi.
\newblock Nonlinear compressible vortex sheets in two space dimensions.
\newblock {\em Ann. Sci. \'Ecole Norm. Sup. (4)}, 41(1):85--139, 2008.



\bibitem{hersh}
R.~Hersh.
\newblock Mixed problems in several variables.
\newblock {\em J. Math. Mech.}, 12:317--334, 1963.


\bibitem{Goed}
\newblock J. P. Goedbloed and S. Poedts,
\newblock \emph{Principles of Magnetohydrodynamics with Applications to Laboratory and Astrophysical Plasmas},
\newblock Cambridge University Press, Cambridge, 2004.


\bibitem{kreiss}
H.-O. Kreiss.
\newblock Initial boundary value problems for hyperbolic systems.
\newblock {\em Comm. Pure Appl. Math.}, 23:277--298, 1970.




\bibitem{lax-phillips}
P.~D. Lax, R. S.~Phillips.
\newblock Local boundary conditions for dissipative symmetric linear
  differential operators.
\newblock {\em Comm. Pure Appl. Math.}, 13:427--455, 1960.


\bibitem{mandrik-trakhinin} 
\newblock N. Mandrik and Y. Trakhinin,
\newblock Influence of vacuum electric field on the stability of a plasma-vacuum interface,
\newblock \emph{Commun. Math. Sci.}, \textbf{12} (2014), 1065--1100.



\bibitem{majda-osher}
A.~Majda, S.~Osher.
\newblock Initial-boundary value problems for hyperbolic equations with
  uniformly characteristic boundary.
\newblock {\em Comm. Pure Appl. Math.}, 28(5):607--675, 1975.



\bibitem{ralston}
J.~V. Ralston.
\newblock Note on a paper of {K}reiss.
\newblock {\em Comm. Pure Appl. Math.}, 24(6):759--762, 1971.


\bibitem{SeTr} 
\newblock P. Secchi and Y. Trakhinin,
\newblock \doititle{Well-posedness of the linearized plasma-vacuum interface problem},
\newblock \emph{Interfaces Free Bound.}, \textbf{15} (2013), 323--357.



\bibitem{SeTrNl}
P.~Secchi, Y.~Trakhinin.
\newblock Well-posedness of the plasma-vacuum interface problem.
\newblock {\em Nonlinearity}, \textbf{27} (2014), 105--169.




\bibitem{SWZ}
Y. Sun, W. Wang, Z. Zhang.
\newblock Nonlinear stability of current-vortex sheet to the incompressible MHD equations,  arXiv:1510.02228.





\bibitem{trakhinin05}
Y.~Trakhinin.
\newblock Existence of compressible current-vortex sheets: Variable coefficients
    linear analysis.
\newblock {\em Arch. Ration. Mech. Anal.}, 177(3) (2005), 331--366.


\bibitem{trakhinin09arma}
Y.~Trakhinin.
\newblock The existence of current-vortex sheets in ideal compressible
  magnetohydrodynamics.
\newblock {\em Arch. Ration. Mech. Anal.}, 191(2):245--310, 2009.



\bibitem{trakhinin12}
Y.~Trakhinin.
\newblock Stability of relativistic plasma-vacuum interface.
\newblock {\em J. Hyperbolic Differential Equations}, \textbf{9} (2012), 469--509.



\bibitem{trakhinin15}
Y.~Trakhinin.
\newblock On well-posedness of the plasma-vacuum interface problem: the case of non-elliptic interface symbol.
\newblock {\em Commun. Pure Appl. Anal.}, \textbf{15} (4) (2016), 1371--1399.



\bibitem{WYu}
\newblock Y.-G. Wang, F. Yu,
\newblock \doititle{Stabilization effect of magnetic fields on two-dimensional
              compressible current-vortex sheets},
\newblock \emph{ Arch. Ration. Mech. Anal.}, \textbf{208} (2013), 341--389.


\bibitem{WYuan}
\newblock Y.-G. Wang, H. Yuan,
\newblock \doititle{Weak stability of transonic contact discontinuities in three-dimensiona steady non-isentropic compressible Euler flows},
\newblock \emph{ Z. Angew. Math. Phys.}, \textbf{66} (2015), 341--388.

\end{thebibliography}
\end{document}